\documentclass[11pt]{article}

\usepackage[ansinew, utf8]{inputenc}
\usepackage[rightcaption]{sidecap}
\usepackage{mathtools, amsthm, amssymb, bbm,amsmath}
\usepackage{multicol,latexsym, array, xcolor}
\usepackage{natbib}
\usepackage[a4paper]{geometry}
\usepackage{a4wide}

\usepackage[colorlinks,citecolor=blue,urlcolor=blue]{hyperref}
\usepackage{dsfont}
\usepackage{parskip}
\usepackage{bm}
\usepackage{ mathrsfs }
\usepackage{tikz}
\usepackage{cleveref}
\usepackage{todonotes}
\usepackage[font=footnotesize,labelfont=bf]{caption}
\usepackage{subcaption}
\usepackage{enumitem} 
\theoremstyle{plain}
\newtheorem{lemma}{Lemma}[section]
\newtheorem{thm}[lemma]{Theorem}
\newtheorem{cond}[lemma]{Condition}
\newtheorem{set}[lemma]{Setting}
\newtheorem{cor}[lemma]{Corollary}
\theoremstyle{definition}

\newtheorem{ex}[lemma]{Example}
\newtheorem{rem}[lemma]{Remark}

\newcommand{\diff}{\ensuremath{\,{d}}}

\newcommand{\Eargs}[1]{\mathbb{E}\left[ #1\right] }
\newcommand{\prob}{\mathbb{P}}

\newcommand{\Cov}{\mathrm{Cov}} 
\newcommand{\vast}{\bBigg@{4}}
\newcommand{\iid}{\overset{i.i.d.}{\sim}}
\newcommand{\R}{\mathbb{R}}
\newcommand{\Rd}{\R^d}
\newcommand{\N}{\mathbb{N}}
\newcommand{\muU}{\mu^U }
\newcommand{\muUn}{\mu^{U}_n }
\newcommand{\muVm}{\mu^{V}_m }

\newcommand{\nuUn}{\nu^{U}_n }
\newcommand{\X}{\mathcal{X}}
\newcommand{\muX}{\mu_\mathcal{X}}
\newcommand{\dX}{d_\mathcal{X}}
\newcommand{\mmspaceX}{\left(\X,\dX,\muX \right) }

\newcommand{\Y}{\mathcal{Y}}
\newcommand{\muY}{\mu_\mathcal{Y}}
\newcommand{\dY}{d_\mathcal{Y}}
\newcommand{\mmspaceY}{\left(\Y,\dY,\muY \right) }

\newcommand{\V}{\mathcal{V}}
\newcommand{\muV}{\mu_\mathcal{V}}
\newcommand{\dV}{d_\mathcal{V}}
\newcommand{\mmspaceV}{\left(\V,\dV,\muV \right) }
\newcommand{\Wtwo}[2]{\mathcal{K}^2_2\left(#1,#2 \right) }
\newcommand{\borelsetsR}{\ensuremath{\mathscr{B}(\R)}}
\newcommand{\indifunc}[1]{\mathds{1}_{\left\lbrace#1 \right\rbrace }}
\newcommand{\ptruedodstat}{DoD_p}
\newcommand{\truedodstat}{DoD_{(0)}}
\newcommand{\truetruncateddod}{DoD_{(\beta)}}
\newcommand{\dodstat}{\widehat{DoD}_{(\beta)}}

\newcommand{\dodstatnn}{\widehat{DoD}_{(\beta)}}
\newcommand{\untruncateddodstat}{\widehat{{DoD}}}
\newcommand{\qu}{\ensuremath{U_n^{-1}}}%
\newcommand{\G}{\ensuremath{\mathbb{G}}}%
\newcommand{\UV}{\ensuremath{\mathbb{Q}_n}^{U,V} }%
\newcommand{\qv}{\ensuremath{V_n^{-1}}}

\newcommand{\XiUVn}{\ensuremath{{\Xi}^{U,V}_n}}
\newcommand{\diam}[1]{\mathrm{diam}\left(#1\right) }
\makeatletter
\newcommand{\specialcell}[1]{\ifmeasuring@#1\else\omit$\displaystyle#1$\ignorespaces\fi}
\makeatother
\newcommand{\HadPhi}{\nabla\phi_{inv_F}}%
\providecommand{\keywords}[1]{\textbf{\textit{Keywords}} #1}
\providecommand{\keywordsMSC}[1]{\textbf{\textit{MSC 2010 subject classification}} #1}

\newcommand{\footremember}[2]{%
	\footnote{#2}
	\newcounter{#1}
	\setcounter{#1}{\value{footnote}}%
}
\newcommand{\footrecall}[1]{%
	\footnotemark[\value{#1}]%
}

\makeatletter				
\def\namedlabel#1#2{\begingroup
	#2%
	\def\@currentlabel{#2}%
	\phantomsection\label{#1}\endgroup
}
\makeatother

\makeatletter
\def\thm@space@setup{%
	\thm@preskip=\parskip \thm@postskip=0pt
}
\makeatother

\begin{document}

		\author{Christoph Alexander Weitkamp\footremember{ims}{\scriptsize Institute for Mathematical
				Stochastics, University of G\"ottingen,				Goldschmidtstra{\ss}e 7, 37077 G\"ottingen} 
			\and 
			Katharina Proksch \footremember{UT}{\scriptsize Faculty of Electrical Engineering, Mathematics \& Computer Science, University of Twente,  Hallenweg 19,  7522NH  Enschede}
			\and
			Carla Tameling \footrecall{ims}{}
			\and 
			Axel Munk \footrecall{ims} \footnote{\scriptsize Max Planck Institute for Biophysical
				Chemistry, Am Fa{\ss}berg 11, 37077 G\"ottingen}}
		
\title{Gromov-Wasserstein Distance based Object Matching: Asymptotic Inference}

		\maketitle
\begin{abstract}
	In this paper, we aim to provide a statistical theory for object matching based on the Gromov-Wasserstein distance. To this end, we model general objects as metric measure spaces. Based on this, we propose a simple and efficiently computable asymptotic statistical test for pose invariant object discrimination. This is based on an empirical version of a $\beta$-trimmed lower bound of the Gromov-Wasserstein distance. We derive for $\beta\in[0,1/2)$ distributional limits of this test statistic. To this end, we introduce a novel $U$-type process indexed in $\beta$ and show its weak convergence. Finally, the theory developed is investigated in Monte Carlo simulations and applied to structural protein comparisons.
\end{abstract}

		
\keywordsMSC{Primary: 62E20, 62G20, 65C60 Secondary:  60E05 }
	
% !TEX root = Dod.tex
\section{Introduction}\label{sec:intro}
Over the last decades, the acquisition of geometrically complex data in various fields of application has increased drastically. For the digital organization and analysis of such data it is important to have meaningful notions of \textit{similarity} between datasets as well as between shapes. This most certainly holds true for the area of 3-D object matching, which has many relevant applications, for example in computer vision \citep{viola2001rapid, TorralbaContextbasedvisionsystem2003}, mechanical engineering \citep{AuFeaturebasedreverseengineering1999, El-Mehalawidatabasesystemmechanical2003} or molecular biology \cite{NussinovEfficientdetectionthreedimensional1991, Sandakautomatedcomputervision1995,KrissinelSecondarystructurematchingSSM2004}. In most of these applications, an important challenge is to distinguish between shapes while regarding identical objects in different spatial orientations as equal. A prominent example is the comparison of 3-D protein structures, which is important for understanding structural, functional and evolutionary relationships among proteins \cite{kolodny2005comprehensive,srivastava2016efficient}. Most known protein structures are published as coordinate files, where for every atom a 3-D coordinate is estimated based on an indirect observation of the protein's electron density (see \citet{rhodes2010crystallography} for further details), and stored e.g. in the protein database \textbf{PDB} \cite{proteindatabase}. These coordinate files lack any kind of orientation and any meaningful comparison has to take this into account. \Cref{fig: pose invariant protein matching} (created with \text{PyMOL} \cite{PyMol}) shows two cartoon representations of the backbone of the protein structure 5D0U in two different poses. These two representations obtained from the same coordinate file highlight the difficulty to identify them from noisy measurements.

\begin{figure}
	\centering
	\begin{subfigure}[c]{0.49\textwidth}
		\centering
		\includegraphics[ height=0.85\textwidth,trim=.01cm .01cm .01cm .01cm, clip]{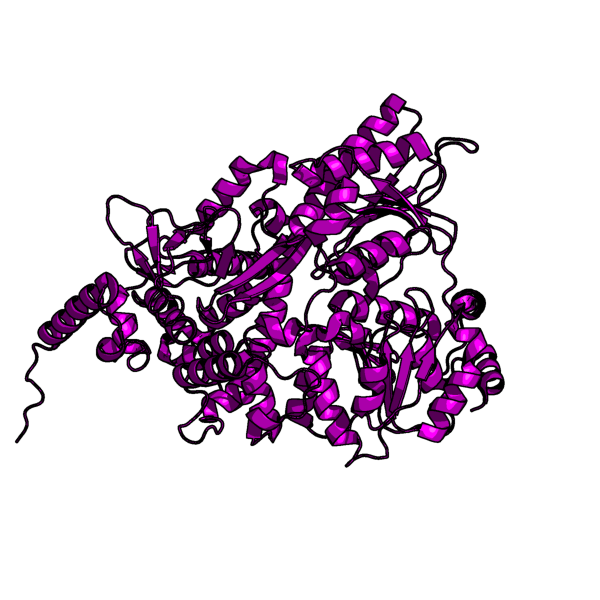}
	\end{subfigure}
	\begin{subfigure}[c]{0.49\textwidth}
				\centering
		\includegraphics[ height=0.85\textwidth,trim=.01cm .01cm .01cm .01cm, clip]{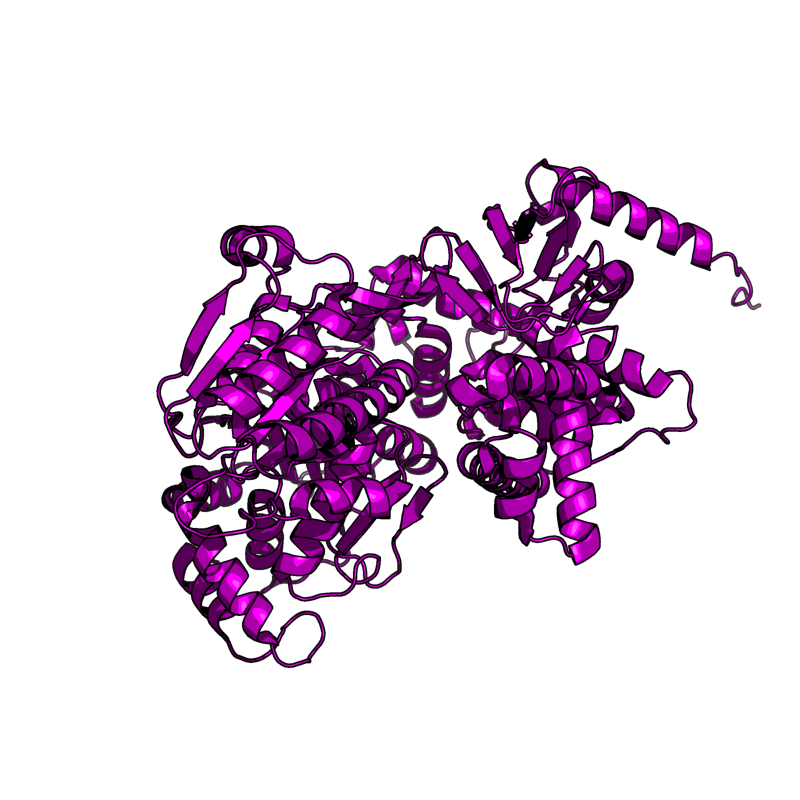}
	\end{subfigure}
	\caption{\textbf{Illustration of the proteins to be compared:} Cartoon representation of the DEAH-box RNA-helicase Prp43 from chaetomium thermophilum bound to ADP (PDB ID: 5D0U \cite{tauchert2016structural}) in two different poses. The DEAH-box helicase Prp43 unwinds double stranded RNA and rearranges RNA/protein complexes. It has essential roles in pre-mRNA splicing and ribosome biogenesis \cite{arenas1997prp43,lebaron2005splicing}.}
	\label{fig: pose invariant protein matching}
\end{figure}

\noindent Consequently, many approaches to pose invariant shape matching, classification and recognition have been suggested and studied in the literature. The majority of these methods computes and compares certain invariants or signatures in order to decide whether the considered objects are \textit{equal} up to a previously defined notion of invariance. In the literature, these methods are often called \textit{feature} (or \textit{signature}) based methods, see \citet{CardenasFeaturebasedSimilaritySearch2005} for a comprehensive survey. Some examples for features are the \textit{shape distributions} \citep{OsadaShapedistributions2002}, that are connected to the distributions of lengths, areas and volumes of an object, the \textit{shape contexts} \citep{BelongieShapematchingobject2002}, that rely in a sense on a local distribution of inter-point distances of the considered object, and \textit{reduced size functions} \citep{dAmicoNaturalpseudodistanceoptimal2008}, which count the connected components of certain lower level sets.\\
As noted by \citet{ MemoliuseGromovHausdorffDistances2007, Memoli2011}, several signatures describe different aspects of a metric between objects.
In these and subsequent papers, the author develops a unifying view point by representing an object as \textit{metric measure space} $\mmspaceX$, where $\left( \mathcal{X},\dX\right)$ is a compact metric space and $\muX$ denotes a Borel probability measure on $\mathcal{X}$. The additional probability measure, whose support is assumed to be $\mathcal{X}$, can be thought of as signaling the importance of different regions of the modeled object. Based on the original work of \citet{gromov1999metric}, \citet{Memoli2011} introduced the \textit{Gromov-Wasserstein distance} of order $p\in[1,\infty)$ between two (compact) metric measure spaces $\mmspaceX$ and $\mmspaceY$ which will be fundamental to this paper. We allow ourselves to rename it to \textit{Gromov-Kantorovich distance} to give credit to Kantorovich's initial work \cite{kantorovich1942translocation} underlying this concept (see \citet{vershik2007lv}). It is defined as	\begin{equation}
{\mathcal{GK}_p}(\mathcal{X},\mathcal{Y})=\inf_{\pi\in\mathcal{M}(\mu_\mathcal{X},\mu_\mathcal{Y})}J_p(\pi),\label{eq:GKdist}\end{equation}
where 
\begin{align*}
J_p(\pi)\coloneqq&\frac{1}{2}\left(\int_{\mathcal{X}\times\mathcal{Y}}\int_{\mathcal{X}\times\mathcal{Y}} \! \left|d_\mathcal{X}(x,x')-d_\mathcal{Y}(y,y') \right|^p \,\pi(dx\times dy)\,\pi(dx'\times dy')\right)^{\frac{1}{p}}. \end{align*}
Here, $\mathcal{M}(\mu_\mathcal{X},\mu_\mathcal{Y})$ stands for the set of all couplings of $\mu_\mathcal{X}$ and $\mu_\mathcal{Y}$, i.e., the set of all measures $\pi$ on the product space $\mathcal{X}\times\mathcal{Y}$ such that
\[\pi\left(A\times\mathcal{Y}\right)=\muX\left(A\right)~\mathrm{and}~\pi\left(\mathcal{X}\times B\right)=\muY\left(B\right)\]
for all measurable sets $A\subset\mathcal{X}$ and $B\subset\mathcal{Y}$. In Section 5 of \citet{Memoli2011} it is ensured that the Gromov-Kantorovich distance ${\mathcal{GK}_p}$ is suitable for pose invariant object matching by proving that it is a metric on the collection of all isomorphism classes of metric measure spaces.\footnote{Two metric measure spaces $\mmspaceX$ and $\mmspaceY$ are isomorphic (denoted as $\mmspaceX\cong\mmspaceY$) if and only if there exists an isometry $\phi:\mathcal{X}\to\mathcal{Y}$ such that $\phi\#\mu_\mathcal{X}=\mu_\mathcal{Y}$. Here, $\phi\#\mu_\mathcal{X}$ denotes the pushforward measure.} Hence, objects are considered to be the same if they can be transformed into each other without changing the distances between their points and such that the corresponding measures are preserved. For example, if the distance is Euclidean, this leads to identifying objects up to translations, rotations and reflections \cite{lomont2014applications}.

\noindent The definition of the Gromov-Kantorovich distance extends the \textit{Gromov-Hausdorff distance}, which is a metric between compact metric spaces \cite{MemoliComparingPointClouds2004,MemoliTheoreticalComputationalFramework2005,MemoliuseGromovHausdorffDistances2007}. The additional measure structure of metric measure spaces allows to replace the Hausdorff distance component in the definition of the Gromov-Hausdorff distance by a relaxed notion of proximity, namely the \textit{Kantorovich distance} \citep{kantorovich1942translocation}. This distance is fundamental to a variety of mathematical developments and is also known in different communities as the Wasserstein distance \cite{vaserstein1969markov}, Kantorovich-Rubinstein distance \citep{kantorovich1958space}, Mallows distance \citep{mallows1972note} or as the Earth Mover's distance \citep{rubner2000earth}.
The Kantorovich distance of order $p$ ($p\geq 1$) between two probability measures $\mu_\mathcal{X}$ and $\nu_\mathcal{X}$ on the compact metric space $\left(\mathcal{X},d_\mathcal{X}\right) $ is defined as
\begin{equation}
{\mathcal{K}_p}(\mu_\mathcal{X},\nu_\mathcal{X})=\left\lbrace\inf_{\pi\in\mathcal{M}(\mu_\mathcal{X},\nu_\mathcal{X})}\int_{\mathcal{X}\times\mathcal{X}}\! \,d^p_\mathcal{X}(x,y)\,d\pi(x,y) \right\rbrace^{\frac{1}{p}}, \label{eq:Wasserstein}
\end{equation}
where $\mathcal{M}(\mu_\mathcal{X},\nu_\mathcal{X})$ denotes the set of all
couplings of $\mu_\mathcal{X}$ and $\nu_\mathcal{X}$. Let $\mathcal{P}(\mathcal{X})$ be the space of all probability measures on $\mathcal{X}$. Then, as compact metric spaces are in particular Polish, the Kantorovich distance defines a metric on
\[\mathcal{P}_p\left(\mathcal{X}\right)=\bigg\{\mu\in \mathcal{P}(\mathcal{X})~\bigg|~\int_\mathcal{X}d^p_\mathcal{X}(x_0,x)\,d\mu(x)<\infty,~x_0\in\mathcal{X}~\mathrm{arbitrary}\bigg\} \]
and it metrizes weak convergence together with convergence of moments of order $p$ \citep{villani2008optimal}. 

\noindent Due to its measure preserving metric invariance, the Gromov-Kantorovich distance is conceptually well suited for pose invariant object discrimination. Heuristically speaking, the Gromov-Kantorovich point of view suggests to regard a data cloud as a metric measure space by itself, which  takes into account its internal metric structure (encoded in its pairwise distances), and provides a transportation between metric spaces without loss of mass. In particular, it does not rely on the embedding in an external space. For the above described protein comparison problem, the coordinate system in which the atoms are represented does not matter. Hence, the Gromov-Kantorovich distance is only influenced by the internal relations between the backbone atoms - which matches the physical understanding of structural similarity in this context. However, the practical usage of the Gromov-Kantorovich approach is severely hindered by its computational complexity: For two finite metric measure spaces $\mathcal{X}=\left\lbrace x_1,\dots,x_n\right\rbrace$ and $\mathcal{Y}=\left\lbrace y_1,\dots,y_m\right\rbrace$ with metrics $d_\mathcal{X}$ and $d_\mathcal{Y}$ and probability measures $\mu_{\mathcal{X}}$ and $\mu_\mathcal{Y}$, respectively, the computation of ${\mathcal{GK}_p}\left(\mathcal{X},\mathcal{Y}\right)$ boils down to solving a (non-convex) quadratic program \citep[Sec. 7]{Memoli2011}. This is in general NP-hard \citep{pardalos1991quadratic}. To circumvent the precise determination of the Gromov-Kantorovich distance, it has been suggested to approximate it via gradient descent \cite{Memoli2011} or to relax the corresponding optimization problem. For example, \citet{solomon2016entropic} proposed the \textit{entropic Gromov-Kantorovich distance}, which has been applied to find correspondences between word embedding spaces, an important task for machine translation \cite{alvarez2018gromov}. In this paper, we take a different route and investigate the potential of a lower bound for $\mathcal{GK}_p$, which is on the one hand extremely simple to compute in $O(n^2)$ elementary operations and on the other hand statistically accessible and useful for inference tasks, in particular for object discrimination when the data are randomly sampled or the data set is massive and subsampling becomes necessary. As this bound quantifies the optimal transport distance between the \textit{distributions of pairwise distances} (see \Cref{subsec:proposed approach}), we believe that our analysis is of quite general statistical interest beyond the described scenario.
\subsection{The Proposed Approach}\label{subsec:proposed approach}
Given two metric measure spaces, denoted as $\mmspaceX$ and $\mmspaceY$, we aim to construct an (asymptotic) test for the hypothesis that these spaces are isomorphic, viz.
\begin{align}H_0:~\mmspaceX\cong\mmspaceY,\label{eq:nullhypothesis}
\end{align}
against the alternative
\begin{align}\label{eq:alternative}H_1:~\mmspaceX\ncong\mmspaceY.
\end{align}
This test will be based on an efficiently computable empirical version of a lower bound of the Gromov-Kantorovich distance. %Several invariants used as features in feature based methods for pose invariant shape matching can be proven to be quantitatively stable under the notion of the Gromov-Kantorovich distance. With these stability statements \citet{Memoli2011} established several lower bounds for the Gromov-Kantorovich distance. The for the purpose of this paper most interesting is stated next.

\noindent Let $\muU$ be the probability measure of the random variable $d_{\mathcal{X}}(X,X')$, where $X,X'\overset{i.i.d.}{\sim}\mu_\mathcal{X}$, and let $\mu^V$ be the one of $d_{\mathcal{Y}}(Y,Y')$, with $Y,Y'\overset{i.i.d.}{\sim}\mu_\mathcal{Y}$. Then, we call $\muU$ and $\mu^V$ the \textit{distribution of the (pairwise) distances} of $\mmspaceX$ and $\mmspaceY$, respectively. Fundamental to our approach is the fact that the Gromov-Kantorovich distance between two metric measure spaces $\mmspaceX$ and $\mmspaceY$ of order $p\in[1,\infty)$ is lower bounded by
\begin{equation}
{\mathcal{GK}_p}(\mathcal{X}, \mathcal{Y})\geq\frac{1}{2}\!\left(\ptruedodstat\left(\mathcal{X},\mathcal{Y}\right)\right)^{\frac{1}{p}} \coloneqq\frac{1}{2}\left(\int_{0}^{1}\!\left| U^{-1}(t)-V^{-1}(t)\right| ^p dt \right)^{\frac{1}{p}}\label{eq:introlowerbound},
\end{equation}
where $U^{-1}$ and $V^{-1}$ are the quantile functions of $\muU$ and $\mu^V$, respectively \cite{Memoli2011}. If the right hand side of \eqref{eq:introlowerbound} is positive, so is ${\mathcal{GK}_p}$ and it is thus possible to base a statistical test for $H_0$ on an empirical version of this quantity. Furthermore, the above lower bound possesses several theoretical and practical features that make it worthy of study:

\noindent 1.) Reformulating \eqref{eq:introlowerbound} yields that the Gromov-Kantorovich distance between two metric measure spaces is lower bounded by the Kantorovich distance of the respective distributions of distances (one dimensional quantities), i.e.,
\[\ptruedodstat(\X,\Y)=\mathcal{K}_p^p\left(\muU,\mu^V\right)=\inf_{\pi\in\mathcal{M}\left(\muU,\mu^V\right)}\int_{\mathbb{R}\times\mathbb{R}}\! \,|x-y|^p\,d\pi(x,y). \]
Interestingly, there is another reformulation of $\ptruedodstat$ in terms of an optimal transport problem between the product measures $\muX\otimes\muX$ and $\muY\otimes\muY$. It is shown in \cite[Thm. 24]{chowdhury2019gromov} that
\begin{equation}\label{eq:SLB OT-form}
\ptruedodstat\left(\mathcal{X},\mathcal{Y}\right)\!=\!\inf_{\pi\in\widetilde{\mathcal{M}}}\,\int_{\X\times\X\times\Y\times\Y}\left|\dX(x,x')-\dY(y,y') \right|^pd\pi(x,x'y,y'),\end{equation}
where $\widetilde{\mathcal{M}}\coloneqq\mathcal{M}(\mu_\mathcal{X}\otimes\mu_\mathcal{X},\mu_\mathcal{Y}\otimes\muY)$. %It shows that the value of $\ptruedodstat$ is given by the optimal transport of the product measures $\muX\otimes\muX$ and $\muY\otimes\muY$ with costs $c(x,x',y,y')=\left|\dX(x,x')-\dY(y,y') \right|^p$. 
%Thus, $\ptruedodstat$ tries, just like the Gromov-Kantorovich distance, to identify points in $\X$ and $\Y$, whose distances to other points in their respective spaces are similar (however in a strictly weaker sense).} 
This representation emphasizes the relation between $\ptruedodstat$ and the Gromov-Kantorovich distance defined in \eqref{eq:GKdist}, as clearly $\pi\otimes\pi\in\widetilde{\mathcal{M}}$ for all $\pi\in\mathcal{M}(\muX,\muY)$.

\noindent 2.) Although it is known that the distribution of distances does not uniquely characterize a metric measure space \cite{Memoli2011}, it was proposed as a feature itself for feature based object matching and was shown to work well in practice in various examples \cite{OsadaShapedistributions2002,brinkman2012invariant,berrendero2016shape,gellert2019substrate}. In fact, \citet{gellert2019substrate} applied several lower bounds of the Gromov-Kantorovich distance stated in \cite{Memoli2011} for the comparison of the isosurfaces of various proteins. The authors empirically found that $\ptruedodstat$ defined in \eqref{eq:introlowerbound} has high discriminative abilities for this task. 

\noindent 3.) Generally, $\ptruedodstat$ is a simple and natural measure to compare distance matrices. Such distance matrices underlie many methods of data analysis, e.g. various multidimensional scaling techniques (see \citet{dokmanic2015euclidean}).
%The distribution of distances has also been considered more theoretically in \citet{Boutinreconstructingnpointconfigurations2004}, where the authors proved that configurations of finite points $\left\lbrace x_1,...,x_n\right\rbrace\subset\mathcal{X}$, that are not uniquely determined up to isometries by their distribution of distances, are rare in the sense that these configurations are zeros of a polynomial $\mathcal{P}$, which lies in the ring of polynomials on $\mathcal{X}^n$. Furthermore, the Kantorovich distance has proven successful in many applications for comparing two distributions, e.g., it was used to quantify the natural selection in evolutionary biology \cite{henshaw2017unified} or for the comparison of synthetic and non-synthetic fingerprints \cite{sommerfeld2018inference}.\\

\noindent 4.) The representation \eqref{eq:introlowerbound} admits an empirical version which is computable in effectively $O\left((m\vee n)^2\right)$ operations, if the computation of one distance is considered as $O(1)$. To this end, let $X_1,\dots,X_n\overset{i.i.d.}{\sim}\mu_\mathcal{X}$ and $Y_1,\dots,Y_m\overset{i.i.d.}{\sim}\mu_\mathcal{Y}$ be two independent samples and let $\mathcal{X}_n=\left\lbrace X_1,\dots,X_n\right\rbrace $ and $\mathcal{Y}_m=\left\lbrace Y_1,\dots,Y_m\right\rbrace $. The sample analog to \eqref{eq:introlowerbound} is to be defined with respect to the empirical measures and we obtain as empirical counterpart to \eqref{eq:introlowerbound} the \textit{DoD-statistic} as
\begin{equation}
 \untruncateddodstat_p=\untruncateddodstat_{p}\left(\mathcal{X}_n,\mathcal{Y}_m\right) \coloneqq\int_{0}^{1}\!\left| U^{-1}_n(t)-V^{-1}_{m}(t)\right| ^p \, dt \label{eq:introlowerboundempversion},\end{equation}
where, for $t\in\R$, $U_n$ is defined as the empirical c.d.f. of all pairwise distances of the sample $\X_{n}$,
\begin{equation}\label{eq:def of U_n}U_n(t)\coloneqq\frac{2}{n(n-1)}\sum_{1\leq i<j\leq n}\mathds{1}_{\left\lbrace d_\mathcal{X}(X_i,X_j)\leq t\right\rbrace }.\end{equation}
Analogously, we define for the sample $\Y_n$
\begin{equation}\label{eq:def of V_m}V_m(t)\coloneqq\frac{2}{m(m-1)}\sum_{1\leq k<l\leq m}\mathds{1}_{\left\lbrace d_\mathcal{Y}(Y_k,Y_l)\leq t\right\rbrace }.\end{equation}
Besides, $U^{-1}_n$ and $V^{-1}_m$ denote the corresponding empirical quantile functions. We stress that the evaluation of $\untruncateddodstat_p$ boils down to the calculation of a sum and no formal integration is required. For $n=m$ it holds
\[\untruncateddodstat_p=\frac{2}{n(n-1)}\sum_{i=1}^{n(n-1)/2}\left|d^\X_{(i)}-d^\Y_{(i)} \right|^p ,\]
 where $ d^\X_{(i)}$ denotes the $i$-th order statistic of the sample $\left\lbrace\dX(X_i,X_j) \right\rbrace_{1\leq i < j\leq n} $ and $d^\Y_{(i)}$ is defined analogously. For $n\neq m$ we obtain that 
\[\untruncateddodstat_p=\sum_{i=1}^{\frac{n(n-1)}{2}}\sum_{j=1}^{\frac{m(m-1)}{2}}\lambda_{ij}\left|d^\X_{(i)}-d^\Y_{(j)} \right|^p ,\] 
 where
 \[\lambda_{ij}=\left( \frac{i}{n}\wedge\frac{j}{m}-\frac{i-1}{n}\vee\frac{j-1}{m}\right)\indifunc{{i}{m}\wedge{j}{n}>{(i-1)}{m}\vee{(j-1)}{n}}. \]
 Here, and in the following, $a\wedge b$ denotes the minimum and $a\vee b$ the maximum of two real numbers $a$ and $b$.
\subsection{Main Results}\label{subsec:main results}
The main contributions of the paper are various upper bounds and distributional limits for the statistic defined in \eqref{eq:introlowerboundempversion} (as well as trimmed variants). Based on these, we design an asymptotic test that compares two distributions of distances and thus obtain an asymptotic test for the hypothesis $H_0$ defined in \eqref{eq:nullhypothesis}. In the course of this, we focus, for ease of notation, on the case $p=2$, i.e., we derive for $\beta\in[0,1/2)$ the limit behavior of the statistic 
\[\dodstat\coloneqq\int_{\beta}^{1-\beta}\!\left(U_n^{-1}(t)-V_m^{-1}(t)\right)^2 \, dt\]
under the hypothesis \eqref{eq:nullhypothesis} as well as under the alternative in \eqref{eq:alternative}. The introduced trimming parameter $\beta$ can be used to robustify the proposed method \cite{czado1998assessing,alvarez2008trimmed}. Most of our findings can easily be transferred to the case of $p\in[1,\infty)$, which is readdressed in \Cref{subsec:secondcase}. Next, we briefly summarize the setting in which we are working and introduce the conditions required.
\begin{set}\label{cond:setting}
	Let $\mmspaceX$ and $\mmspaceY$ be two metric measure spaces and let $\mu^U$ and $\mu^V$ denote the distributions of (pairwise) distances of the spaces $\mmspaceX$ and $\mmspaceY$, respectively. Let $U$ denote the c.d.f. of $\mu^U$, assume that $U$ is differentiable with derivative $u$ and let $U^{-1}$ be the quantile function of $U$. Let $V$, $V^{-1}$ and $v$ be defined analogously. Further, let the samples $X_1,\dots,X_n\iid \muX$ and $Y_1,\dots,Y_m\iid \muY$ be independent of each other and let $U^{-1}_n$ and $V^{-1}_m$ denote the empirical quantile functions of $U_n$ defined in \eqref{eq:def of U_n} and $V_m$ defined in \eqref{eq:def of V_m}. 
\end{set}
Since the statistic $\dodstat$ is based on empirical quantile functions, or more precisely empirical $U$-quantile functions, we have to ensure that the corresponding $U$-distribution functions are well-behaved. In the course of this, we distinguish the cases $\beta\in(0,1/2)$ and $\beta=0$. The subsequent condition guarantees that the inversion functional $\phi_{inv}:F\mapsto F^{-1}$ is Hadamard differentiable as a map from the set of restricted distribution functions into the space of all bounded functions on $[\beta,1-\beta]$, in the following denoted as $\ell^\infty[\beta,1-\beta]$. 
\begin{cond}\label{cond:firstcondition} Let $\beta\in(0,1/2)$ and let $U$ be continuously differentiable on an interval \[[C_1,C_2]=[U^{-1}(\beta)-\epsilon,U^{-1}(1-\beta)+\epsilon]\] for some $\epsilon>0$ with strictly positive derivative $u$ and let the analogue assumption hold for $V$ and its derivative $v$.
\end{cond}
When the densities of $\muU$ and $\mu^V$ vanish at the boundaries of their support, which commonly happens (see \Cref{ex:dod3}), we can no longer rely on Hadamard differentiability to derive the limit distribution of $\dodstat$ under $H_0$ for $\beta=0$. In order to deal with this case we require stronger assumptions. The following ones resemble those of \citet{mason1984weak}.
\begin{cond}\label{cond:secondcondition} Let $U$ be continuously differentiable on its support. Further, assume there exist constants $-1<\gamma_1,\gamma_2<\infty$ and $c_U>0$ such that
	\[\left| (U^{-1})'(t)\right| \leq c_Ut^{\gamma_1}(1-t)^{\gamma_2}\]
	for $t\in(0,1)$ and let the analogue assumptions hold for $V$ and $(V^{-1})'$.
\end{cond}
Both, \Cref{cond:firstcondition} and \Cref{cond:secondcondition} are comprehensively discussed in \Cref{subsec:conditions} and various illustrative examples are given there.

\noindent With the main assumptions stated we can now specify the results derived under the hypothesis $H_0$ and afterwards those under the alternative $H_1$.
Under $H_0$ we have that the distributions of distances of the considered metric measure spaces, $\muU$ and $\mu^V$, are equal, i.e., $U(t)=V(t)$ for $t\in\R$. Given \Cref{cond:firstcondition} we find that for $\beta\in(0,1/2)$ (resp. given \Cref{cond:secondcondition} for $\beta=0$) and $n,m\to\infty$
\begin{align}
\frac{nm}{n+m}\dodstat\rightsquigarrow\Xi=\Xi(\beta) \coloneqq\int_{\beta}^{1-\beta}\!\left(\mathbb{G}(t)\right) ^2 \, dt,\label{eq:defofXi}
\end{align} 
where $\mathbb{G}$ is a centered Gaussian process with covariance depending on $U$ (under $H_0$ we have $U=V$) in an explicit but complicated way, see \Cref{thm:main2}.
Further, ``$\rightsquigarrow$'' denotes weak convergence in the sense of Hoffman-J{\o}rgensen (see \citet[Part 1]{vaartWeakConvergenceEmpirical1996}). 
%In the supplement we derive this limit with results on the weak convergence of the underlying $U$-quantile processes $\mathbb{U}_n^{-1}=\sqrt{n}\left(U_n^{-1}-U^{-1}\right)$ and $\mathbb{V}_n^{-1}=\sqrt{n}\left(V_n^{-1}-V^{-1}\right)$.
Additionally, we establish in \Cref{sec:Limit distribution} a simple concentration bound for $\dodstat$
% using its connection to empirical Kantorovich distances
and demonstrate that for $\beta\in(0,1/2)$ and $\alpha\in(0,1)$ the corresponding $\alpha$-quantile of $\Xi$, which is required for testing, can be obtained by a bootstrap scheme, see \Cref{sec:bootsrapping the quantiles}.
% Therefore, we extend bootstrap results known for the empirical $U$-process.

\noindent Next, we summarize our findings under $H_1$. As we work with a lower bound, the limit behavior under the alternative is a little more complex.
% By the derivation of the proposed test statistic $\dodstat$ it is evident that we cannot distinguish the case $\truetruncateddod=0$, i.e., $U^{-1}=V^{-1}$ almost surely on $[\beta,1-\beta]$, from the hypothesis, although this may as previously argued occur in rare cases under the alternative even when $\beta=0$. In such a case the limit behavior of $\dodstat$ is the same as under the hypothesis. Fortunately, this will not be happening in most cases.
Under the additional assumption that \[\truetruncateddod\coloneqq\int_{\beta}^{1-\beta}\!\left(U^{-1}(t)-V^{-1}(t)\right)^2 \, dt > 0,\] we can prove (cf. \Cref{thm:main}) that given \Cref{cond:firstcondition} it holds for $n,m\to\infty$ and $\beta\in(0,1/2)$ (resp. given \Cref{cond:secondcondition} for $\beta=0$) that 
\begin{equation}\sqrt{\frac{nm}{n+m}}\left(\dodstat-\truetruncateddod \right)\rightsquigarrow N(0,\sigma_{U,V,\lambda}^2),\label{eq:statement2}\end{equation}
where $N(0,\sigma_{U,V,\lambda}^2)$ denotes a normal distribution with mean 0 and variance $\sigma_{U,V,\lambda}^2$ depending on $U$, $V$, $\beta$ and $\lambda=\lim_{n,m\to\infty}\frac{n}{m+n}$. 
%
%Interestingly, \eqref{eq:statement2} can be shown by establishing the Hadamard differentiability of $\phi_{inv}$ as a map from the set of restricted distribution functions into the space of integrable functions on $(0,1)$ (see \Cref{lem:Hadamardfiff2}). Such an approach is not possible in order to show \eqref{eq:defofXi}.
%In order to prove the above statement we combine distributional limit results on the empirical $U$-quantile processes $\mathbb{U}_n^{-1}$ and $\mathbb{V}_n^{-1}$ with a continuous mapping theorem and ideas from \citet{MunkNonparametricvalidationsimilar1998}.
\subsection{Applications}\label{subsec:applications}
From our theory it follows that for $\beta\in[0,1/2)$ a (robust) asymptotic level-$\alpha$-test for $H_0$ against $H_1$ is given by rejecting $H_0$ in \eqref{eq:nullhypothesis} if
\begin{equation}\label{eq:decision rule}
	\frac{nm}{n+m}\dodstat> \xi_{1-\alpha},
\end{equation}
where $\xi_{1-\alpha}$ denotes the $(1-\alpha)$-quantile of $\Xi$. The simulations in \Cref{sec:simulations} demonstrate that, although it is based on a lower bound of the Gromov-Kantorovich distance $\mathcal{GK}_p$ between the metric measure spaces, the proposed test (as well as its bootstrap version) represents a powerful method to detect deviations between metric measure spaces in the sense of \eqref{eq:nullhypothesis}. This has many possible applications.
% On the one hand, it can be applied for the comparison of two objects that can only be observed approximately (e.g. protein structures). On the other hand, it provides a useful tool to save computational cost for comparing data intensive objects by subsampling.\\ 
\begin{figure}
	\centering
	\begin{subfigure}[c]{0.49\textwidth}
			\centering
		\includegraphics[ height=0.8\textwidth,trim=.01cm .01cm .01cm .01cm, clip]{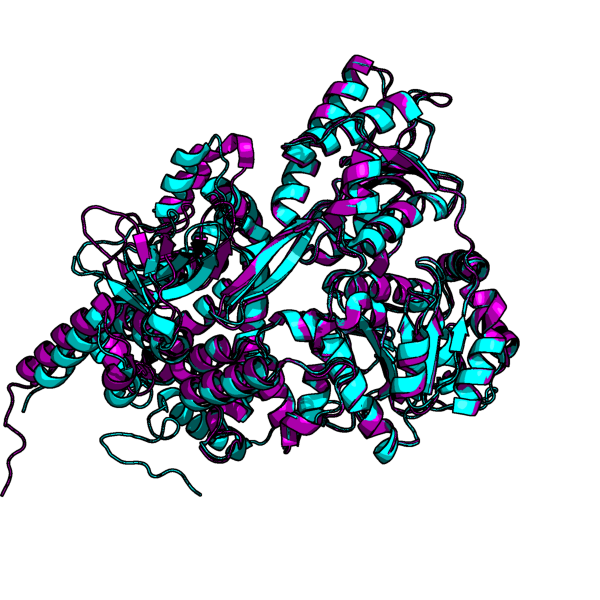}
	\end{subfigure}
	\begin{subfigure}[c]{0.49\textwidth}
			\centering
		\includegraphics[ height=0.8\textwidth,trim=.01cm .01cm .01cm .01cm, clip]{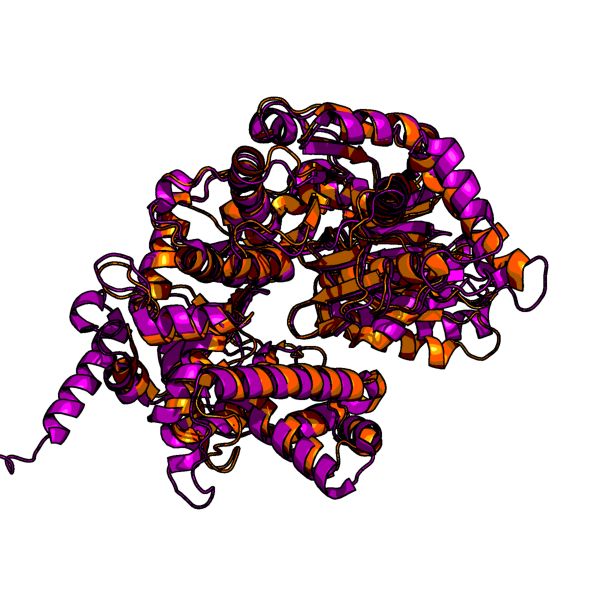}
	\end{subfigure}
	\caption{\textbf{Illustration of the proteins to be compared:} Cartoon representation of the DEAH-box RNA-helicase Prp43 from chaetomium thermophilum bound to ADP (purple, PDB ID: 5D0U \cite{tauchert2016structural}) in alignment with Prp43 from saccharomyces cerevisiae in complex with CDP (cyan, PDB ID: 5JPT \cite{robert2016functional}, left) and in alignment with the DEAH-box RNA helicase Prp2 in complex with ADP (orange, PDB ID: 6FAA \cite{schmitt2018crystal}, right). Prp2 is closely related to Prp43 and is necessary for the catalytic activation of the spliceosome in pre-mRNA splicing \cite{kim1996spliceosome}.}
	\label{fig: protein visiualization}
\end{figure}
\noindent %By construction, the test defined in \eqref{eq:decision rule} discriminates well between two metric measure spaces which are sensitive to $\truetruncateddod$, i.e., the corresponding distributions of distances are dissimilar.
% We briefly illustrate what this means in practice with the help of the protein structures illustrated in \Cref{fig: protein visiualization}.
Exemplarily, in \Cref{sec:application}, we model proteins as such metric measure spaces by assuming that the coordinate files are samples from (unknown) distributions (see \cite{rhodes2010crystallography}) and apply the theory developed to compare the protein structures depicted in \Cref{fig: protein visiualization}. Our major findings can be summarized as follows:

\noindent\textbf{5D0U vs 5JPT:} 5D0U and 5JPT are two structures of the same protein extracted from different organisms. Consequently, their secondary structure elements can almost be aligned perfectly (see \Cref{fig: protein visiualization}, left). Only small parts of the structures are slightly shifted and do not overlap in the alignment. Applying \eqref{eq:decision rule} for this comparison generally yields no discrimination between these two protein structures, as $\truetruncateddod$ is robust with respect to these kinds of differences. This robustness indeed makes the proposed method particularly suitable for protein structure comparison. 

\noindent\textbf{5D0U vs 6FAA:} 5D0U and 6FAA are structures from closely related proteins and thus they are rather similar. Their alignment (\Cref{fig: protein visiualization}, right) shows minor differences in the orientation of some secondary structure elements and that 5D0U contains an $\alpha$-helix that is not present in 6FAA. We find that $\truetruncateddod$ is highly sensitive to such a deviation from $H_0$, as the proposed procedure discriminates very well between both structures already for small sample sizes.

% What this means in practice can be highlighted in the example of protein structure comparisons, where we model proteins as metric measure spaces (for details see \Cref{sec:application}) by assuming that the coordinate files are samples from (unknown) distributions. In this setting, the proposed test easily distinguishes between a protein and a version of the same protein which exhibits a small, additional structural element.

%A related example is shown on the right of \Cref{fig: protein visiualization}, where the alignment of the proteins 5D0U (purple) and 6FAA (orange) is shown. These are two proteins from the same protein family and hence the corresponding structures are highly similar. However, 5D0U contains some structural elements that are not present in 6FAA. This strongly influences $\truetruncateddod$ and the proposed procedure discriminates between both structures already for small sample sizes very well (see \Cref{sec:application}). On the other hand, on the left of \Cref{fig: protein visiualization}, we see an alignment of the protein structures 5D0U (in purple) and 5JPT (in light blue). These structures represent the same protein extracted from two different organisms. Hence, they are almost identical. Only small parts of the structures are slightly shifted and do not overlap in the alignment. Applying \eqref{eq:decision rule} for this comparison generally yields no discrimination between these two protein structures as $\truetruncateddod$ is robust with respect to these kinds of differences. However, this is perfectly reasonable in this type of application.
\noindent Besides of testing, we mention that our theory also justifies subsampling (possibly in combination with bootstrapping) as an effective scheme to reduce the computational costs of $O\left((m\vee n)^2\right)$ further to evaluate $\dodstat$ for large scale applications.
\subsection{Related Work}
 First, we note that $U_n$ and $V_m$ can be viewed as empirical c.d.f.'s of the $N\coloneqq n(n-1)/2$ and $M\coloneqq m(m-1)/2$ random variables $d_\mathcal{X}(X_i,X_j)$ , $1\leq i<j\leq n$, and $d_\mathcal{Y}(Y_k,Y_l)$, $1\leq k<l\leq m$, respectively. Hence, \eqref{eq:introlowerboundempversion} can be viewed as the one dimensional empirical Kantorovich distance with $N$ and $M$ data, respectively. There is a long standing interest in distributional limits for the one dimensional empirical Kantorovich distance \citep{MunkNonparametricvalidationsimilar1998,BarrioTestsgoodnessfit1999,BarrioAsymptoticsL2functionals2005, bobkov2016one,sommerfeld2018inference,tameling2019empirical} as well as for empirical Kantorovich type distances with more general cost functions \cite{berthet2017central, berthet2019weak}. Apparently, the major difficulty in our setting arises from the dependency of the random variables $\{\dX(X_i,X_j)\}$ and the random variables $\{\dY(Y_k,Y_l)\}$, respectively.
% Different from the case, where the random variables are independent and identically distributed, almost no tools for dealing with empirical quantile processes based on these random variables exist, when the corresponding density vanishes at the boundary. In particular, there is no explicit distribution for the order statistics of the samples $\{\dX(X_i,X_j)\}$ and $\{\dY(Y_k,Y_l)\}$.
 Compared to the techniques available for stationary and $\alpha$-dependent sequences \cite{dede2009empirical,dedecker2017behavior}, the statistic $\dodstat$ admits an intrinsic structure related to $U$- and $U$-quantile processes \cite{NolanProcessesRatesConvergence1987,nolan1988functional,arcones1993limit,arcones1994u,wendler2012u}. %which we exploit in our proofs. % More precisely, we give an explicit partition of $\{\dX(X_i,X_j)\}$ and $\{\dY(Y_k,Y_l)\}$ into relatively large sets of independent random variables and use these to estimate the remaining dependency structure. This allows us to deal with the asymptotics of $\dodstat$ for $\beta=0$.
 Note that for $\beta>0$ we could have used the results of \citet{wendler2012u} to derive the asymptotics of $\dodstat$ as well, as they provide almost sure approximations of the empirical $U$-quantile processes $\mathbb{U}_n^{-1}\coloneqq\sqrt{n}\left(U_n^{-1}-U^{-1}\right)$ and $\mathbb{V}_m^{-1}\coloneqq\sqrt{m}\left(V_m^{-1}-V^{-1}\right)$ in $\ell^\infty[\beta,1-\beta]$, however at the expense of slightly stronger smoothness requirements on $U$ and $V$. In contrast, the case $\beta=0$ is much more involved as the processes $\mathbb{U}_n^{-1}$ and $\mathbb{V}_m^{-1}$ do in general not converge in $\ell^\infty(0,1)$ under \Cref{cond:secondcondition} and the technique in \cite{wendler2012u} fails. Under the hypothesis, we circumvent this difficulty by targeting our statistic for $\beta=0$ directly, viewed as a process indexed in $\beta$. Under the alternative, we show the Hadamard differentiability of the inversion functional $\phi_{inv}$ onto the space $\ell^1(0,1)$ and verify that this is sufficient to derive \eqref{eq:statement2}.\\
 %\textcolor{red}{Additionally, we remark that we rely on distributional limits of the $U$-quantile processes $\mathbb{U}_n^{-1}\coloneqq\sqrt{n}\left(U_n^{-1}-U^{-1}\right)$ and $\mathbb{V}_m^{-1}\coloneqq\sqrt{m}\left(V_m^{-1}-V^{-1}\right)$ in $\ell^\infty[\beta,1-\beta]$ to derive \eqref{eq:defofXi} for $\beta>0$. These distributional limits, which are derived in \Cref*{supplementsubsec:firstcase} and \Cref*{subsec:supp gen. U-quantile process weak convergence} of the supplementary material \cite{supplement} using the Hadamard differentiability of the inversion functional, are weaker than the almost sure approximation of empirical $U$-quantile processes derived in \citet{WendlerUquantileprocessesgeneralized2010}. However, they also require slightly weaker assumptions on the $U$-distribution functions $U$ and $V$.} 
 Notice that tests based on distance matrices appear naturally in several applications, see, e.g., the recent works \cite{baringhaus2004new,sejdinovic2013equivalence,montero2019two}, where the two sample homogeneity problem, i.e., testing whether two probability measures $\mu,\nu\in\mathcal{P}(\R^d)$ are equal, is considered for high dimensions. Most similar in spirit to our work is \citet{brecheteau2019statistical} who also considers an asymptotic statistical test for the hypothesis defined in \eqref{eq:nullhypothesis}. However, the latter method is based on a nearest neighbor-type approach and subsampling, which relates to a different lower bound of the Gromov-Kantorovich distance. Moreover, the subsampling scheme is such that asymptotically all distances considered are independent, while we explicitly deal with the dependency structures present in the entire sample of the $n(n-1)/2$ distances. In \Cref{subsec:bootstrap test} and \Cref{subsec:comparisontoBrecheteau} we empirically demonstrate that this leads to an increase of power and compare our test with the one proposed by \citet{brecheteau2019statistical} in more detail.
  %The relation between both approaches is discussed in more detail in \Cref{subsec:comparisontoBrecheteau}, where we compare our test with the one proposed in the latter work.\\
 
 \noindent Due to its exponential computational complexity the practical potential of the Gromov-Kantorovich distance has rarely been explored. Notable exceptions are very recent. We mention  \citet{liebscher2018new}, who suggested a poly-time algorithm for a Gromov-Hausdorff type metric on the space of phylogenetic trees, \citet{chowdhury2019gromovaveraging}, who applied the Gromov-Kantorovich distance to develop new tools for network analysis, and \citet{gellert2019substrate}, who used and empirically compared several lower bounds for the Gromov-Kantorovich distance for clustering of various redoxins, including our lower bound in \eqref{eq:introlowerbound}. In fact, to reduce the computational complexity they employed a bootstrap scheme related to the one investigated in this paper and reported empirically good results. Finally, we mention that permutation based testing for $U$-statistics (see e.g. \citet{berrett2020optimal}) is an interesting alternative to our bootstrap test and worth to be investigated further in our context. 
\subsection{Organization of the Paper}
\Cref{sec:Limit distribution} states the main results and is concerned with the derivation of \eqref{eq:defofXi}, a simple finite sample bound for the expectation of $\dodstat$ as well as the proof of \eqref{eq:statement2}. In \Cref{sec:bootsrapping the quantiles} we propose for $\beta\in(0,1/2)$ a bootstrapping scheme to approximate the quantiles of $\Xi$ defined in \eqref{eq:defofXi}. Afterwards in \Cref{sec:simulations} we investigate the speed of convergence of $\dodstat$ to its limit distribution under $H_0$ as well as its behavior under the alternative in a Monte Carlo study. In this section we further study the introduced bootstrap approximation and investigate what kind of differences are detectable employing $\dodstat$ by means of various examples. We apply the proposed test for the discrimination of 3-D protein structures in \Cref{sec:application} and compare our results to the ones obtained by the method of \citet{brecheteau2019statistical}. Our simulations and data analysis of the example introduced previously (see \Cref{fig: protein visiualization}) suggest that the proposed $\dodstat$ based test %is well suited for protein structure comparisons and that it
 outperforms the one proposed by \citet{brecheteau2019statistical} for protein structure comparisons.\\
As the proofs of our main results are quite technical and involved, the key ideas are stated in \Cref{appendix:proofs} and the full proofs are given in Part I of the supplement \cite{supplement}. Part II of the supplement \cite{supplement} contains several technical auxiliary results that seem to be folklore, but have not been written down explicitly in the literature, to the best of our knowledge.

\textbf{Notation} Throughout this paper, $\borelsetsR$ denotes the Borel sets on $\R$ and ``$\Rightarrow$'' stands for the classical weak convergence of measures (see \citet{billingsley2008probability}). Let $T$ be an arbitrary set. Then, the space $\ell^\infty(T)$ denotes the usual space of all uniformly bounded, $\R$-valued functions on $T$ and $\ell^p(T)$, $p\in[1,\infty)$, the space of all $p$-integrable, $\R$-valued functions on $T$. Given an interval $[a,b]$, let $D[a,b]$ be the c\`{a}dl\`{a}g functions on $[a,b]$ (see \citet{BillingsleyConvergenceProbabilityMeasures2013}) and $\mathbb{D}_2\subset D[a,b]$ the set of distribution functions of measures concentrated on $(a,b]$. 
%Given an interval $[a,b]\subset\mathbb{R}$ let $\mathbb{D}_1[a,b]$ be the set of all restrictions of distribution functions on $\mathbb{R}$ to $[a,b]$ and let $\mathbb{D}_2[a,b]$ be the subset of $\mathbb{D}_1$ of distribution functions of measures that concentrate on $(a,b]$.
% !TEX root = Dod.tex
\section{Limit Distributions}\label{sec:Limit distribution}
For the investigation of the limit behavior of the proposed test statistic, we have to distinguish two cases.

\noindent$\boldsymbol{\truedodstat=0}$: Then, it holds $\muU=\mu^V$, see \Cref{thm:main2}.

$\boldsymbol{\truedodstat>0}$: Here, we have $\muU\neq\mu^V$, see \Cref{thm:main}.

These cases do not correspond exactly to the hypothesis $H_0$ and the alternative $H_1$. Under $H_0$ it always holds that the distributions of distances of the considered metric measure spaces are equal. Therefore, the limit distribution of $\dodstat$ in this case is essential for proving that the test induced by \eqref{eq:decision rule} asymptotically is a level $\alpha$ test. However, as already mentioned in \Cref{subsec:proposed approach} the distributions of distances do not characterize isomorphic metric measure spaces uniquely, i.e., $\muU=\mu^V$ can happen in some rare cases under $H_1$, as well. Consequently, to analyze the test's asymptotic power we assume that the distributions of distances of the considered metric measure spaces do not coincide, i.e., $\truedodstat>0$.
% On the other hand, consider two non isomorphic metric measure spaces, whose quantile functions of the respective distributions of distances are indeed different on a non-null set $A\subset[\beta,1-\beta]$.
%In order to analyze the asymptotic power of the test in this setting it is necessary to investigate the limit behavior of $\dodstat$ in Case 2.\\
%Both cases will be treated extensively in the next sections. However, before we come to the treatment of the first case, we have to establish some technical assumptions.\vspace*{3mm}\\
\subsection{Conditions on the distributions of distances}\label{subsec:conditions}
%Since the statistic proposed in this writing is based on empirical quantile functions, or more precisely empirical $U$-quantile functions, we have to ensure that the corresponding $U$-distribution functions are well-behaved. In the course of this, we distinguish the cases $\beta\in(0,1/2)$ and $\beta=0$.\vspace*{3mm}\\  The subsequent condition comprises some of the standard regularity assumptions when working with quantile functions on the interval $[\beta,1-\beta],~\beta\in(0,1/2)$.
%\begin{cond}\label{cond:firstcondition} Let $\mmspaceX $ be a metric measure space such that $U$ is continuously differentiable on an interval \[[C_1,C_2]=[U^{-1}(\beta)-\epsilon,U^{-1}(1-\beta)+\epsilon]\] for some $\epsilon>0$ with strictly positive derivative $u$ and let the analogue assumptions hold for $\mmspaceY$, $V$ and $v=V'$.
%\end{cond}
%When densities of distributions of distances vanish at the boundary of their support, the case $\beta=0$ becomes considerably more complicated and requires stronger assumptions. The following ones resemble those of \citet{mason1984weak}.
%\begin{cond}\label{cond:secondcondition} Let $\mmspaceX $ be a metric measure space such that $U$ is continuously differentiable on its support with $u=U'$. Assume constants $-1<\gamma_1,\gamma_2<\infty$ and $c_U>0$ such that
%	\[\left| (U^{-1})'(t)\right| \leq c_Ut^{\gamma_1}(1-t)^{\gamma_2}\]
%for $t\in(0,1)$ and let the analogue assumptions hold for $\mmspaceY$, $V$ and $v=V'$.
%\end{cond}
Before we come to the limit distributions of the test statistic $\dodstat$ under $H_0$ and $H_1$, we discuss \Cref{cond:firstcondition} and \Cref{cond:secondcondition}. We ensure that these conditions comprise reasonable assumptions on metric measure spaces that are indeed met in some standard examples.   
\begin{ex}\label{ex:dod3}
	\begin{enumerate} 
	\item Let $\mathcal{X}$ be the unit square in $\R^2$, $d_\mathcal{X}(x,y)=||x-y||_\infty$ for $x,y\in\R^2$ and let $\mu_\mathcal{X}$ the uniform distribution on $\mathcal{X}$. Let $X,X'\iid\muX$.
	%	 We realize that since $X=(X_1,X_2)^t,~X'=(X_1',X_2')^t$ with $X_1,X_2,X_1'X_2'\iid U[0,1]$, where $U[0,1]$ denotes the uniform distribution on $[0,1]$, it holds
%	\[\dX(X,X')=\max\left\lbrace|X_1-X_1'|, |X_2-X_2'|\right\rbrace =\max\left\lbrace|Z_1|,|Z_2| \right\rbrace. \]
%Here, $Z_1$ and $Z_2$ are independent, identically distributed random variables with density
%\[f_{Z}(t)=
%\begin{cases}
%t+1, &\mathrm{if}~ -1\leq t< 0 \\
%1-t, &\mathrm{if}~ 0\leq t\leq 1\\
%0, & \mathrm{else}.
%\end{cases}
%\]
%As the density $f_Z$ is symmetric, the distribution of $|Z_1|$ and respectively of $|Z_2|$ can be calculated with little effort. Further, the independence of $|Z_1|$ and $|Z_2|$ guarantees that one can easily determine the distribution of $\max\left(\left|Z_1\right|,\left| Z_2\right| \right)$. Thus, after some calculation we find that the density of $\dX(X,X')$ is given as
Then, a straight forward calculation shows that the density $u_1$ of $\dX(X,X')$ is given as
\[u_1(s)=\begin{cases}
4s^3-12s^2+8s, &\mathrm{if}~ 0\leq s\leq1\\
0, & \mathrm{else}.
\end{cases}\]
For an illustration of $u_1$ see \Cref{fig:doddensities} (a). Obviously, $u_1$ is strictly positive and continuous on $(0,1)$ and thus \Cref{cond:firstcondition} is fulfilled for any $\beta\in(0,1/2)$ in the present setting. Furthermore, we find in this framework that for $t\in(0,1)$	\begin{equation}\label{eq:example quantile function}U_1^{-1}(t)=-\sqrt{-\sqrt{t}+1}+1.\end{equation}
Since 
\[\left|(U_1^{-1})'(t)\right| =\frac{1}{4\sqrt{1-\sqrt{t}}\sqrt{t}}\leq t^{-\frac{1}{2}}(1-t)^{-\frac{1}{2}}\]
for $t\in(0,1)$, the requirements of \Cref{cond:secondcondition} are satisfied.
\item Let $\mathcal{X}$ be a disc in $\mathbb{R}^2$ with diameter one, $d_\mathcal{X}$ the Euclidean distance and $\mu_\mathcal{X}$ the uniform distribution on $\mathcal{X}$. Let $X,X'\overset{i.i.d.}{\sim}\mu_\mathcal{X}$. Then, the density $u_2$ of $d_\mathcal{X}(X,X')$ (see \cite{moltchanov2012distance}, shown in \Cref{fig:doddensities} (a)) is given as
\[u_2(s)= \begin{cases}
\begin{aligned}
8s\left(\frac{2}{\pi}\arccos(s)-\frac{2s}{\pi}\sqrt{1-s^2} \right), 
\end{aligned} &\mathrm{if}~ 0\leq s \leq 1 \\
0 ,& \mathrm{else}.
\end{cases}\]
Once again, we can easily verify \Cref{cond:firstcondition} for any $\beta\in(0,1/2)$ in this setting. 
%As to be expected from a geometrical point of view, it strongly resembles the density $u_1$.
Additionally, we find by an application of \Cref{lem:Assumptions for unif and Eucl} below with $\epsilon=1/4$, $\eta=2$ and $c_\X=\frac{16}{\pi}$ that also \Cref{cond:secondcondition} is met.
\item[3.] Let $\mathcal{X}=[0,1]^2\cup\left([5,6]\times[0,1]\right)$ be the union of two horizontally translated unit squares in $\mathbb{R}^2$. Once again, let $d_\mathcal{X}$ be the distance induced by the supremum norm and let $\mu_\mathcal{X}$ be the uniform distribution on $\mathcal{X}$. Let $X,X'\overset{i.i.d.}{\sim}\mu_\mathcal{X}$.
Then, the density $u_3$ of $\dX(X,X')$ (see \Cref{fig:doddensities} (b)) is given as
\[u_3(s)=\begin{cases}
	2s^3-6s^2+4s, &\mathrm{if}~ 0\leq s\leq1\\
	\frac{1}{2}s-2, &\mathrm{if}~ 4\leq s< 5 \\
	3-\frac{1}{2}s, &\mathrm{if}~ 5\leq s\leq 6\\
	0, & \mathrm{else}.
	\end{cases}
	\]
	We obtain that $\prob(\dX(X,X')\in[0,1])=\prob(\dX(X,X')\in[4,5])=0.5$, hence there exists no $\beta\in(0,1/2)$ such that $u_3$ is strictly positive on $[C_1,C_2]=[U^{-1}(\beta)-\epsilon,U^{-1}(1-\beta)+\epsilon]$, i.e., \Cref{cond:firstcondition} cannot be satisfied in this setting. This is due to the fact that the set $\mathcal{X}$
	is disconnected such that the diameters of both connected parts are smaller than the gap in between. In such a case the cumulative distribution function of $\dX(X,X')$ is not strictly increasing and thus \Cref{cond:firstcondition} cannot hold. The same arguments show that neither does \Cref{cond:secondcondition}. 
\end{enumerate}
\end{ex}
\begin{rem}
In the above examples we have restricted ourselves to $\X\subset\R^2$ for the ease of readability. Clearly, the same arguments (with more tedious calculations) can be applied to general $\X\subset\R^d$, $d\geq 2$.
\end{rem}
\begin{figure}
	\captionsetup[sub]{justification=centering}
	\centering
	\begin{subfigure}[c]{0.32\textwidth}
		\centering
		\includegraphics[width =0.88 \textwidth,height=0.88\textwidth]{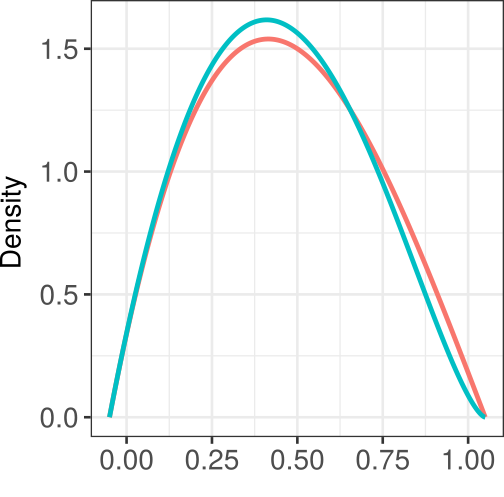}
		\caption*{~~~~~(a)}
	\end{subfigure}
\hspace*{1cm}
	\begin{subfigure}[c]{0.32\textwidth}
		\centering
		\includegraphics[width = 0.85\textwidth, height=0.85\textwidth]{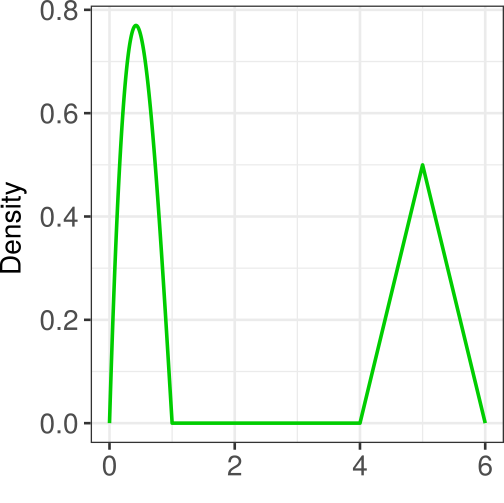}
				\caption*{~~~~~(b)}
	\end{subfigure}
	\caption{\textbf{Distribution of distances:} Representation of the densities $u_1$ (Figure (a), red), $u_2$ (Figure (a), blue) and $u_3$ (Figure (b)) calculated in \Cref{ex:dod3}. }
	\label{fig:doddensities}
\end{figure}
In many applications it is natural to model the objects at hand as compact subsets of $\mathbb{R}^2$ or $\mathbb{R}^3$ and to equip them with the Euclidean metric and the uniform distribution on the respective sets. Hence, the distributions of distances of these metric measure spaces deserve special attention.  Before we can state the next result, which provides simpler conditions than \Cref{cond:firstcondition} and \Cref{cond:secondcondition} in this setting, we have to introduce some notation.\\
Let $A\subset\mathbb{R}^d$, $d\geq 2$ be a bounded Borel set and let  $\lambda^d$ denote the Lebesgue measure in $\mathbb{R}^d$. Let $\mathbb{S}_{d-1}$ stand for the unit sphere in $\Rd$. Then, $y\in \Rd$ is determined by its polar coordinates $(t,v)$, where $t=\left\| y\right\|_2$ and $v\in\mathbb{S}_{d-1}$ is the unit length vector $y/t$. Thus, we define the \textit{covariance function} \cite[Sec. 3.1]{stoyan2008stochastic} for $y=tv\in\Rd$ as
\[K_A(t,v)=K_A(y)=\lambda^d\left( A\cap(A-y)\right),\]
where $A-y=\left\lbrace a-y:a\in A\right\rbrace $, and introduce the \textit{isotropized set covariance function} \cite[Sec. 3.1]{stoyan2008stochastic}
\[k_A(t)=\frac{1}{(\lambda^d(A))^2}\int_{\mathbb{S}_{d-1}}\!K_A(t,v)\,dv.\]
Furthermore, we define the diameter of a metric space $\left(\X,\dX\right)$ as $\diam{\X}=\sup\{\dX(x_1,x_2):x_1,x_2\in\X\} $.
\begin{lemma}\label{lem:Assumptions for unif and Eucl} \ 
	Let $\mathcal{X}\subset\mathbb{R}^d$, $d\geq2$, be a compact Borel set, $\dX$ the Euclidean metric and $\muX$ the uniform distribution on $\mathcal{X}$. Let $\diam{\X}=D$.
	\begin{itemize}[leftmargin=*]
		\item[(i)] If $k_\X$ is strictly positive on $[0,D)$, then the induced metric measure space $\mmspaceX$ meets the requirements of \Cref{cond:firstcondition} for any $\beta\in(0,1/2)$.
		\item[(ii)] If additionally there exists $\epsilon>0$ and $\eta>0$ such that
		\begin{enumerate}
		\item the function $k_\X$ is monotonically decreasing on $(D-\epsilon,D)$;
		\item we have $k_\X(t)\geq c_\X (D-t)^\eta$ for $t\in (D-\epsilon,D)$, where $c_\X$ denotes a finite, positive constant,
		\end{enumerate}
	then $\mmspaceX$ also fulfills the requirements of \Cref{cond:secondcondition}.
	\end{itemize}
\end{lemma}
The full proof of the above lemma is deferred to Section B.1 of \citet{supplement}.
\subsection{The Case \texorpdfstring{\bm{$\truedodstat=0$}}{Case 1}}\label{subsec:first case} Throughout this subsection we assume that the distribution of distances of the two considered metric measure spaces $\mmspaceX$ and $\mmspaceY$ are equal, i.e., that $\muU=\mu^V$. Assume that $X_1,\dots,X_n\iid\muX$ and $Y_1,\dots,Y_m\iid\muY$ are two independent samples. The next theorem states that $\dodstat$, based on these samples, converges, appropriately scaled, in distribution to the integral of a squared Gaussian process. The case $\beta\in(0,1/2)$ is considered in part $(i)$, whereas the case $\beta=0$ is considered in part $(ii)$. 
% main theorem hypothesis 
\begin{thm}\label{thm:main2} Assume \Cref{cond:setting} and suppose that $\muU=\mu^V$.\begin{itemize}[leftmargin=*]
\item[(i)] Let \Cref{cond:firstcondition} be met and let $m,n\to\infty$ such that ${n}/({n+m})\to\lambda\in \left(0,1 \right) $. Then, it follows
	\begin{align*}
	\frac{nm}{n+m}\int_{\beta}^{1-\beta}\!\left(U_n^{-1}(t)-V_m^{-1}(t)\right)^2 \, dt\rightsquigarrow\Xi\coloneqq\int_{\beta}^{1-\beta}\!\mathbb{G}^2(t) \, dt,
	\end{align*}
	where $\mathbb{G}$ is a centered Gaussian process with covariance
	\begin{equation}\label{eq:definition of the covariance}
	\Cov\left(\mathbb{G}(t),\mathbb{G}(t') \right)\!=\! \frac{4}{(u\circ U^{-1}(t))(u\circ U^{-1}(t'))}\Gamma_{\dX}(U^{-1}(t),U^{-1}(t')). \\ 
	\end{equation}
	Here,
	\begin{align*}
	\Gamma_{\dX}\left(t,t'\right)= \int\!\int\!\mathds{1}_{\left\lbrace d_\mathcal{X}(x,y)\leq t\right\rbrace }\,d\muX(y)\int\!\mathds{1}_{\left\lbrace d_\mathcal{X}(x,y)\leq t'\right\rbrace }\,d\muX(y)\,d\muX(x)\\
	-\int\!\int \mathds{1}_{\left\lbrace d_\mathcal{X}(x,y)\leq t\right\rbrace }\,d\muX(y)\,d\muX(x)\int\!\int \!\mathds{1}_{\left\lbrace d_\mathcal{X}(x,y)\leq t'\right\rbrace }\,d\muX(y)\,d\muX(x). 
	\end{align*}
	\item[(ii)] If we assume \Cref{cond:secondcondition} instead of \Cref{cond:firstcondition}, then the analogous statement holds for the untrimmed version, i.e., for $\beta =0$.
		\end{itemize}
\end{thm}
The main ideas for the proof of \Cref{thm:main2} are illustrated in \Cref{appendix:proofs} and the full proof can be found in the supplementary material \cite[Sec. B.2]{supplement}.
% For a better understanding of the derived limit distribution we calculate the covariance structure of the Gaussian process in a simple, familiar example explicitly.
\begin{ex}\label{ex:explixitvar}
	Recall \Cref{ex:dod3}, i.e., $\mathcal{X}$ is the unit square in $\mathbb{R}^2$, $d_\mathcal{X}$ is the distance induced by the supremum norm and $\mu_\mathcal{X}$ is the uniform distribution on $\mathcal{X}$. Let $X,X'\iid\mu_\mathcal{X}$. From \eqref{eq:example quantile function}, we obtain
	\[U_1(t)=\prob\left(\left\| X-X'\right\|_\infty\leq t\right)=\begin{cases}
	0, &\mathrm{if}~t\leq 0\\
	\left(2t-t^2 \right)^2,&\mathrm{if}~ 0\leq t\leq 1\\
	1, & t\geq1. \end{cases}\]
	Hence, in order to obtain an explicit expression of the covariance structure in the present setting, it remains to determine the first term of $\Gamma_{\dX}$	\begin{align*}
	\Gamma_{\dX,1}(t,t')&\coloneqq\int\!\int\!\mathds{1}_{\left\lbrace d_\mathcal{X}(x,y)\leq t\right\rbrace }\,d\muX(y)\int\!\mathds{1}_{\left\lbrace d_\mathcal{X}(x,y)\leq t'\right\rbrace }\,d\muX(y)\,d\muX(x)\\
&=\begin{cases}
	\left(-\frac{1}{3}t'^3-t'^2t-2t't^2+4t't\right)^2, &\mathrm{if}~ t'\leq t<1/2\\
	\left(-\frac{1}{3}t'^3-t'^2t-2t't^2+4t't\right)^2 ,&\!\begin{aligned}\mathrm{if}~& t'<1/2\leq t,\\~& t'\leq 1-t \end{aligned} \\
	\left(-(t'-t)^2-t't^2+t'+\frac{1}{3}t^3+t-\frac{1}{3}\right)^2, &\!\begin{aligned}\mathrm{if}~& t'<1/2\leq t,\\~& t'> 1-t\end{aligned}\\
	\left(-(t'-t)^2-t't^2+t'+\frac{1}{3}t^3+t-\frac{1}{3}\right)^2, &\mathrm{if}~ 1/2\leq t'\leq t\leq 1. \end{cases}
	\end{align*}
%In \Cref{fig:covaricane} the covariance function of $\mathbb{G}$ for this example is visualized as heat map. There, we see that it mostly takes values between 0 and 0.06 and that its values generally increase with $t$ and $t'$ until either $t$ or $t'$ come close to one, the diameter of the set $\X$. In that case, the covariance functions drops to zero. \\ 
%\begin{figure}
%	\centering
%		\includegraphics[width = 0.5\textwidth,height=0.35\textwidth,trim=.01cm .01cm .01cm .01cm, clip]{Graphics/supsquareexcovarianceplot.pdf}
%	\caption{\textbf{Representation of the covariance:} Illustration of the covariance structure calculated in \Cref{ex:explixitvar} as heat map.}
%	\label{fig:covaricane}
%\end{figure}
\end{ex}
Based on the limit distribution derived in \Cref{thm:main2} it is possible to construct an asymptotic level $\alpha$ test using (estimates of) the theoretical $1-\alpha$ quantiles of $\Xi$, denoted as $\xi_{1-\alpha}$, in \eqref{eq:decision rule}. However, in order to study its finite sample bias, the following bound is helpful (for its proof see \Cref{subsec:proof of the finite sample boud} and \citet[Sec. B.3]{supplement}).
\begin{thm}\label{thm:finite sample bound for p=2}
	Let $\beta \in [0,1/2)$, let \Cref{cond:setting} be met and suppose that $\mu^U=\mu^V$. Further, let \[J_2\left(\mu^U\right) =\int_{-\infty}^{\infty}\! \frac{ U(t)(1-U(t)) }{u(t)}\,dt<\infty.\]
	Then it holds for $m,n\geq 3$ that
	\[\Eargs{\dodstatnn}\leq \left( \frac{8}{n+1}+\frac{8}{m+1}\right) J_2(\muU).\]
\end{thm}
 For instance in the setting of \Cref{ex:explixitvar} it holds $J_2\left(\mu^U\right)=\frac{5}{48}<\infty$ and thus $\Eargs{\dodstatnn}\leq \frac{5}{6}\left( \frac{1}{n+1}+\frac{1}{m+1} \right) $ for $m,n\geq 3$.
\subsection{The Case \texorpdfstring{$\bm{\truedodstat>0$}}{Case 2}} \label{subsec:secondcase}
 In this subsection, we are concerned with the behavior of $\dodstat$ given that the distributions of distances of the metric measure spaces $\mmspaceX$ and $\mmspaceY$ do not coincide.
 %Let $X_1,\dots,X_n\iid\muX$ and $Y_1,\dots,Y_m\iid\muY$ be two independent samples.
% As integrals over scaled Gaussian processes are normally distributed under certain conditions, which are given in the present framework, we can even characterize the limit distribution obtained in this setting explicitly. 
%	Assume that \Cref{cond:firstcondition} holds. Let $\truetruncateddod\neq 0$.
%	Then, we have for $\frac{n}{m+n}\to\lambda\in(0,1)$ that
%	\begin{align*}\sqrt{\frac{nm}{n+m}}\left(\int\limits_{\beta}^{1-\beta}\!|U_n^{-1}(t)-V_m^{-1}(t)|^2 \, dt-\int\limits_{\beta}^{1-\beta}\!|U^{-1}(t)-V^{-1}(t)|^2 \, dt \right)\\ \rightsquigarrow \zeta_1\int\limits_{\beta}^{1-\beta}\!(U^{-1}(t)-V^{-1}(t))\mathbb{G}_1(t) \, dt -\zeta_2\int\limits_{\beta}^{1-\beta}\!(U^{-1}(t) - V^{-1}(t))\mathbb{G}_2(t) \, dt,
%	\end{align*}
%	where $\zeta_1=2\sqrt{\lambda}$ and $\zeta_2=2\sqrt{1-\lambda}$. Further, $\mathbb{G}_1$ and $\mathbb{G}_2$ are centered, independent continuous Gaussian processes with covariance structures
%	\begin{align*}
%	&\Cov\left(\mathbb{G}_1(t),\mathbb{G}_1(t') \right)= \frac{4}{(u\circ U^{-1}(t))(u\circ U^{-1}(t'))}\Gamma_{\dX}(U^{-1}(t),U^{-1}(t')) 
%	\end{align*}
%	and
%	\begin{align*}
%	&\Cov\left(\mathbb{G}_2(t),\mathbb{G}_2(t') \right)= \frac{4}{(v\circ V^{-1}(t))(v\circ V^{-1}(t'))}\Gamma_{\dY}(V^{-1}(t),V^{-1}(t')),
%	\end{align*}
%	where $u=U'$ and $v=V'$.
%\end{thm}
Just as for \Cref{thm:main2}, we distinguish the cases $\beta\in(0,1/2)$ and $\beta = 0$. 
\begin{thm}\label{thm:main} Assume \Cref{cond:setting}.
	\begin{itemize}[leftmargin=*]
	\item[(i)] Assume that \Cref{cond:firstcondition} holds, let $m,n\to\infty$ such that $\frac{n}{n+m}\to\lambda\in \left(0,1 \right) $ and let $\truetruncateddod> 0$.
	Then, it follows that
		\[\sqrt{\frac{nm}{n+m}}\left(\dodstat-\truetruncateddod \right)\]
		converges in distribution to a normal distribution with mean 0 and variance
		\begin{align*}
		&16\lambda\!\!\!\!\!\!\!\int\limits_{U^{-1}(\beta)}^{U^{-1}(1-\beta)}\int\limits_{U^{-1}(\beta)}^{U^{-1}(1-\beta)}\!\!\!\!\!\!\!\!(x-V^{-1}(U(x)))(y-V^{-1}(U(y)))\Gamma_{\dX}\left(x,y\right) dxdy\\
		+&16(1-\lambda)\!\!\!\!\!\!\!\int\limits_{V^{-1}(\beta)}^{V^{-1}(1-\beta)}\int\limits_{V^{-1}(\beta)}^{V^{-1}(1-\beta)}\!\!\!\!\!\!\!\!(U^{-1}(V(x))-x))(U^{-1}(V(y))-y)\Gamma_{\dY}\left(x,y\right) dxdy\label{eq:maincorvar}.
		\end{align*}
		Here, $\Gamma_{\dX}\left(x,y\right)$ is as defined in \Cref{thm:main2} and $\Gamma_{\dY}\left(x,y\right)$ is defined analogously.
		\item[(ii)] If we assume \Cref{cond:secondcondition} instead of \Cref{cond:firstcondition}, then the analogous statement holds for the untrimmed version, i.e., for $\beta =0$.
	\end{itemize}
\end{thm}
The proof of \Cref{thm:main} is sketched in \Cref{appendix:proofs}. A detailed version is given in Section B.4 of \cite{supplement}.
\begin{rem}\label{rem:maintheoremextraassumption}
	The assumptions of \Cref{thm:main} $(i)$ include that $\beta$ is chosen such that
	\begin{equation*}\label{assump:slb}
	\truetruncateddod> 0.
	\end{equation*}
	 Suppose on the other hand that $\muU\neq\mu^V$, but $\truetruncateddod= 0$, i.e., their quantile functions agree Lebesgue almost surely on the considered interval $[\beta,1-\beta]$. Then, the limits found in \Cref{thm:main} are degenerate and it is easy to verify along the lines of the proof of \Cref{thm:main2} that $\dodstat$ exhibits the same distributional limit as in the case $\truedodstat=0$.
\end{rem}
\begin{rem}
	So far we have restricted ourselves to the case $p=2$. However, most of our findings directly translate to results for the statistic $\untruncateddodstat_p$, $p\in[1,\infty)$, defined in \eqref{eq:introlowerboundempversion}. Using the same ideas one can directly derive \Cref{thm:main2} and \Cref{thm:finite sample bound for p=2} for (a trimmed version) of $\untruncateddodstat_p$ (see Sections B.2 and B.3 of \citet{supplement}) under slightly different assumptions. Only the proof of \Cref{thm:main} requires more care (see \cite[Sec. B.4]{supplement}).
\end{rem}
\section{Bootstrapping the Quantiles}\label{sec:bootsrapping the quantiles}
The quantiles of the limit distribution of $\dodstat$ under $H_0$ depend on the unknown distribution $U$ and are therefore in general not accessible. One possible approach, which is quite cumbersome, is to estimate the covariance matrix of the Gaussian limit process $G$ from the data and use this to approximate the quantiles required. Alternatively, we suggest to directly bootstrap the quantiles of the limit distribution of $\dodstat$ under $H_0$. To this end, we define and investigate the bootstrap versions of $U_n$, $U_n^{-1}$ and $\mathbb{U}_n^{-1}\coloneqq \sqrt{n}\left( U_n^{-1}-U^{-1}\right)$.

\noindent Let $\mu_n$ denote the empirical measure based on the sample $X_1,\dots,X_n$. Given the sample values, let ${X}_1^*,\dots,{X}_{n_B}^*$ be an independent identically distributed sample of size $n_B$ from $\mu_n$. Then, the bootstrap estimator of $U_n$ is defined as 
\begin{align*}
&{U}_{n_B}^*(t):= \frac{2}{n_B(n_B-1)}\sum_{1\leq i<j\leq n_B}\mathds{1}_{\left\lbrace d_\mathcal{X}({X}_i^*,{X}_j^*)\leq t\right\rbrace},
\end{align*}
the corresponding bootstrap empirical $U$-process is for $t\in\R$ given as ${\mathbb{U}}_{n_B}^*(t)=\sqrt{n_B}\big({U}_{n_B}^*(t)-U_n(t) \big) $ and the corresponding bootstrap quantile process for $t\in(0,1)$ as $\left( {\mathbb{U}}_{n_B}^*\right) ^{-1}(t)=\sqrt{n_B}\Big(\big( {U}_{n_B}^*\big) ^{-1}(t)-U_n^{-1}(t) \Big)$.\\
%To formally state the goal of this section, we aim to approximate the quantiles of $\Xi$, where
%\[\Xi=\int_{\beta}^{1-\beta}\!\left( \mathbb{G}(t)\right)^2 \,dt,\]
%by the quantiles of its bootstrapped version
%\begin{equation}
%\hat{\Xi}_{n_B}\coloneqq\int_{\beta}^{1-\beta}\!\left(\hat{\mathbb{U}}_{n_B}^{-1}(t)\right)^2 \,dt.\label{eq:defofhatXinun}\end{equation}
\noindent One can easily verify along the lines of the proof of \Cref{thm:main2} that for $n\to\infty$ it also holds for $\beta\in(0,1/2)$
\begin{equation}\label{eq:bootstrapideabasis}
\int_{\beta}^{1-\beta}\!\left( \mathbb{U}_n^{-1}(t)\right)^2 \,dt\rightsquigarrow\Xi=\int_{\beta}^{1-\beta}\!\mathbb{G}^2(t) \,dt.
\end{equation}
Hence, this suggests to to approximate the quantiles of $\Xi$ by the quantiles of its bootstrapped version
\begin{equation}
{\Xi}_{n_B}^*\coloneqq\int_{\beta}^{1-\beta}\!\left(\left( {\mathbb{U}}_{n_B}^*\right) ^{-1}(t)\right)^2 \,dt.\label{eq:defofhatXinun}\end{equation}
	Let $\beta\in(0,1/2)$, suppose that \Cref{cond:firstcondition} holds, let $\sqrt{n_B}=o(n)$ and let $\xi_{n_B,\alpha}^{(R)}$ denote the empirical bootstrap quantile of $R$ independent bootstrap realizations ${\Xi}_{n_B}^{*,(1)},\dots,{\Xi}_{n_B}^{*,(R)}$. Under these assumptions, we derive (cf. Section C of the supplement \cite{supplement}) that for any $\alpha\in(0,1)$ it follows
\begin{equation}\label{eq:bootstrapconsitency}\lim_{n,n_B,R\to\infty}\prob\left(\int_{\beta}^{1-\beta}\!\left(\mathbb{U}_n^{-1}(t)\right)^2\,dt \geq \xi_{n_B,\alpha}^{(R)} \right)=\alpha. \end{equation}

%As one can easily verify along the lines of the proof of \Cref{thm:main2} that for $n\to\infty$ it holds
%\[\int_{\beta}^{1-\beta}\!\left(\mathbb{U}_n^{-1}(t)\right)^2\,dt\rightsquigarrow\Xi,\]
Because of \eqref{eq:bootstrapideabasis} the statement \eqref{eq:bootstrapconsitency} guarantees the consistency of $\xi_{n_B,\alpha}^{(R)}$ for $n,n_B,R\to\infty$. Hence, a consistent bootstrap analogue of the test defined by the decision rule \eqref{eq:decision rule} is for $\beta\in(0,1/2)$ given by the \textit{bootstrapped Distribution of Distances (DoD)}-test 
\begin{equation}\label{eq:bootstraptest}\Phi_{DoD}^*(\mathcal{X}_n,\mathcal{Y}_m)=\begin{cases}
1, &\text{if}~\frac{nm}{n+m}\dodstat>\xi_{n_B,1-\alpha}^{(R)}\vspace*{5mm}\\
0, &\text{if}~\frac{nm}{n+m}\dodstat\leq \xi_{n_B,1-\alpha}^{(R)}.
\end{cases}\end{equation}
% This is, especially considering the complicated covariance structure from the simple setting in \Cref{ex:explixitvar}, extremely convenient for the practical application of the DoD-test analyzed next.
% !TEX root = Dod.tex
\section{Simulations}\label{sec:simulations}
We investigate the finite sample behavior of $\dodstatnn$ in Monte Carlo simulations. To this end, we simulate the speed of convergence of $\dodstatnn$ under $H_0$ to its limit distribution (\Cref{thm:main2}) and its behavior under $H_1$. Moreover, we showcase the accuracy of the approximation by the bootstrap scheme introduced in \Cref{sec:bootsrapping the quantiles} and investigate what kind of differences are detectable in the finite sample setting using the bootstrapped DoD-test $\Phi_{DoD}^*$ defined in \eqref{eq:bootstraptest}. All simulations were performed using $\mathsf{R}$ (\citet{Rbasicversion}). 
\subsection{The Hypothesis}\label{subsec:hypothesis simulation}
We begin with the simulation of the finite sample distribution under the hypothesis and consider the metric measure space $\mmspaceX$ from \Cref{ex:dod3}, where $\mathcal{X}$ denotes the unit square in $\R^2$, $d_\mathcal{X}$ the distance induced by the supremum norm and $\muX$ the uniform distribution on $\mathcal{X}$. We generate for $n=m=10,50,100,250$ two samples $\X_n$ and $\X_n'$ of $\muX$ and calculate for $\beta=0.01$ the statistic $\frac{n}{2}\dodstatnn$. For each $n$, we repeat this process 10.000 times. The finite sample distribution is then compared to a Monte Carlo sample of its theoretical limit distribution (sample size 10.000). Kernel density estimators (Gaussian kernel with bandwidth given by Silverman's rule) and Q-Q-plots are displayed in \Cref{fig: hypothesis visiualization}.
%We demonstrate the results in \Cref{fig: hypothesis visiualization}. The kernel density estimators (Gaussian kernel with bandwidth given by Silverman's rule) displayed in \Cref{fig: hypothesis visiualization}
All plots highlight that the finite sample distribution of $\dodstatnn$ is already well approximated by its theoretical limit distribution for moderate sample sizes.
% While the kernel density estimator of the sample of $\dodstatnn$ (shown in blue) for $n=10$ is remarkably flatter and wider than the one based on a sample of the limit distribution (displayed in red), both general shapes already resemble each other. For growing $n$ the kernel density estimator become more and more similar until for $n=250$ there is hardly any difference between them. The same observation can be made regarding the Q-Q-plots in \Cref{fig: hypothesis visiualization}.
% 
Moreover, for $n=10$ the quantiles of the finite sample distribution of $\dodstatnn$ are in general larger than the ones of the sample of its theoretical limit distribution, which suggests that the DoD-test will be rather conservative for small $n$. For $n\geq 50$ most quantiles of the finite sample distribution of $\dodstatnn$ match the ones of its theoretical limit distribution reasonably well.
\begin{figure}
	\includegraphics[width = \textwidth,height=0.5\textwidth]{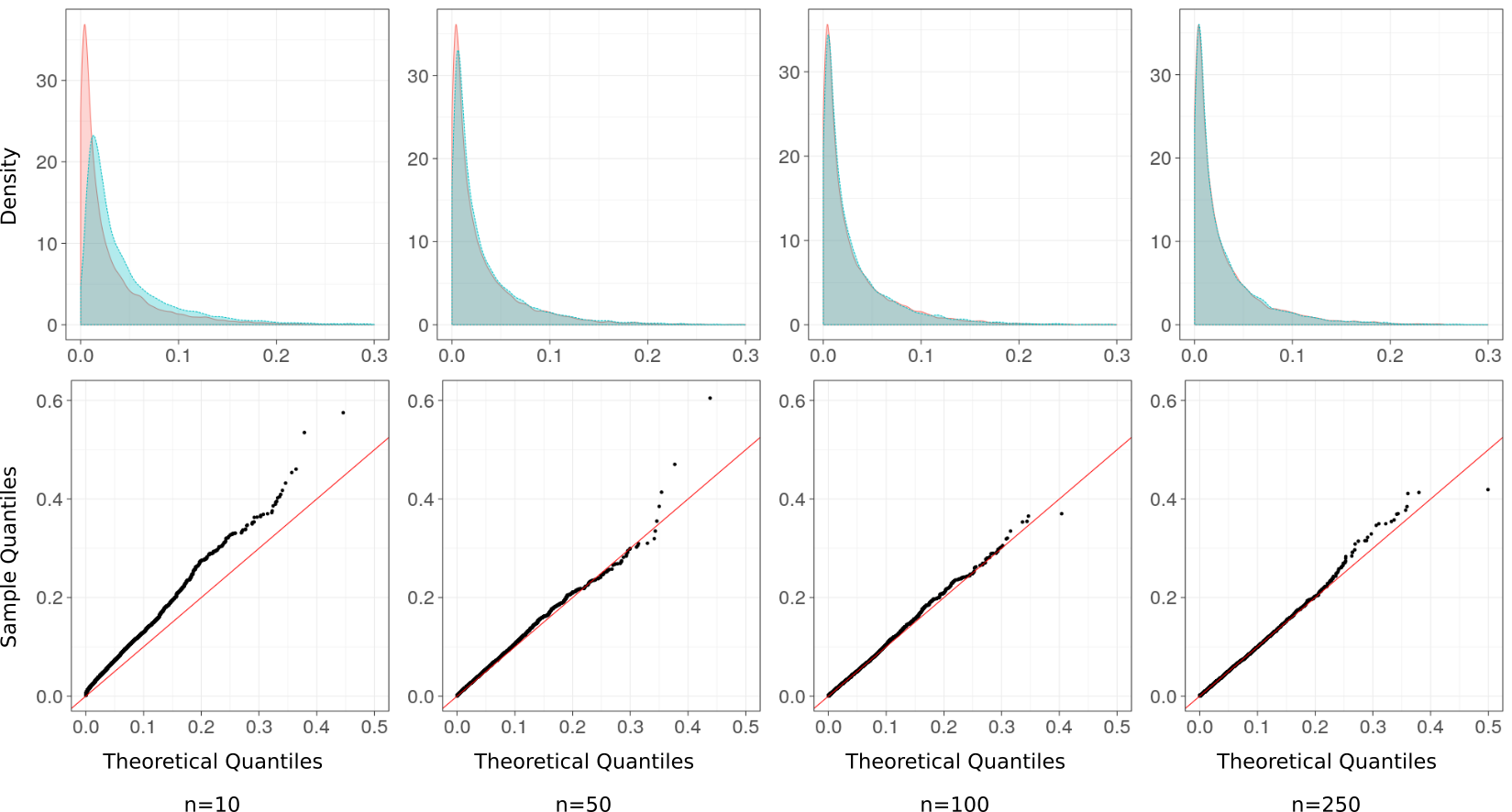}
	\caption{\textbf{Finite sample accuracy of the limit law under the hypothesis:} 
%		Comparison of the behavior of $\dodstatnn(\X_n,\X_n')$ (10.000 realizations), which is based on two different samples of $\mmspaceX$ (unit square in $\R^2$ with uniform distribution and supremum norm distance) to its theoretical limit distribution under $H_0$ for different $n$.
	Upper row: Kernel density estimators of the sample of $\dodstatnn$ (in blue) and a Monte Carlo sample of its theoretical limit distribution (in red, sample size 10.000) for $n=10,50,100,250$ (from left to right).
Lower row: The corresponding Q-Q-plots.}
	\label{fig: hypothesis visiualization}
\end{figure}
\subsection{Alternative}
Next we investigate the behavior of the statistic $\dodstatnn$ under the alternative. To this end, we consider the metric measure spaces $\mmspaceX$ and $\mmspaceY$, where $\mmspaceX$ denotes the one as defined in \Cref{subsec:hypothesis simulation} and $\mmspaceY$ the one, where $\mathcal{Y}$ is a disc in $\R^2$ with radius 0.5, $\dY$ the distance induced by the supremum norm and $\muY$ the uniform distribution on $\mathcal{Y}$. From a testing point of view it is more interesting to compare the finite sample distribution under the alternative to the limit distribution under $H_0$ (the considered metric measure spaces are isomorphic) than to investigate the speed of convergence to the limit derived in \Cref{thm:main}. Thus, we repeat the course of action of \Cref{subsec:hypothesis simulation} for $n=m=10,50,100,250$ and $\beta=0.01$ with samples $\X_n$ and $\Y_n$ from $\muX$ and $\muY$, respectively.\\
In order to highlight the different behavior of $\dodstatnn(\X_n,\Y_n)$ in this setting, we compare its finite sample distributions to the theoretical limit distribution under the hypothesis, which has already been considered in \Cref{subsec:hypothesis simulation}.\\
The results are visualized as kernel density estimators (Gaussian kernel with bandwidth given by Silverman's rule) and Q-Q-plots in \Cref{fig: alternative visiualization}.
% The figure illustrates that for $n=10$ the kernel density estimator based on the realizations of $\dodstatnn(\X_n,\Y_n)$ (in blue) still resembles the kernel density estimator based on the theoretical sample (in red). This changes as $n$ grows. 
As $n$ grows, the kernel density estimator based on the realizations of $\dodstatnn(\X_n,\Y_n)$ shifts to the right and becomes less and less concentrated. Furthermore, it becomes more and more symmetric around its peak which matches its theoretical Gaussian limit behavior (recall \Cref{thm:main}). For $n \geq100$ we see in \Cref{fig: alternative visiualization} that the densities based on the realizations of $\dodstatnn(\X_n,\Y_n)$ differ drastically from the ones based on the Monte Carlo samples of the theoretical limit distribution under $H_0$. The corresponding Q-Q-plots underline this observation and highlight that for $n\geq50$ essentially all quantiles of the sample are drastically larger than the ones of the theoretical limit distribution. This suggests that the proposed test discriminates between these metric measure spaces with high probability already for moderate values of $n$.
\begin{figure}
	\includegraphics[width = \textwidth,height=0.45\textwidth]{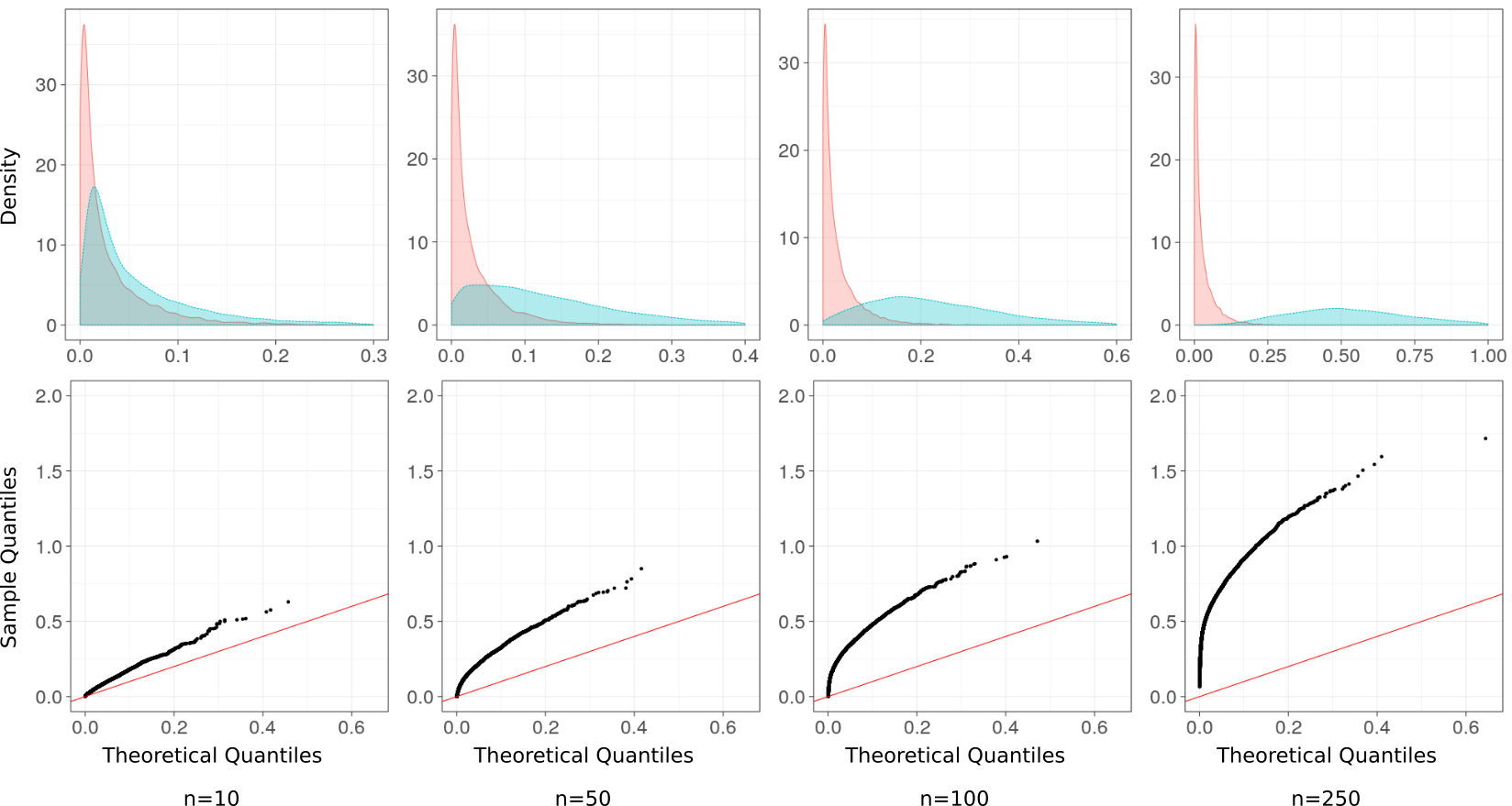}
	\caption{\textbf{The behavior of \boldmath$\dodstatnn$ under the alternative:} 
%		Comparison of the behavior of $\dodstatnn(\X_n,\Y_n)$ (10.000 realizations), which is based on samples of $\mmspaceX$ (unit square in $\R^2$ with uniform distribution and supremum norm distance) and $\mmspaceY$ (disc with radius 0.5 in $\R^2$ with uniform distribution and supremum norm distance) to their theoretical limit distribution under $H_0$ for different $n$.
%		
Upper Row: Kernel density estimators based on the Monte Carlo sample of the theoretical limit distribution under $H_0$ (red, sample size 10.000) and the realizations of $\dodstatnn(\X_n,\Y_n)$ (blue) for $n=10,50,100,250$ (from left to right). Lower row: The corresponding Q-Q-plots.}
	\label{fig: alternative visiualization}
\end{figure}
\subsection{The Bootstrap Test}\label{subsec:bootstrap test}
We now investigate the finite sample properties of the bootstrap test $\Phi_{DoD}^*$ (defined in \eqref{eq:bootstraptest}). Therefore, we compare the metric measure space $\mmspaceV$, where $\mathcal{V}$ is the unit square, $\dV$ is the Euclidean distance and $\muV$ the uniform distribution on $\mathcal{V}$, with the spaces $\left\lbrace\left(\mathcal {W}_i,d_{\mathcal {W}_i},\mu_{\mathcal {W}_i} \right) \right\rbrace_{i=1}^5$. Here, $\mathcal {W}_i$ denotes the intersection of the unit square with a disc of radius $r_i\in\left\lbrace \sqrt{2}/2,0.65,0.6,0.55,0.5\right\rbrace$ both centered at $(0,0)$, $d_{\mathcal {W}_i}$ the Euclidean distance and $\mu_{\mathcal {W}_i}$ the uniform distribution on $\mathcal {W}_i$. In \Cref{fig:mmspacevisiualization} the sets $\mathcal{V}$ (white) and $\left\lbrace\mathcal {W}_i \right\rbrace_{i=1}^5$ (red) are displayed. It highlights the increasing similarity of the sets for growing $r_i$.\\
	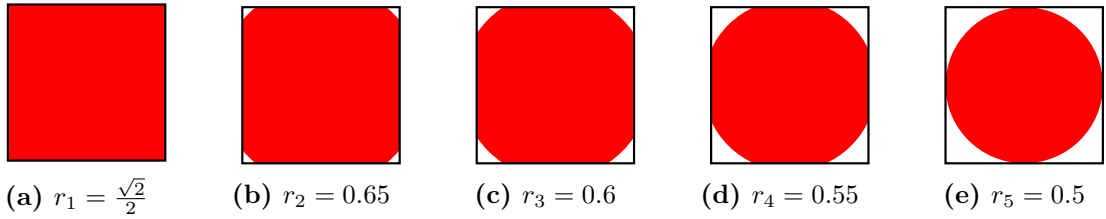
\begin{figure}
	\begin{subfigure}[c]{0.19\textwidth}
		\centering\resizebox{0.8\textwidth}{!}{
			\begin{tikzpicture}
			\begin{scope}
			\clip (-0.5,-0.5) rectangle (0.5,0.5);
			\fill[red] (0,0) circle (0.7);
			\end{scope}
			\draw (-0.5,-0.5) rectangle (0.5,0.5);
			\end{tikzpicture}}
		\caption{$r_1=\frac{\sqrt{2}}{2}$}
	\end{subfigure}
	\begin{subfigure}[c]{0.19\textwidth}
		\centering\resizebox{0.8\textwidth}{!}{
			\begin{tikzpicture}
			\begin{scope}
			\clip (-0.5,-0.5) rectangle (0.5,0.5);
			\fill[red] (0,0) circle (0.65);
			\end{scope}
			\draw (-0.5,-0.5) rectangle (0.5,0.5);
			\end{tikzpicture}}
		\caption{$r_2=0.65$}
	\end{subfigure}
	\begin{subfigure}[8]{0.19\textwidth}
		\centering
		\resizebox{0.8\textwidth}{!}{
			\begin{tikzpicture}
			\begin{scope}
			\clip (-0.5,-0.5) rectangle (0.5,0.5);
			\fill[red] (0,0) circle (0.6);
			\end{scope}
			\draw (-0.5,-0.5) rectangle (0.5,0.5);
			\end{tikzpicture}}
		\caption{$r_3=0.6$}
	\end{subfigure}
	\begin{subfigure}[c]{0.19\textwidth}
		\centering
		\resizebox{0.8\textwidth}{!}{
			\begin{tikzpicture}
			\begin{scope}
			\clip (-0.5,-0.5) rectangle (0.5,0.5);
			\fill[red] (0,0) circle (0.55);
			\end{scope}
			\draw (-0.5,-0.5) rectangle (0.5,0.5);
			\end{tikzpicture}}
		\caption{$r_4=0.55$}
	\end{subfigure}
	\begin{subfigure}[c]{0.19\textwidth}
		\centering
		\resizebox{0.8\textwidth}{!}{
			\begin{tikzpicture}
			\begin{scope}
			\clip (-0.5,-0.5) rectangle (0.5,0.5);
			\fill[red] (0,0) circle (0.5);
			\end{scope}
			\draw (-0.5,-0.5) rectangle (0.5,0.5);
			\end{tikzpicture}}
		\caption{$r_5=0.5$}
	\end{subfigure}
	\caption{\textbf{Different metric measure spaces:} Comparisons of the metric measure space $\mmspaceV$ (white)
%		 ($\V$ the unit square, $\dV$ the Euclidean distance, $\muV$ the uniform distribution on $\V$, displayed in white)
		  to the spaces $\left\lbrace\left( \mathcal {W}_i,d_{\mathcal{W}_i},\mu_{\mathcal{W}_i}\right) \right\rbrace_{i=1}^5$ (red).	  
%		   ($\mathcal {W}_i$ the intersection of the unit square with a disc of radius $r_i\in\left\lbrace \sqrt{2}/2,0.65,0.6,0.55,0.5\right\rbrace$ both centered at $(0,0)$, $d_{\mathcal {W}_i}$ the Euclidean distance and $\mu_{\mathcal {W}_i}$ the uniform distribution on $\mathcal {W}_i$, displayed in red).
	  }\label{fig:mmspacevisiualization}
\end{figure}Before we employ the bootstrap DoD-test with $\beta = 0.01$ in the present setting, we consider the bootstrap approximation proposed in \Cref{sec:bootsrapping the quantiles} in this simple setting. Therefore, we generate $n=10,50,100,250$ realizations of $\mu_\mathcal{V}$ and calculate for $n_B=n$ based on these samples 1000 times
\begin{equation*}
{\Xi}_{n_B}^*=\int_{0.01}^{0.99}\!\left(\left( {\mathbb{U}}_{n_B}^*\right) ^{-1}(t)\right)^2 \,dt\end{equation*}
as described in \Cref{sec:bootsrapping the quantiles}. We then compare for the different $n$ the obtained finite sample distributions to ones of $\dodstatnn(\V_n,\mathcal{W}_{1,n})$ (generated as described in \Cref{subsec:hypothesis simulation}).
\begin{figure}
	\includegraphics[width = \textwidth,height=0.45\textwidth]{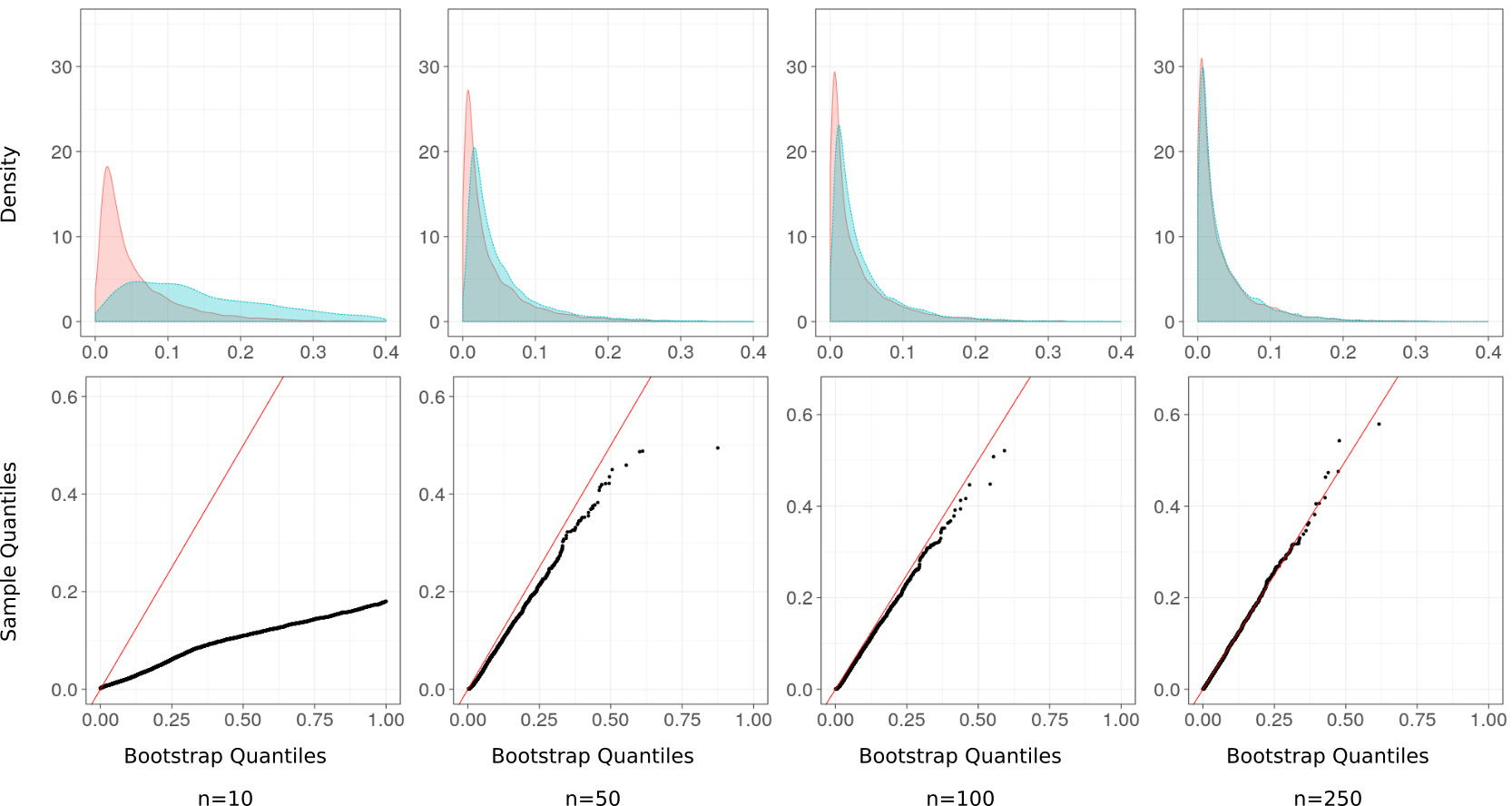}
	\caption{\textbf{Bootstrap under the hypothesis:} Illustration of the $n$ out of $n$ plug-in bootstrap approximation for the statistic $\dodstatnn$ based on samples from $\mmspaceV$ and $\left( \mathcal {W}_1,d_{\mathcal{W}_1},\mu_{\mathcal{W}_1}\right)$. Upper Row: Kernel density estimators of 1000 realizations of $\dodstatnn$ (in red) and its bootstrap approximation (blue, 1000 replications) for $n=10,50,100,250$ (from left to right). Lower row: The corresponding Q-Q-plots. }
	\label{fig: bootstrap visiualization}
\end{figure}
 The results are summarized as kernel density estimators (Gaussian kernel with bandwidth given by Silverman's rule) and Q-Q-plots in \Cref{fig: bootstrap visiualization}. Both, the kernel density estimators and the Q-Q-plots show that for $n\leq50$ the bootstrap quantiles are clearly larger than the empirical quantiles leading to a rather conservative procedure for smaller $n$, an effect that disappears for large $n$.
 
\noindent Next, we aim to apply $\Phi_{DoD}^*$ for $\beta=0.01$ at $5\%$-significance level for discriminating between $\mmspaceV$ and each of the spaces $\left( \mathcal {W}_i,d_{\mathcal{W}_i},\mu_{\mathcal{W}_i}\right)$, $i=1,\dots,5$. To this end, we bootstrap the quantile $\xi_{0.95}$ based on samples from $\muV$ as described in \Cref{sec:bootsrapping the quantiles} ($R=1000$) and then we apply the test $\Phi_{DoD}^*$, defined in \eqref{eq:bootstraptest}, with the bootstrapped quantile $\xi_{n_B,\alpha}^{(R)}$ on 1000 samples of size $n=10, 50, 100,250,500,$ $1000$ as illustrated in \Cref{sec:bootsrapping the quantiles}. The results are summarized in \Cref{tab:res}. In accordance to the previous simulations, we find that the prespecified significance level (for $r_1=\sqrt{2}/2$ the sets are equal) is approximated well for $n\geq 100$. Concerning the power of the test we observe that it is conservative for small $n$, but already for $n\geq 100$ the cases (d) and (e) (see \Cref{fig:mmspacevisiualization}) are detected reasonably well. If we choose $n=1000$, even the spaces in (c) are distinguishable, although in this case, $\mathcal {W}_3$ fills out about $94\%$ of $\mathcal{V}$. For (b), i.e.,  $r_2=0.65$, where more than $98\%$ of $\mathcal{V}$ is covered by $\mathcal {W}_4$, the power of the test falls below $0.11$. 
\begin{table}
	\centering
\begin{tabular}{|c|c|c|c|c|c|}
	\hline\multicolumn{6}{|c|}{$\Phi_{DoD}^*$\rule{0pt}{10pt}} \\
	\hline\rule{0pt}{10pt}
	Sample Size & $r_1={\sqrt{2}}/{2}$ &$r_2=0.65$& $r_3=0.6$ & $r_4=0.55$ &  $r_5=0.5$  \\\hline
	10  & 0.010&0.018&0.005&0.006&0.009    \\
	50  & 0.031&0.034&0.041&0.139&0.406      \\
	100 &0.048& 0.048&0.098&0.323&0.824 \\
	250 &0.038& 0.058&0.203&0.722&1.000 \\
	500 &0.045& 0.080&0.402&0.962&1.000 \\
	1000&0.051& 0.108&0.713&1.000&1.000 \\\hline
\end{tabular}
\medskip
\caption{\textbf{Comparison of different metric measure spaces I:} The empirical power of the DoD-test $\Phi_{DoD}^*$ (1000 replications) for the comparison of the metric measure spaces represented in \Cref{fig:mmspacevisiualization} for different $n$.}\label{tab:res}
\end{table}

\noindent In order to highlight how much power we gain in the finite sample setting by carefully handling the occurring dependencies we repeat the above comparisons, but calculate $\dodstat$ only based on the independent distances, i.e., on $\lbrace \dX(X_1,X_2),\dX(X_3,X_4),\dots,\dX(X_{n-1},X_n)\rbrace $ and $\{ \dY(Y_1,Y_2),\dY(Y_3,Y_4),\dots,\dY(Y_{m-1},Y_m) \} $, instead of all available distances. In the following, the corresponding statistic is denoted as $\widehat{D}_{\beta,ind}$. From the existing theory on testing with the empirical (trimmed) Wasserstein distance \cite{MunkNonparametricvalidationsimilar1998,BarrioTestsgoodnessfit1999,BarrioAsymptoticsL2functionals2005} it is immediately clear, how to construct an asymptotic level $\alpha$ test $\Phi_{{D}_{ind}}$ based on $\widehat{D}_{\beta,ind}$. The results for comparing $\mmspaceV$ and $\left\lbrace\left(\mathcal {W}_i,d_{\mathcal {W}_i},\mu_{\mathcal {W}_i} \right) \right\rbrace_{i=1}^5$ using $\Phi_{{D}_{ind}}$ with $\beta = 0.01$ are displayed in \Cref{tab:comparison indep}. Apparently, $\Phi_{{D}_{ind}}$ keeps its prespecified significance level of $\alpha=0.05$, but develops significantly less power than $\Phi_{DoD}^*$ in the finite sample setting.

\begin{table}
	\centering
	\begin{tabular}{|c|c|c|c|c|c|}
		 \hline\multicolumn{6}{|c|}{$\Phi_{{D}_{ind}}$\rule{0pt}{10pt}} \\
		\hline 	\rule{0pt}{10pt} Sample Size & $r_1={\sqrt{2}}/{2}$ &$r_2=0.65$& $r_3=0.6$ & $r_4=0.55$&$r_5=0.5$  \\\hline
		
		100  &0.036  & 0.040    &0.048 	& 0.050	  & 0.177 \\
		250  & 0.041 &  0.043   & 0.043	& 0.179   & 0.799   \\
		500  & 0.051 &   0.045  &0.084	& 0.583   & 0.998 \\
		1000 & 0.043 &   0.044  &0.231	& 0.974   & 1 \\\hline 
	\end{tabular}
	\medskip
	\caption{\textbf{Comparison of different metric measure spaces III:} The empirical power of the test based on $\widehat{D}_{\beta,ind}$ (1000 applications) for the comparison of the metric measure spaces represented in \Cref{fig:mmspacevisiualization}.} \label{tab:comparison indep}
\end{table}

\noindent Furthermore, we investigate the influence of $\beta$ on our results. To this end, we repeat the previous comparisons with $n=250$  and $\beta=0,0.01,0.05,0.25$. The results of the corresponding comparisons are displayed in \Cref{tab:resfordiffbeta}. It highlights that the test $\Phi_{DoD}^*$ holds its level for all $\beta$. Furthermore, we observe a decrease in power with increasing $\beta$, i.e., increasing degree of trimming. This is due to the fact that excluding too many large distances will no longer show small differences in the diameter.
\begin{table}\centering
	\begin{tabular}{|c|c|c|c|c|c|}
			\hline\multicolumn{6}{|c|}{$\Phi_{DoD}^*$\rule{0pt}{10pt}} \\
		\hline\rule{0pt}{10pt}
		$\beta$ & $r_1={\sqrt{2}}/{2}$ &$r_2=0.65$& $r_3=0.6$ & $r_4=0.55$ &  $r_5=0.5$  \\\hline
		0 & 0.048   &   0.058   &   0.228   & 0.776 & 0.997 \\
		0.01  &  0.051   &   0.065   &  0.232   &    0.736   &     0.995   \\
		0.05 & 0.049 & 0.059& 0.189 & 0.676 & 0.998 \\
		0.25 & 0.045 & 0.061 & 0.156 & 0.579&  0.979\\\hline
	\end{tabular}
	\medskip
	\caption{\textbf{The influence of $\beta$:} The empirical power of the DoD-test $\Phi_{DoD}^*$ (1000 replications) for the comparison of the metric measure spaces represented in \Cref{fig:mmspacevisiualization} for different $\beta$.}\label{tab:resfordiffbeta}
\end{table}

\noindent To conclude this subsection, we remark that in the above simulations the quantiles required for the applications of $\Phi_{DoD}^*$ were always estimated based on samples of $\muV$. Evidently, this slightly affects the results obtained, but we found that this influence is not significant.
\section{Structural Protein Comparisons}\label{sec:application}
Next, we apply the DoD-test to compare the protein structures displayed in \Cref{fig: protein visiualization}. First, we compare 5D0U with itself, in order to investigate the actual significance level of the proposed test under $H_0$ in a realistic example. Afterwards, 5D0U is compared with 5JPT and with 6FAA, respectively. However, before we can apply $\Phi_{DoD}^*$, we need to model proteins as metric measure spaces. Thus, we briefly recap some well known facts about proteins to motivate the subsequent approach. A protein is a polypeptide chain made up of amino acid residues linked together in a definite sequence. Tracing the repeated amide, $C^\alpha$ and carbonyl atoms of each amino acid residue, a so called \textit{backbone} can be identified. It is well established that the distances between the $C^\alpha$ {atoms} of the backbone contain most of the information about the protein's structure \cite{rossman1974recognition,kuntz1975approach,jones1986using,holm1993protein}. For the following comparisons, we randomly select $n = 10,50,100,250,500$ from the 650-750 $C^\alpha$ atoms of the respective proteins and assume that the corresponding coordinates are samples of unknown distributions $\left\lbrace \mu_{\X_i}\right\rbrace_{i=1}^3 $ supported on Borel sets $\mathcal{X}_i\subset\R^3$ equipped with the Euclidean distance. Furthermore, we choose $\beta=0.01$, $\alpha=0.05$ and determine for each $n$ the bootstrap quantile  $\xi_{n_B,0.95}^{(R)}$ based on a sample of size $n$ from 5D0U ($R=1000$, $n_B=n$) as illustrated in \Cref{sec:bootsrapping the quantiles}. This allows us to directly apply the test ${\Phi}_{DoD}^*$ on the drawn samples.
 
%\begin{figure}
%	\begin{subfigure}[c]{0.49\textwidth}
%	\includegraphics[height=\textwidth,trim=.01cm .01cm .01cm .01cm, clip]{Graphics/5dou.png}
%	\caption{}
%		\end{subfigure}
%	\begin{subfigure}[c]{0.49\textwidth}
%	\includegraphics[ height=\textwidth,trim=.01cm .01cm .01cm .01cm, clip]{Graphics/5d0uvs5jpt.png}
%	\caption{}
%		\end{subfigure}\\
%	\begin{subfigure}[c]{0.49\textwidth}
%	\includegraphics[ height=\textwidth,trim=.01cm .01cm .01cm .01cm, clip]{Graphics/5d0uvs6faa.png}
%	\caption{}
%	\end{subfigure}
%\caption{\textbf{Illustration of the proteins to be compared:} Cartoon representation of the DEAH-box RNA-helicase Prp43 from chaetomium thermophilum bound to ADP (PDB ID: 5D0U \cite{tauchert2016structural}, Figure (a)) in alignment with Prp43 from saccharomyces cerevisiae in complex with CDP (PDB ID: 5JPT \cite{robert2016functional}, Figure (b)) and in alignment with the DEAH-box RNA helicase Prp2 in complex with ADP (PDB ID: 6FAA \cite{schmitt2018crystal}, Figure (c)).}
%\label{fig: protein visiualization}
%\end{figure}
\noindent The results of our comparisons are summarized in \Cref{fig:protein results}. It displays the empirical significance level resp. the empirical power of the proposed method as a function of $n$.

\noindent\textbf{5D0U vs 5D0U:} In accordance with the previous simulation study this comparison (see \Cref{fig:protein results}, left) shows that $\Phi_{DoD}^*$ is conservative in this application as well. 
	
\noindent\textbf{5D0U vs 5JPT:} We have already mentioned in \Cref{subsec:applications} that 5D0U and 5JPT are structures of the same protein extracted from two different organisms and thus highly similar (their alignment has a root mean deviation of less than 0.59 \AA). The empirical power for this comparison (\Cref{fig:protein results}, middle) stays for all $n$ below $\alpha = 0.05$ and thus the test does not discriminate between these two protein structures in accordance with our biological knowledge. 

\noindent\textbf{5D0U vs 6FAA:} Although the protein structures 5D0U and 6FAA are similar at large parts (their alignment has a root mean square deviation of $0.75$ \AA), the DoD-test is able to discriminate between them with high statistical power. The empirical power (\Cref{fig:protein results}, right) is a strictly monotonically increasing function in $n$ that is greater than 0.63 for $n\geq100$ and approaches 1 for $n=500$ (recall that we use random samples of the $650-750$ $C^\alpha$ atoms).

\begin{figure}
	\centering
	\begin{subfigure}[c]{0.32\textwidth}
		\centering
		\includegraphics[width = \textwidth,height=\textwidth]{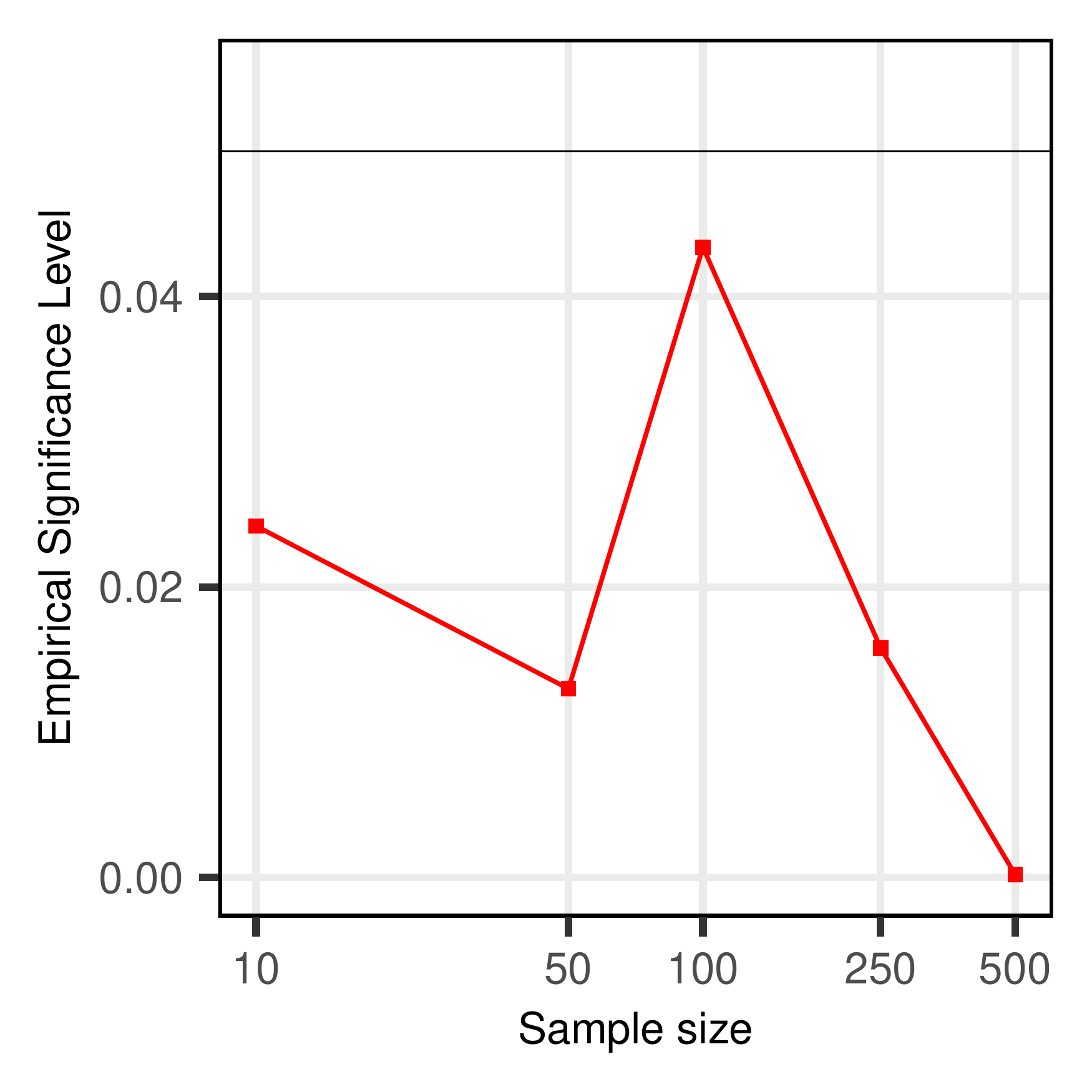}
	\end{subfigure}
	\begin{subfigure}[c]{0.32\textwidth}
				\centering
		\includegraphics[width =\textwidth,height=\textwidth]{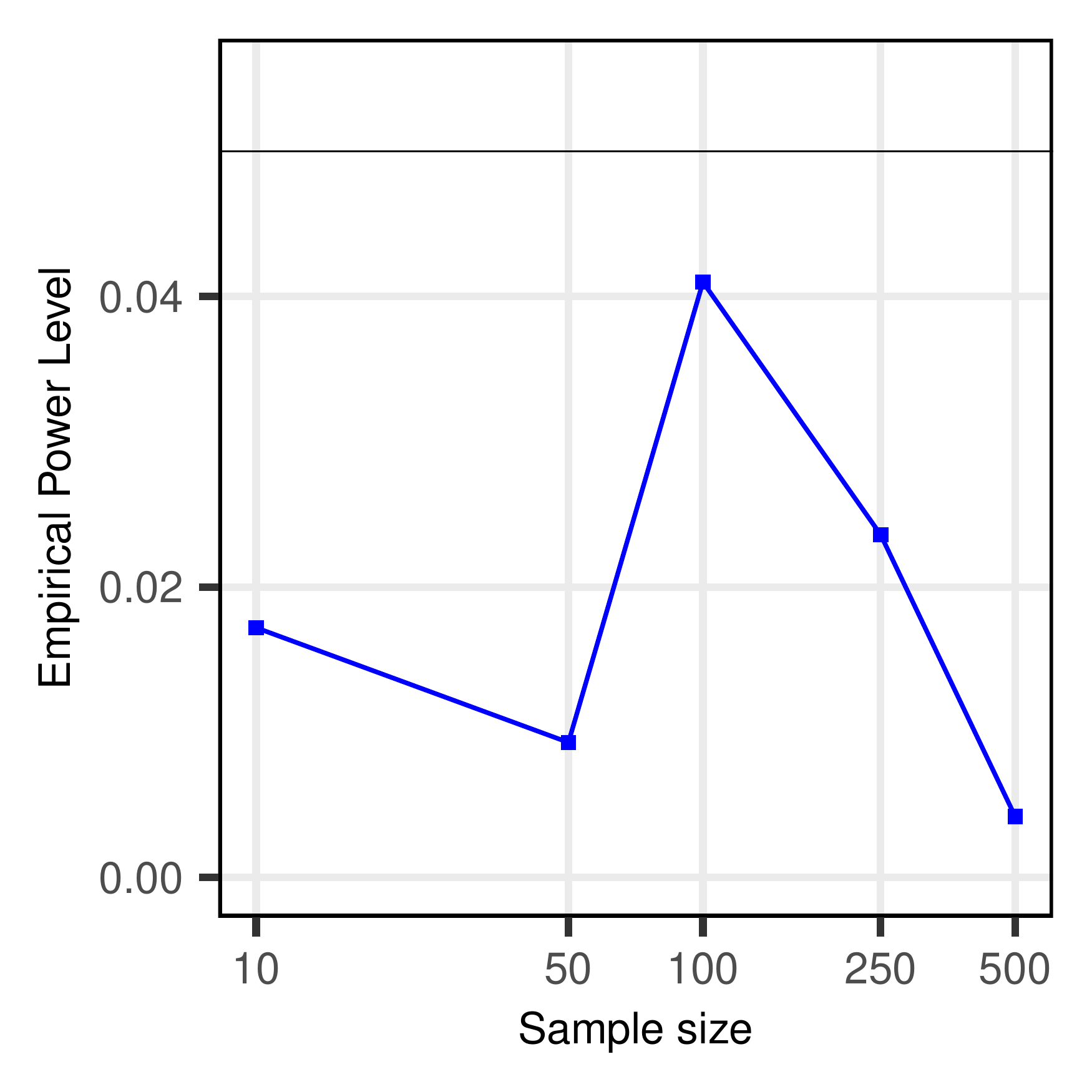}
	\end{subfigure}
\begin{subfigure}[c]{0.32\textwidth}
			\centering
			\includegraphics[width = \textwidth,height=\textwidth]{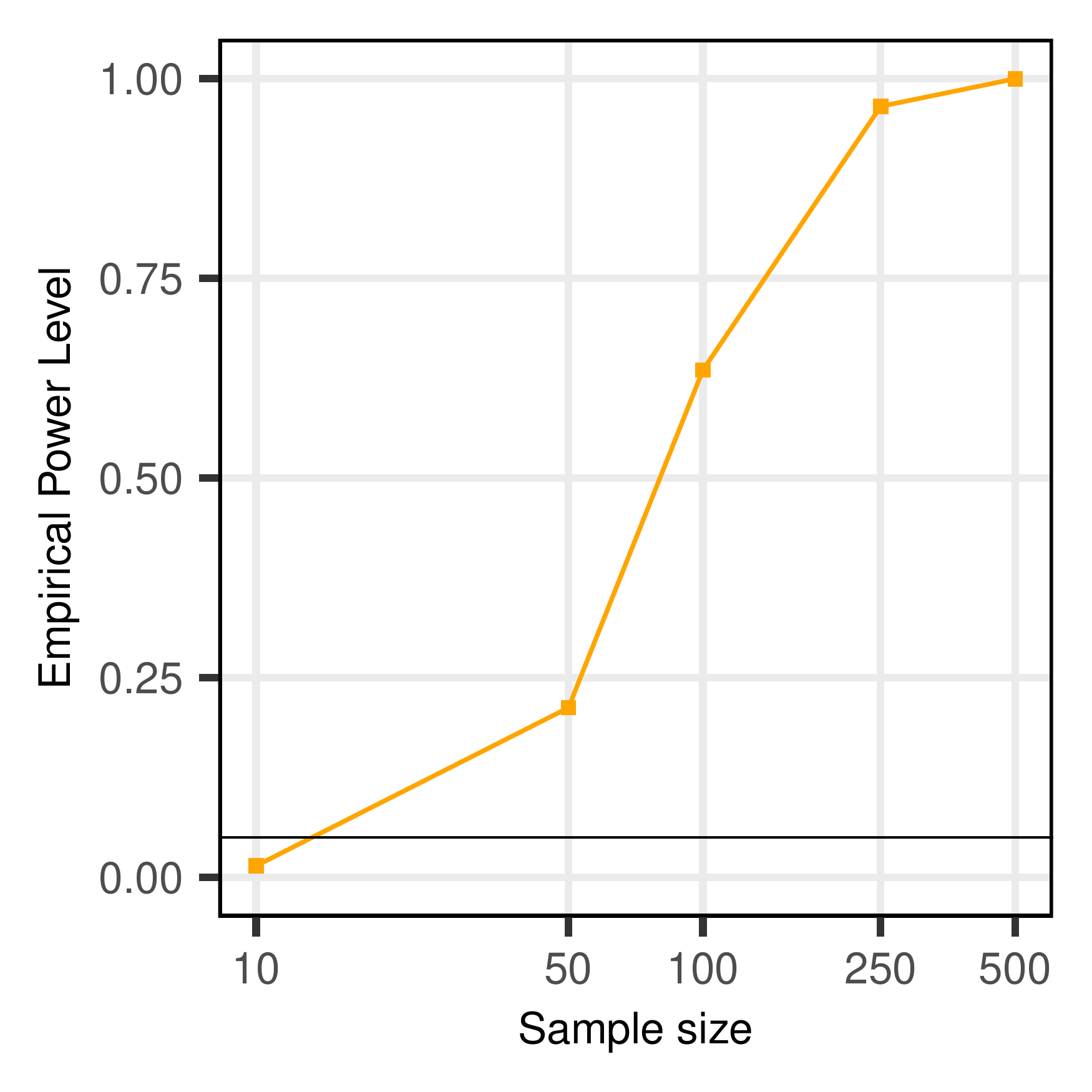}
		\end{subfigure}
%	\begin{subfigure}[c]{0.32\textwidth}
%	\includegraphics[width = \textwidth,height=\textwidth,trim=.01cm .01cm .01cm .01cm, clip]{Graphics/proteinhypdensityn100.pdf}
%	\caption{$n=100$}
%\end{subfigure}
%\begin{subfigure}[c]{0.32\textwidth}
%	\includegraphics[width = \textwidth, height=\textwidth,trim=.01cm .01cm .01cm .01cm, clip]{Graphics/proteinhypdensityn500.pdf}
%	\caption{$n=500$}
%\end{subfigure}
	\caption{\textbf{Protein Structure Comparison:} Empirical significance level for comparing 5D0U with itself (left), empirical power for the comparison of 5D0U with 5JPT (middle) as well as the the empirical power for comparing 5D0U with 6FAA (right). 1000 repetitions of the test ${\Phi}^*_{DoD}$ have been simulated for each $n$. 		
%		Second row: Kernel density estimators based on the samples of the test statistic for 5D0U against itself (in red), 5D0U against 5JPT (in blue) and the bootstrap sample (size 10.000) used to derive the required quantile (in green) for $n=100$ (Figure (c)) and $n=500$ (Figure (d)).
}
	\label{fig:protein results}
\end{figure} 
%Remarkably, although the proteins 5D0U and 6FAA are similar (their alignment has a root mean square deviation of $0.75$ \AA), the DoD-test is able to discriminate between them with high statistical power. The empirical power (\Cref{fig:protein results}, right) is a strictly monotonically increasing function in $n$ that is greater than 0.63 for $n\geq100$ and approaches 1 for $n=500$ (recall that we use samples of the $650-750$ $C^\alpha$ atoms).\\

\noindent Finally, we remark that throughout this section we have always based the quantiles required for testing on samples of the protein structure 5D0U. By the definition of $\Phi_{DoD}^*$ it is evident that this influences the results. If we compared the proteins 6FAA and 5D0U using $\Phi_{DoD}^*$ with quantiles obtained by a sample of 6FAA, the results would change slightly, but remain comparable. 
%\begin{figure}
%	\begin{subfigure}[c]{0.32\textwidth}
%		\includegraphics[width = \textwidth,height=\textwidth,trim=.01cm .01cm .01cm .01cm, clip]{Graphics/proteincompaltpower.pdf}
%		\caption{}
%	\end{subfigure}
%	\begin{subfigure}[c]{0.32\textwidth}
%		\includegraphics[width = \textwidth, height=\textwidth,trim=.01cm .01cm .01cm .01cm, clip]{Graphics/proteincompaltpvalues.pdf}
%		\caption{}
%	\end{subfigure}
%	\caption{\textbf{Protein Comparison II:} Illustration of the empirical power based on 10.000 repetitions of the DoD-test for comparing 5D0U and 6FAA (Figure (a)) and the mean of the corresponding $p$-values (Figure (b)) in dependence of $n$.}
%	\label{fig:protein results alt}
%\end{figure}
\subsection{Comparison to the DTM-test}\label{subsec:comparisontoBrecheteau}
In this section, we investigate how the test proposed by \citet{brecheteau2019statistical} compares to $\Phi_{DoD}^*$ for protein structure comparison. To put this method in our context, we briefly introduce the method proposed in the latter reference, comment on the underlying theoretical signature and compare the empirical power of the two tests in some simple scenarios.

\noindent We begin with the introduction of the \textit{empirical distance to measure signature}. Let $\mmspaceX$ be a metric measure space with $X_1,\dots,X_n\iid \muX$ and let $\X_n=\left\lbrace X_1,\dots,X_n\right\rbrace $. Then, the \textit{(empirical) distance to measure function} with mass parameter $\kappa={k}/{n}$ is given as
\[\delta_{\X_{n},\kappa}(x)\coloneqq\frac{1}{k}\sum_{i=1}^{k}\dX\left( X^{(i)},x\right),\]
where $X^{(i)}$ denotes the $i$'th nearest neighbor of $x$ in the sample $\X_n$ (for general $\kappa$ see \cite{brecheteau2019statistical}). For $n_S\leq n$ the \textit{empirical distance to measure signature} with mass parameter $\kappa={k}/{n}$ is then defined as
\begin{equation}\label{eq:definition distance to measure}D_{\X_{n},\kappa}\left({n_S}\right)\coloneqq\frac{1}{n_{S}}\sum_{i=1}^{n_S}\delta_{\X_{n},\kappa}(X_i),\end{equation}
which is a discrete probability distribution on $\R$. Let $\mmspaceY$ be a second metric measure space, $Y_1,\dots,Y_n\iid\muY$, let $\Y_n=\{Y_1,\dots,Y_n\}$ and let $D_{\Y_{n},\kappa}\left({n_S}\right)$ be defined analogously to \eqref{eq:definition distance to measure}. Then, given that $\frac{n_S}{n}=o(1)$, \citet{brecheteau2019statistical} constructs an asymptotic level $\alpha$ test for $H_0$ defined in \eqref{eq:nullhypothesis} based on the 1-Kantorovich distance between the respective empirical distance to measure signatures, i.e., on the test statistic
\begin{equation}\label{eq:def of Brestat}T_{n_S,\kappa}(\X_n,\Y_n)\coloneqq\mathcal{K}_1\left( D_{\X_{n},\kappa}\left({n_S}\right),D_{\Y_{n},\kappa}\left({n_S}\right)\right). \end{equation}
The corresponding test, that rejects if \eqref{eq:def of Brestat} exceeds a bootstrapped critical value $q^{DTM}_{\alpha}$, is in the following denoted as $\Phi_{DTM}$. The test statistic $T_{n_S,\kappa}$ is related to the Gromov-Kantorovich distance as follows (see \cite[Prop. 3.2]{brecheteau2019statistical})
\[T_\kappa(\X,\Y)\coloneqq\mathcal{K}_1(D_{\X,\kappa},D_{\Y,\kappa})\leq\frac{2}{\kappa}\mathcal{GK}_1 \left( \X,\Y\right). \]
Here, $D_{\X,\kappa}$ and $D_{\Y,\kappa}$ denote the true \textit{distance to measure signature} (see \cite[Sec. 1]{brecheteau2019statistical} or Section B.6 in the supplementary material \cite{supplement} for a formal definition) of $\mmspaceX$ and $\mmspaceY$, respectively.

\noindent The first step for the comparison of both methods is now to analyze how the respective signatures, the distribution of distances and the distance to measure signature, relate to each other. By the definition of $D_{\X,\kappa}$, which coincides for $n=n_S$ with \eqref{eq:definition distance to measure} for discrete metric measure spaces with $n$ points and the uniform measure, we see that $T_\kappa(\X,\Y)$ puts emphasis on local changes. This allows it to discriminate between the metric measure spaces in Figure 7 of \citet{Memoli2011} for $\kappa>1/4$, whereas $\ptruedodstat\left(\X,\Y \right)$ is always zero for this example. One the other hand, for $\kappa\leq 1/2$, $T_\kappa(\X,\Y)$ cannot distinguish between the metric measure spaces displayed in \Cref{fig:differnt metric measure spaces}, whereas $\ptruedodstat(\X,\Y)$ can become arbitrarily large, if the represented triangles are moved further apart. More precisely, $\ptruedodstat$ scales as the $p$'th power of the distance between the triangles.
\begin{figure}
	\centering
		\includegraphics[scale = 0.4]{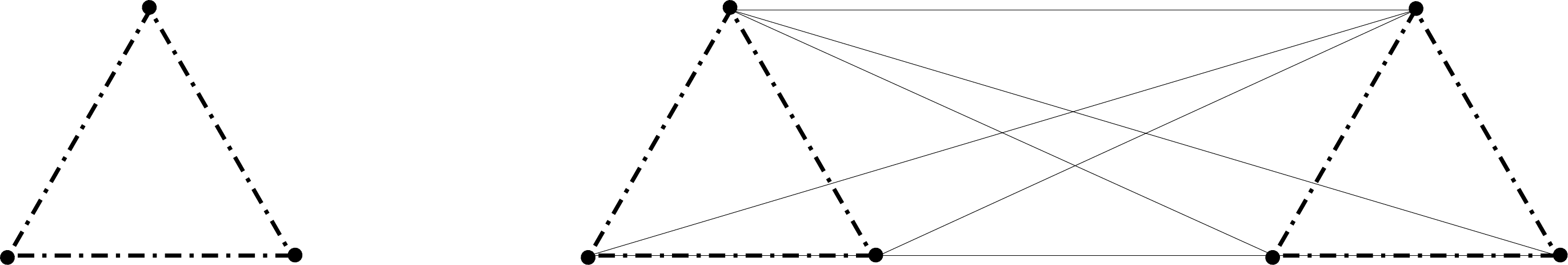}

	%	\begin{subfigure}[c]{0.32\textwidth}
	%	\includegraphics[width = \textwidth,height=\textwidth,trim=.01cm .01cm .01cm .01cm, clip]{Graphics/proteinhypdensityn100.pdf}
	%	\caption{$n=100$}
	%\end{subfigure}
	%\begin{subfigure}[c]{0.32\textwidth}
	%	\includegraphics[width = \textwidth, height=\textwidth,trim=.01cm .01cm .01cm .01cm, clip]{Graphics/proteinhypdensityn500.pdf}
	%	\caption{$n=500$}
	%\end{subfigure}
	\caption{\textbf{Different metric measure spaces:} Representation of two different, discrete metric measure spaces that are both equipped with the respective uniform distribution. Left: Three points with the same pairwise distances (dash-dotted lines). Right: Two translated copies that are further than one side length apart. 
	}
	\label{fig:differnt metric measure spaces}
\end{figure} Generally, one can show that $\ptruedodstat(\X,\Y)$ and $T_\kappa(\X,\Y)$ are related to different sequences of lower bounds for the Gromov-Kantorovich distance (see \cite[Sec. B.6]{supplement} for more details). 
%Let $\mathbf{SLB}_p$ and $\mathbf{TLB}_p$ be the lower bounds of the Gromov-Kantorovich function defined in \cite[Sec. 6]{Memoli2011}. Then, it follows that $\truedodstat_p$ is a lower bound for  $\mathbf{SLB}_p$ \cite[Cor. 6.2]{Memoli2011} and we show in \Cref*{sec:lower bounds} of the supplementary material that $T_\kappa$ is a lower bound for $\mathbf{TLB}_p$ (see \cite[Def. 6.3]{Memoli2011}).\\

 \noindent The next step of the comparison of both methods consists in comparing their performance for simulated examples. To this end, we repeat the comparisons of $\mmspaceV$ with the spaces $\left\lbrace\left(\mathcal {W}_i,d_{\mathcal {W}_i},\mu_{\mathcal {W}_i} \right) \right\rbrace_{i=1}^5$ as done in \Cref{subsec:bootstrap test}. Furthermore, we simulate the empirical power of $\Phi_{DoD}^*$ in the setting of Section 4.2 of \cite{brecheteau2019statistical}. For both comparisons, we choose a significance level of $\alpha=0.05$. We begin with applying $\Phi_{DTM}$ in the setting of \Cref{subsec:bootstrap test}. By the definition of the test statistic in \eqref{eq:def of Brestat} and the representation of the metric measure spaces in \Cref{fig:mmspacevisiualization}, it is clear that this setting is difficult for the method based on $T_{n_S,\kappa}$. Furthermore, we remark that the test $\Phi_{DTM}$ is not easily applied in the finite sample setting. Although it is an asymptotic test of level $\alpha$, the parameters $n_S$ and $\kappa$ have to be chosen carefully for the test to hold its prespecified significance level for finite samples. In particular, choosing $n_S$ and $\kappa$ large violates the independence assumption underlying the results of \cite{brecheteau2019statistical}.\\ Generally, we found that the choices $\kappa\leq0.1$ and $n_S\leq n/15$ yield reasonable results which we list in  \Cref{tab: comparison brecheteau}. The results show that the DTM-test holds its significance level (for $r_1$ both spaces are the same) and develops a significant amount of power for the cases $r_4$ and $r_5$. As to be expected it struggles to discriminate $\mmspaceV$ and $\left(\mathcal {W}_3,d_{\mathcal {W}_3},\mu_{\mathcal {W}_3} \right)$, as for this task especially the large distances are important. Here, $\Phi_{DoD}^*$ clearly outperforms $\Phi_{DTM}$.
 
 \begin{table}
 	\centering
 	\begin{tabular}{|c|c|c|c|c|c|}\hline
 		\multicolumn{6}{|c|}{$\kappa=0.05$, $\alpha=0.05$\rule{0pt}{10pt}} \\
 		\hline 	\rule{0pt}{10pt} Sample Size & $r_1={\sqrt{2}}/{2}$ &$r_2=0.65$& $r_3=0.6$ & $r_4=0.55$&$r_5=0.5$  \\\hline
 		
 		100  & 0.039 &     0.054	 &0.049 	&0.064& 0.096  \\
 		250  & 0.053 &  0.049   &0.070 	&0.144&  0.379   \\
 		500  & 0.065 &  0.076   &0.128	&0.361& 0.853  \\
 		1000 & 0.072 & 		0.088 &0.259	&0.740&  1\\\hline
 		\multicolumn{6}{|c|}{$\kappa=0.1$, $\alpha=0.05$\rule{0pt}{10pt}} \\
 		\hline\rule{0pt}{10pt}	Sample Size & $r_1={\sqrt{2}}/{2}$ &$r_2=0.65$& $r_3=0.6$ & $r_4=0.55$&$r_5=0.5$\\\hline
 		100 &  0.067&     0.047	 & 0.070&0.083&  0.170 \\
 		250  &   0.066  &   0.069  & 0.094&0.163& 0.533    \\
 		500  &  0.064 &   0.081  &0.147 &0.430&   0.926  \\
 		1000 & 0.064 & 	0.091	 &  0.255    &0.790&1  \\\hline
 	\end{tabular}
 	\medskip
 	\caption{\textbf{Comparison of different metric measure spaces II:} The empirical power of the test based on $T_{n_S,\kappa}$ (1000 applications) for the comparison of the metric measure spaces represented in \Cref{fig:mmspacevisiualization} for different $n$ and $\kappa$ ($n_S=n/15$).}\label{tab: comparison brecheteau}
 \end{table}
 
\noindent In a further example, we investigate how well $\Phi_{DoD}^*$ and $\Phi_{DTM}$ discriminate between different spiral types (see \Cref{fig:spirals}). These spirals are constructed as follows. Let $R\sim U[0,1]$ be uniformly distributed and independent of $S,S'\iid N(0,1)$. Choose a significance level of $\alpha = 0.05$ and let $\beta=0.01$. For $v=10,15,20,30,40,100$ we simulate samples of
\begin{equation}
	\label{eq:spiral creation} (R\sin(vR)+0.03S,R\cos(vR)+0.03S')\sim \mu_v
\end{equation}
and considers these to be samples from a metric measure spaces equipped with the Euclidean metric.
% As for all $v$ the distribution of large distances is approximately the same, the  For $v\in\{15,20,30,40,100\}$ the empirical significance level is reported for each $\mu_v$ and the empirical power for the comparison of $\mu_{10}$ with $\mu_v$. 
\begin{figure}
	\centering
	\begin{subfigure}[c]{0.32\textwidth}
		\centering
		\includegraphics[width = \textwidth,height=\textwidth]{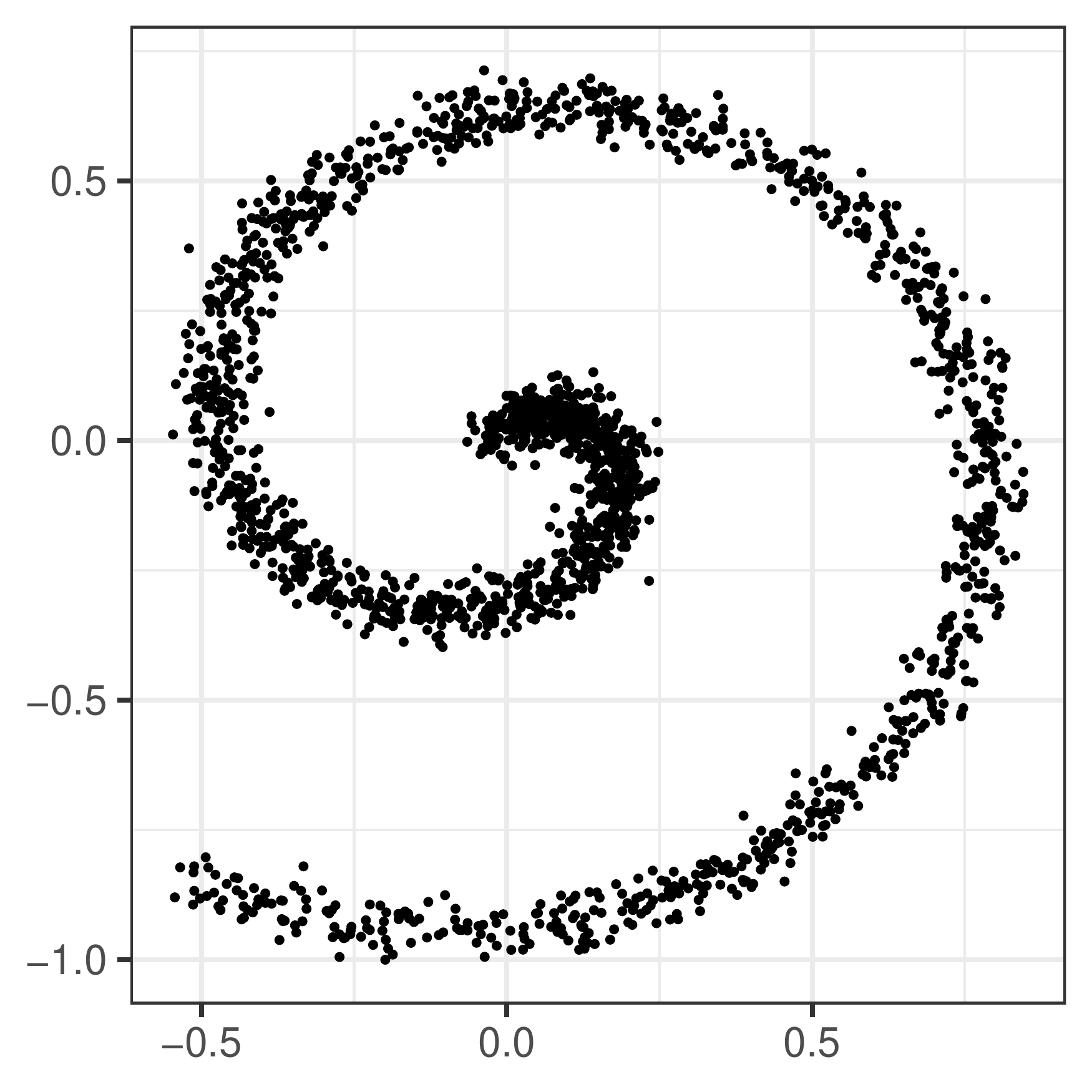}
	\end{subfigure}
	\begin{subfigure}[c]{0.32\textwidth}
	\centering
	\includegraphics[width = \textwidth,height=\textwidth]{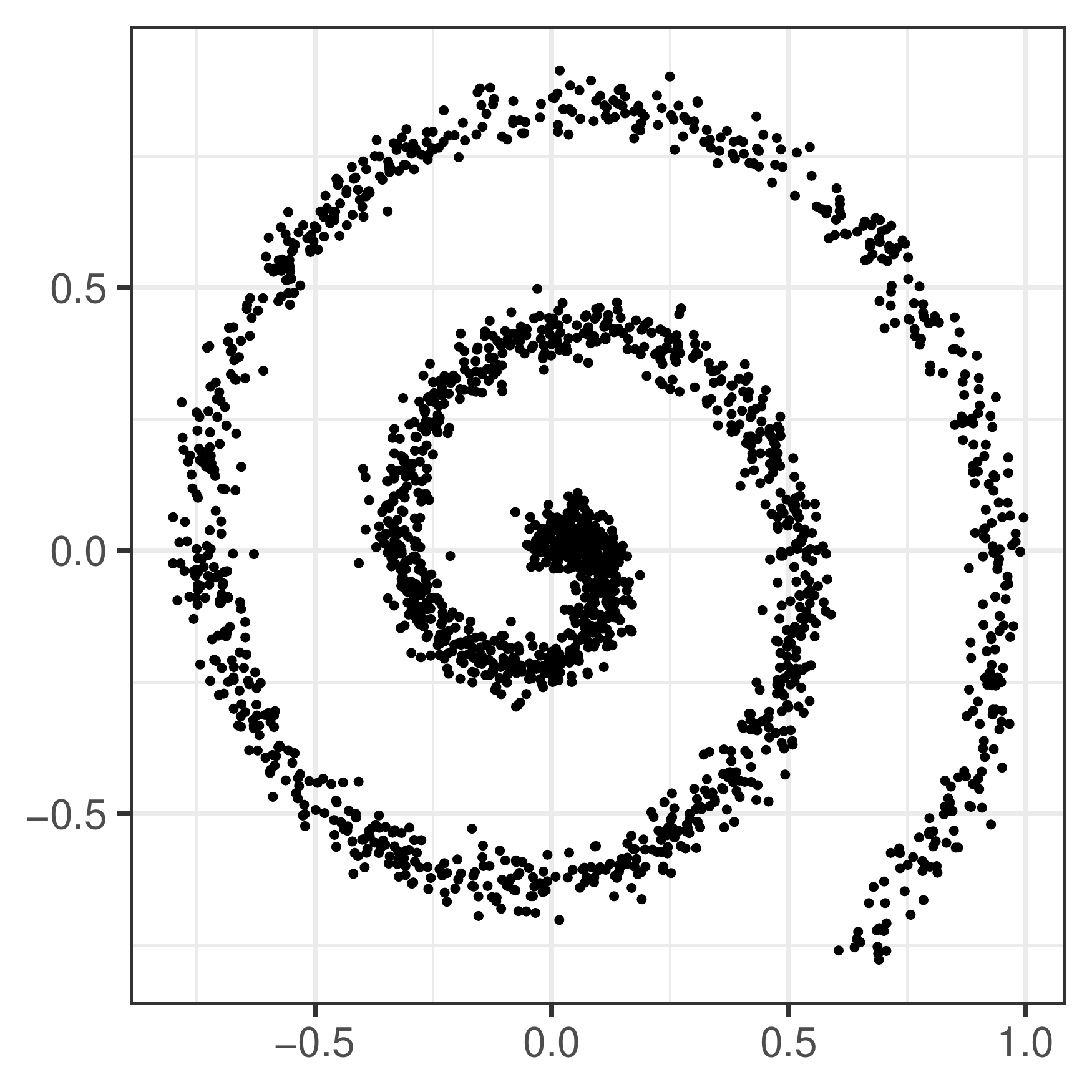}
\end{subfigure}
	\begin{subfigure}[c]{0.32\textwidth}
	\centering
	\includegraphics[width = \textwidth,height=\textwidth]{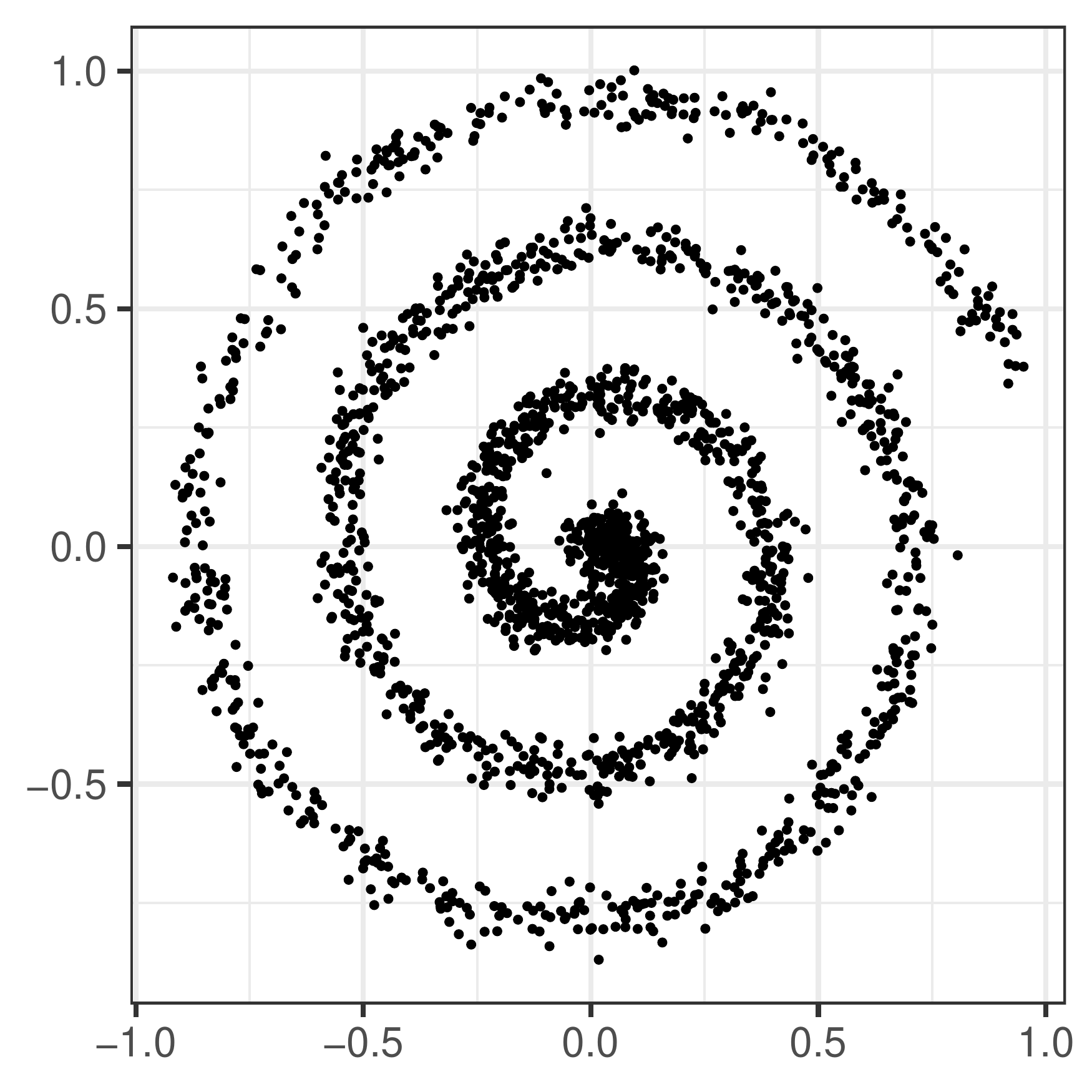}
\end{subfigure}
	%	\begin{subfigure}[c]{0.32\textwidth}
	%	\includegraphics[width = \textwidth,height=\textwidth,trim=.01cm .01cm .01cm .01cm, clip]{Graphics/proteinhypdensityn100.pdf}
	%	\caption{$n=100$}
	%\end{subfigure}
	%\begin{subfigure}[c]{0.32\textwidth}
	%	\includegraphics[width = \textwidth, height=\textwidth,trim=.01cm .01cm .01cm .01cm, clip]{Graphics/proteinhypdensityn500.pdf}
	%	\caption{$n=500$}
	%\end{subfigure}
	\caption{\textbf{Different metric measure spaces II:} Representation of samples created by \eqref{eq:spiral creation} for $v=10,15,20$ (from left to right).
	}
	\label{fig:spirals}
\end{figure}
We apply $\Phi_{DoD}^*$ with quantiles based on $\mu_v$ in order to compare $\mu_v$ with $\mu_v$ (based on different samples) and $\mu_v$ with $\mu_{10}$, $v=15,\dots,100$. The results presented in \Cref{tab:galaxy comparison} show that $\Phi_{DoD}^*$ holds its significance level and always discriminates between $\mu_{10}$ and $\mu_v$ in this setting. The reported values from \cite{brecheteau2019statistical} are also listed for the ease of readability. They show that the additional subsampling in the definition of \eqref{eq:def of Brestat} leads to a loss of power in this example. Moreover, for $\Phi_{DTM}$ the results always depend on the choice of the parameters. 
\begin{table}\centering	\begin{tabular}{|c|c|c|c|c|c|}\hline
		\multicolumn{6}{|c|}{$\Phi_{DoD}^*$\rule{0pt}{10pt}} \\
			\hline\rule{0pt}{10pt}
		$v$ & 15 &20& 30&40&100  \\\hline
		\rule{0pt}{10pt} Type-I error & 0.036   & 0.051   & 0.051&0.054&0.048    \\
		\rule{0pt}{10pt}  Emp. power & 1   & 1	& 1&1&1  \\\hline
		\multicolumn{6}{|c|}{$\Phi_{DTM}$\rule{0pt}{10pt}} \\
	\hline\rule{0pt}{10pt}
$v$ & 15 &20& 30&40&100  \\\hline
\rule{0pt}{10pt} Type-I error&  0.043&0.049&0.050&0.051&0.050\\
\rule{0pt}{10pt}  Emp. power &0.525&0.884&0.987&0.977&0.985\\\hline
\end{tabular}
	\medskip
	\caption{\textbf{Spiral Comparison:} Empirical significance level and power of  $\Phi_{DoD}^*$ and  $\Phi_{DTM}$ for the comparisons the metric measure spaces represented in \Cref{fig:spirals}.}\label{tab:galaxy comparison}
\end{table}

\noindent Finally, we come to the protein structure comparison. We repeat the previous comparisons of 5D0U, 5JPT and 6FAA for a significance level $\alpha=0.05$, $n=100,250,500$, $n_S=N/5$ and $\kappa=0.05,0.1$. The results are reported in \Cref{tab:protein comparison brecheteau}. We see that also $\Phi_{DTM}$ approximately holds its significance level and is more sensitive to small local changes such as slight shifts of structural elements for small $\kappa$.
%
% It sometimes discriminates between 5D0U and 5JPT. Thus, if one wants to detect differences between these two structures, which depends on the application (recall that 5D0U and 5JPT are structures of the same protein from different organisms), then the test by \citet{brecheteau2019statistical} is a good alternative. 
 However, the evident differences between 5D0U and 6FAA are detected much better by ${\Phi}_{DoD}^*$ (see \Cref{fig:protein results}).
\begin{table}\centering
	\begin{tabular}{|c|c|c|c|}\hline
		\multicolumn{4}{|c|}{$\kappa=0.05$, $\alpha=0.05$\rule{0pt}{10pt}} \\
		\hline\rule{0pt}{10pt}
		$n$ & 5D0U vs 5D0U &5D0U vs 5JPT& 5D0U vs 6FAA  \\\hline
		
		100 & 0.055   & 0.068   & 0.109    \\
		250 & 0.049   & 0.080   & 0.297     \\
		500 & 0.037   & 0.090	& 0.690  \\\hline
\multicolumn{4}{|c|}{$\kappa=0.1$, $\alpha=0.05$\rule{0pt}{10pt}} \\
		\hline\rule{0pt}{10pt}
$n$ & 5D0U vs 5D0U &5D0U vs 5JPT& 5D0U vs 6FAA  \\\hline
	100 & 0.068 &   0.061   &   0.166    \\
	250  & 0.056&   0.084   &  0.420     \\
	500 & 0.047 &   0.104   & 0.760  \\\hline
\end{tabular}
	\medskip
	\caption{\textbf{Protein Comparison II:} The empirical power of the Test proposed by \citet{brecheteau2019statistical} (1000 applications) for the comparison of proteins represented in \Cref{fig: protein visiualization} for different $n$ and $\kappa$ ($n_S=n/5$).}\label{tab:protein comparison brecheteau}
\end{table}

\subsection{Discussion}
We conclude this section with some remarks on the way we have modeled proteins as metric measure spaces in this section. The flexibility of metric measure spaces offers possible refinements which might be of interest for further investigation. For example, we have treated all $C^\alpha$ atoms as equally important, although it appears to be reasonable for some applications to put major emphasis on the cores of the proteins. Further, one could have included that the error of measurement that is in general higher for some parts of the protein by adjusting the measure on the considered space accordingly. Finally, we remark that throughout this section we have considered proteins as rigid objects and shown that this allows us to efficiently discriminate between them. However, it is well known that proteins undergo different conformational states. In such a case the usage of the Euclidean metric as done previously will most likely cause $\Phi_{DoD}^*$ to discriminate between the different conformations, as the Euclidean distance is not suited for the matching of flexible objects \cite{elad2003bending}. Depending on the application one might want to take this into account by adopting a different metric reflecting (estimates of the) corresponding intrinsic distances and to modify the theory developed. Conceptually, this is straight forward but beyond the scope of this illustrative example.
\appendix

\section{Sketch of the Proofs}\label{appendix:proofs}
	% !TEX root = Dod.tex

The proofs of both \Cref{thm:main2} $(i)$ and \Cref{thm:main} $(i)$, that treat the case $\beta>0$, are based on distributional limits for the empirical $U$-quantile processes $\mathbb{U}_n^{-1}\coloneqq\sqrt{n}\big(U_n^{-1}\allowdisplaybreaks -U^{-1}\big) $ and $\mathbb{V}_m^{-1}\coloneqq\sqrt{m} \big(V_m^{-1}-V^{-1} \big)$ in $\ell^\infty[\beta,1-\beta]$. These limits are derived in Section B.2 of \citet{supplement} using the Hadamard differentiability of the inversion functional $\phi_{inv}:F\mapsto F^{-1}$ regarded as a map from the set of restricted distribution functions into the space $\ell^\infty[\beta,1-\beta]$ \cite[Lemma 3.9.23]{vaartWeakConvergenceEmpirical1996}. Once the distributional limits of $\mathbb{U}_n^{-1}$ and $\mathbb{V}_m^{-1}$ are established, both \Cref{thm:main2} $(i)$ and \Cref{thm:main} $(i)$ follow by standard arguments (see Section B.2 and Section B.4 of \cite{supplement}). Therefore, we focus in the following on the proofs of \Cref{thm:main2} $(ii)$, \Cref{thm:finite sample bound for p=2} and \Cref{thm:main} $(ii)$, which are concerned with the case $\beta=0$. As the proof of \Cref{thm:finite sample bound for p=2} is essential for the derivation of the second part of \Cref{thm:main2}, we prove it first.
\subsection{Proof of Theorem \ref*{thm:finite sample bound for p=2}}\label{subsec:proof of the finite sample boud}
 The key idea is to exploit the dependency structure of the samples $\left\lbrace \dX\left(X_i,X_j \right) \right\rbrace_{1\leq i <j \leq n} $ and $\left\lbrace \dY\left(Y_k,Y_l \right) \right\rbrace_{1\leq k <l \leq m}$. Let us consider the set $\left\lbrace \dX\left(X_i,X_j \right) \right\rbrace_{1\leq i <j \leq n} $. Since $X_1,\dots,X_n$ are independent, it follows that the random variable $\dX(X_i,X_j)$ is independent of $\dX(X_{i'},X_{j'})$, whenever $i,j,i',j'$ are pairwise different. This allows us to divide the sample into relatively large groups of independent random variables. This idea is represented in \Cref{fig:idearepresentation}. It highlights a possibility to divide the set $\left\lbrace \dX\left(X_i,X_j \right) \right\rbrace_{1\leq i <j \leq 6} $ into five sets of three independent distances.
\begin{figure}\centering
	\begin{subfigure}[c]{0.2\textwidth}
		\centering
		\includegraphics[width = \textwidth,height=\textwidth]{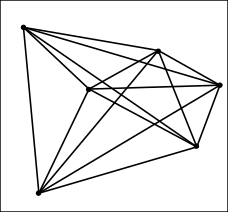}
	\end{subfigure}
\hspace*{0.5cm}
	\begin{subfigure}[c]{0.2\textwidth}
		\centering
		\includegraphics[width = \textwidth, height=\textwidth]{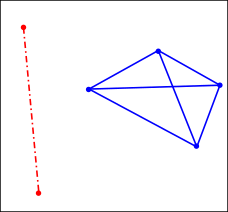}
	\end{subfigure}
\hspace*{0.5cm}
	\begin{subfigure}[c]{0.2\textwidth}
		\centering
		\includegraphics[width = \textwidth, height=\textwidth]{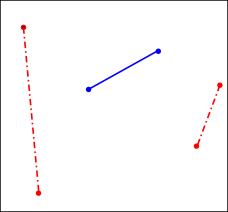}
	\end{subfigure}\vspace*{2mm}\\
	\begin{subfigure}[c]{0.2\textwidth}
		\centering
	\includegraphics[width = \textwidth,height=\textwidth]{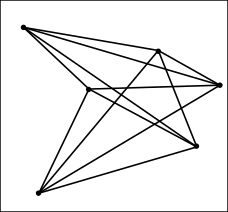}
\end{subfigure}
\hspace*{0.5cm}
\begin{subfigure}[c]{0.2\textwidth}\centering
	\includegraphics[width = \textwidth, height=\textwidth]{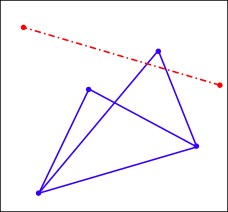}
\end{subfigure}
\hspace*{0.5cm}
\begin{subfigure}[c]{0.2\textwidth}\centering
	\includegraphics[width = \textwidth, height=\textwidth]{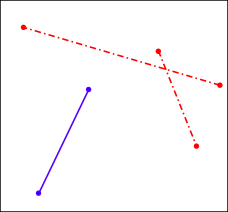}
\end{subfigure}
\caption{\textbf{Partitioning the Distances:} Illustration how to partition the set $\left\lbrace \dX\left(X_i,X_j \right) \right\rbrace_{1\leq i <j \leq 6} $ successively into five set of independent distances of size three. Top row: Left: All pairwise distances between six points. Middle: All distances (blue) that are independent of one chosen distance (red). Right: All distances that are independent of two chosen, independent distances (same color code, right) are shown. Bottom row: The same selection process for the set, where the independent distances displayed in the top right plot were removed.}\label{fig:idearepresentation}
\end{figure}
Generally, one can prove the following (see Section B.3 of the supplementary material \cite{supplement}).

\begin{lemma}\label{lem:Ustatsamplesplitup}
	Let $n\geq 3$ and let $X_1,\dots,X_n\iid\muX$. If $n$ is even, there exists a partition $\left\lbrace \Pi^n_k\right\rbrace_{1\leq k\leq n-1}$ of $\left\lbrace(i,j) \right\rbrace_{1\leq i<j\leq n}$ such that $|\Pi^n_k|=n/2$ for each $k$ and such that the random variables in the set $\left\lbrace \dX(X_i,X_j)\right\rbrace_{(i,j)\in\Pi^n_k}$ are independent, $1\leq k\leq n/2$. If $n$ is odd, there exists a partition $\left\lbrace \Pi^n_k\right\rbrace_{1\leq k\leq n}$ of $\left\lbrace(i,j) \right\rbrace_{1\leq i<j\leq n}$ such that $|\Pi^n_k|=(n-1)/2$ for each $k$ and such that the random variables in the set $\left\lbrace \dX(X_i,X_j)\right\rbrace_{(i,j)\in\Pi^n_k}$ are independent, $1\leq k\leq n/2$. 
\end{lemma}
The idea of the proof of \Cref{thm:finite sample bound for p=2} is now to write the problem at hand as a certain assignment problem and then to restrict the assignments to assignments between groups of independent distances.

%\begin{proof}
	\noindent We observe that for $\beta\in[0,1/2)$ 
	\[\Eargs{ \dodstatnn }\leq\Eargs{\int\limits_{0}^{1}\!|U_n^{-1}(t)-V_m^{-1}(t)|^2 \, dt }=\Eargs{\Wtwo{\muUn}{\muVm}},\]
	where $\muUn$ and $\muVm$ are the empirical measures corresponding to $U_n$ and $V_m$, i.e., for $A\in\borelsetsR$
	\[\muUn(A)=\frac{2}{n(n-1)}\sum_{1 \leq i <j\leq n}\indifunc{\dX(X_i,X_j)\in A}\]
	and
	\[\muVm(A)=\frac{2}{m(m-1)}\sum_{1 \leq k <l\leq m}\indifunc{\dY(Y_k,Y_l)\in A}.\]
	Since it holds $\muU=\mu^V$ by assumption, we obtain by the triangle inequality
	\[\Eargs{\Wtwo{\muUn}{\muVm}}\leq2\left( \Eargs{\Wtwo{\muUn}{\muU}}+\Eargs{\Wtwo{\muVm}{\mu^V}}\right). \]
	Thus, it remains to show
\begin{equation}\label{eq:aimfinitesamplebound}
	\Eargs{\Wtwo{\muUn}{\muU}}\leq\frac{4}{n+1}J_2\left(\muU\right).
\end{equation}	
In order to demonstrate \eqref{eq:aimfinitesamplebound}, we will use \Cref{lem:Ustatsamplesplitup}. Hence, it is notationally convenient to distinguish the cases $n$ even and $n$ odd, although the proof is essentially the same in both settings. In the following, we will therefore restrict ourselves to $n$ odd.

\noindent The first step to prove \eqref{eq:aimfinitesamplebound} is to realize (cf. \citet[Sec. 4]{bobkov2016one}), that
		\begin{equation}\label{eq:firstU-quantileLpnorm estimate}
\Eargs{\Wtwo{\muUn}{\muU}}\leq\Eargs{\Wtwo{\muUn}{\nuUn}},
\end{equation}
where $\nuUn$ denotes an independent copy of $\muUn$, i.e.,
\[\nuUn(A)=\frac{2}{n(n-1)}\sum_{1 \leq i <j\leq n}\indifunc{\dX(X_i',X_j')\in A},\]
for $X_1',\dots,X_n'\iid \muX$. By \Cref{lem:Ustatsamplesplitup}, there exists a partition $\Pi^n_1,\dots,\Pi^n_{n}$ of the index set $\left\lbrace (i,j)\right\rbrace_{1\leq i<j\leq n}$ with $|\Pi^n_k|=(n-1)/2$, $1\leq k\leq n$, such that the random variables in the sets $\{\dX(X_i,X_j)\}_{(i,j)\in\Pi^n_k}$ and the ones in the sets $\{ \dX(X_i',X_j')\}_{(i,j)\in\Pi^n_k}$ are independent. Let $\left\lbrace d_{(i)}^{\Pi^n_k,X}\right\rbrace_{1\leq i\leq(n-1)/2}$ stand for the ordered sample of $\lbrace \dX(X_i,X_j )\rbrace_{(i,j)\in\Pi^n_k}$ and the let $\Big\lbrace d_{(i)}^{\Pi^n_k,X'}\Big\rbrace_{1\leq i\leq(n-1)/2}$ stand for the one of $\left\lbrace \dX\left(X_i',X_j' \right)\right\rbrace_{(i,j)\in\Pi^n_k}$, $1\leq k\leq n$. The application of Corollary B.12 with this partition yields that
\begin{align*}
\Eargs{\Wtwo{\muUn}{\nu^U_n}}\leq&\frac{2}{n(n-1)}\Eargs{\sum_{k=1}^{n}\sum_{i=1} ^{(n-1)/2}\left| d^{\Pi^n_k,X}_{(i)}-d^{\Pi^n_k,X'}_{(i)}\right|^2}. 
\end{align*}
Furthermore, we realize that, as $X_1,\dots,X_n,X_1',\dots,X_n'$ are independent, identically distributed, it holds for $1\leq k,l\leq n$ that 
\begin{equation*}
\sum_{i=1} ^{(n-1)/2}\left| d^{\Pi^n_k,X}_{(i)}-d^{\Pi^n_k,X'}_{(i)}\right|^2\overset{D}{=}	\sum_{i=1} ^{(n-1)/2}\left| d^{\Pi^n_l,X}_{(i)}-d^{\Pi^n_l,X'}_{(i)}\right|^2.
\end{equation*}
Consequently, we have
\begin{align*}
\Eargs{\sum_{k=1}^{n}\sum_{i=1} ^{(n-1)/2}\left| d^{\Pi^n_k,X}_{(i)}-d^{\Pi^n_k,X'}_{(i)}\right|^2}=n\Eargs{\sum_{i=1} ^{(n-1)/2}\left| d^{\Pi^n_1,X}_{(i)}-d^{\Pi^n_1,X'}_{(i)}\right|^2}.
\end{align*}
We come to the final step of this proof. Let for any $A\in\borelsetsR$
\[\mu_n^*(A)=\frac{2}{n-1}\sum_{i=1} ^{(n-1)/2}\indifunc{d_{(i)}^{\Pi^n_1,X}\in A}\]
and let $\nu_n^*(A)$ be defined analogously.
%\[\nu_n^*(A)=\frac{2}{n-1}\sum_{i=1} ^{(n-1)/2}\indifunc{d_{(i)}^{\Pi^n_1,X'}\in A}.\]
Then, Theorem 4.3 of \citet{bobkov2016one} implies that
\[\Eargs{\Wtwo{\mu_n^*}{\nu_n^*}}=\frac{2}{n-1}\Eargs{\sum_{i=1} ^{(n-1)/2}\left| d^{\Pi^n_1,X}_{(i)}-d^{\Pi^n_1,X'}_{(i)}\right|^2}.\]
By construction, the samples $\left\lbrace d_{i}^{\Pi^n_1,X}\right\rbrace_{1\leq i\leq (n-1)/2}$ and $\left\lbrace d_{i}^{\Pi^n_1,X'}\right\rbrace_{1\leq i\leq (n-1)/2}$ consist of independent random variables and are independent of each other. Furthermore, we have $\Eargs{\mu_n^*}=\Eargs{\nu_n^*}=\muU$. Since $J_2\left(\mu^U\right)<\infty$ by assumption, it follows by Theorem 5.1 of \citet{bobkov2016one} that
\begin{align*}
\frac{2}{n-1}\Eargs{\sum_{i=1} ^{(n-1)/2}\left| d^{\Pi^n_1,X}_{(i)}-d^{\Pi^n_1,X'}_{(i)}\right|^2}=\Eargs{\Wtwo{\mu_n^*}{\nu_n^*}}\leq \frac{4}{n+1}J_2\left(\mu^U\right).
\end{align*}
This yields \eqref{eq:aimfinitesamplebound} and thus concludes the proof. \hfill\qed
%\end{proof}
\subsection{Proof of Theorem \ref*{thm:main2} \texorpdfstring{$(ii)$}{(ii)}}

\noindent For notational convenience we restrict ourselves from now on to the case $n=m$. However, the same strategy of proof also gives the general case $n\neq m$ (for some additional details on this issue see \cite[Sec. B.2.2]{supplement}). We first demonstrate that $\Big\{\XiUVn (\beta)\,|\,\beta\in[0,1/2]\Big\}_{n\in\mathbb{N}}\subset\left(C[0,1/2],||\cdot||_\infty\right),$ where 
		\begin{align*}
		\XiUVn(\beta)=\frac{n}{2}\int_{\beta}^{1-\beta}\!\left( U_n^{-1}(t)-V_n^{-1}(t)\right)^2 \, dt,
	\end{align*}
	is tight. Under \Cref{cond:secondcondition}, which implies \Cref{cond:firstcondition} for $\beta\in (0,1)$, we already have by \Cref{thm:main2} $(i)$ that $\XiUVn(\beta)\rightsquigarrow\Xi(\beta)$ for all $\beta\in(0,1/2)$. In order to prove tightness of the sequence $\left\{\XiUVn\right\}_{n\in\mathbb{N}}\subset C[0,1/2]$ we process the subsequent steps:
	\begin{enumerate}
		\item Show that the sequence of real valued random variables $\left\lbrace \XiUVn(0)\right\rbrace_{n\in\N}$ is tight;
		\item Control the following expectations for small $\beta$
		\[	\Eargs{\int_{0}^{\beta}\! \left(U_n^{-1}(t)-V_n^{-1}(t)\right)^2\,dt } \text{  and  } \Eargs{\int_{1-\beta}^{1}\! \left(U_n^{-1}(t)-V_n^{-1}(t)\right)^2\,dt }.\]
	\end{enumerate}
While the first step of the above strategy directly follows by a combination of Lemma B.6 and \Cref{thm:finite sample bound for p=2}, we have to work for the second. Using a technically somewhat more involved variation of the partitioning idea of the proof of \Cref{thm:finite sample bound for p=2}, we can demonstrate the subsequent bounds. 
\begin{lemma}\label{lemma:integrated difference for small beta}
	Suppose \Cref{cond:setting} and \Cref{cond:secondcondition} are met. Let $\muU=\mu^V$, let $n\geq 100$, let $0\leq\beta=\beta_n<1/6$ and let $n\beta > 8$. Then, it follows that 
	\begin{align}\label{eq:important estimate1}
	&\Eargs{\int_{0}^{\beta}\! \left(U_n^{-1}(t)-V_n^{-1}(t)\right)^2\,dt } \leq \frac{2C_1}{n-1}\left( 4\beta\left( 1+2\frac{\sqrt{\log\left( n\right)}}{\sqrt{n }}\right)\right)^{2\gamma_1+2}+o\left(n^{-1}\right)
	\end{align}
	as well as 
	\begin{align}\label{eq:important estimate2}
	&\Eargs{\int_{1-\beta}^{1}\! \left(U_n^{-1}(t)-V_n^{-1}(t)\right)^2\,dt } \leq \frac{2C_2}{n-1}\left( 4\beta\left( 1+2\frac{\sqrt{\log\left( n\right)}}{\sqrt{n }}\right)\right)^{2\gamma_2+2} +o\left(n^{-1}\right),
	\end{align}
	where $C_1$ and $C_2$ denote finite constants that are independent of $\beta$.
\end{lemma}
Based on this, we prove the subsequent technical lemma.
\begin{lemma}\label{lemma:Stetigkeitsmodul}
	Under \Cref{cond:secondcondition}, it holds for all $\epsilon>0$ that
	\[\lim_{\delta\to0^+}\limsup_{n\to\infty}\prob\left(\omega(\XiUVn,\delta)>\epsilon)\right)=0,\]
	where $\omega(\cdot,\cdot)$ is defined for $f:[0,1/2]\to\R$ and $0<\delta\leq 1/2$ as $\omega(f,\delta)=\sup_{|s-t|\leq \delta}|f(s)-f(t)|$.
\end{lemma}
\begin{proof}
	Let $0<\delta<1/20$ and let $0\leq s, t\leq1/2$. We have that
	\begin{align}
	&\prob\left( \omega(\XiUVn,\delta)>\epsilon\right)= \prob\left( \sup_{|s-t|\leq \delta}\left|\XiUVn(s)-\XiUVn(t) \right| >\epsilon\right)\nonumber\\
	 = &\prob\left( \sup_{\substack{|s-t|\leq \delta, \\ t< 2\delta}}\left|\XiUVn(s)-\XiUVn(t) \right| >\epsilon \text{ or }  \sup_{\substack{|s-t|\leq \delta,\\ t\geq 2\delta}}\left|\XiUVn(s)-\XiUVn(t) \right| >\epsilon\right)\nonumber\\
	\leq& \prob\left( \sup_{\substack{|s-t|\leq \delta, \\ t< 2\delta}}\left|\XiUVn(s)-\XiUVn(t) \right| >{\epsilon} \!\right)\!+\!\prob\left(\sup_{\substack{|s-t|\leq \delta, \\ t\geq 2\delta}}\left|\XiUVn(s)-\XiUVn(t) \right| >{\epsilon}\!\right)\nonumber\\
	=& \text{ \textbf{I}}+\text{\textbf{II}}\label{eq:twosummands}.
	\end{align}
	In the following, we consider both summands separately.\vspace*{2mm}
	
	\noindent\textit{First Summand:} Since $\XiUVn(\beta)$ is monotonically decreasing in $\beta$, 
	\begin{align*}
	\text{ \textbf{I}}\leq& 	\prob\left( \XiUVn(0)-\XiUVn(3\delta) >{\epsilon} \right)\\
	=&	\prob\left( \int_{0}^{3\delta}\!|\UV(s)|^2\diff s+ \int_{1-3\delta}^{1}\!|\UV(s)|^2\diff  s >{\epsilon} \right),
	\end{align*}
	where $\UV=\sqrt{\frac{n}{2}}\left(\qu-\qv\right)$. Further, we obtain that
	\begin{align*}
	&\prob\left( \int_{0}^{3\delta}\!|\UV(s)|^2\diff s+ \int_{1-3\delta}^{1}\!|\UV(s)|^2\diff  s >{\epsilon} \right)\\
	%\leq & \prob\left( \int_{0}^{3\delta}\!|\UV(s)|^2\diff s>\frac{\epsilon}{2}\right) +  \prob\left( \int_{1-3\delta}^{1}\!|\UV(s)|^2\diff  s >\frac{\epsilon}{2} \right)\\
	\leq &\,\frac{2}{\epsilon} \left(\Eargs{ \int_{0}^{3\delta}\!|\UV(s)|^2\diff s}+   \Eargs{ \int_{1-3\delta}^{1}\!|\UV(s)|^2\diff  s  }\right).
	\end{align*}
%	Thus, we find that
%	\begin{align*}
%	&\limsup_{n\to\infty}\prob\left( \sup_{|s-t|\leq \delta, \atop t< 2\delta}\left|\XiUVn(s)-\XiUVn(t) \right| >{\epsilon} \right)\\\leq &\limsup_{n\to\infty} \frac{2}{\epsilon} \Eargs{ \int_{0}^{3\delta}\!|\UV(s)|^2\diff s}+ \limsup_{n\to\infty} \frac{2}{\epsilon} \Eargs{ \int_{1-3\delta}^{1}\!|\UV(s)|^2\diff  s  }.
%	\end{align*}
	As $3\delta<1/6$, we can conclude with \Cref{lemma:integrated difference for small beta} that 
	\begin{align*}
	&\limsup_{n\to\infty}\frac{2}{\epsilon}\mathbb{E}\left[ \int_{0}^{3\delta}\!\left|\UV(s)\right|^2\diff s\right]\\
	%=&\limsup_{n\to\infty}\frac{n}{\epsilon}\mathbb{E}\left[ \int_{0}^{3\delta}\!\left(U_n^{-1}(t)-V_n^{-1}(t)\right)^2\diff s\right]\\
	\leq& \limsup_{n\to\infty} \frac{n}{\epsilon}\left( \frac{2C_1}{n-1}\left( 6\delta\left( 1+2\frac{\sqrt{\log\left( n\right)}}{\sqrt{n }}\right)\right)^{2\gamma_1+2} +o\left(n^{-1}\right)\right) 
	\leq C_2\delta^{2\gamma_1+2}.
	\end{align*}
	Here, $C_1$ and $C_2$ denote finite constants that are independent of $\delta$. Similarly,
	\[\limsup_{n\to\infty}\Eargs{ \int_{1-3\delta}^{1}\!|\UV(s)|^2\diff  s  }\leq C_3\delta^{2\gamma_2+2},\]
	where $C_3$ denotes a finite constant independent of $\delta$. Since we have by assumption that $\gamma_1,\gamma_2>-1$, i.e., $2\gamma_1+2>0$ and $2\gamma_2+2>0$, it follows that
	\[\lim_{\delta\to0+}\limsup_{n\to\infty}\prob\left( \sup_{\substack{|s-t|\leq \delta, \\ t< 2\delta}}\left|\XiUVn(s)-\XiUVn(t) \right| >{\epsilon} \right)=0.\]
	\textit{Second Summand:} In order to handle \textbf{II} in \eqref{eq:twosummands}, we want to make use of the fact that for $\delta >0$ the process $\UV=\sqrt{\frac{n}{2}}(U_n^{-1}-V_n^{-1})$ converges to a Gaussian process in $\ell^\infty[\delta,1-\delta]$ given \Cref*{cond:secondcondition} (see Lemma B.7 of \citet{supplement}). To this end, we verify that the function 
	\begin{align*}&\Upsilon:\ell^\infty[\delta,1-\delta]\times \ell^\infty[\delta,1-\delta] \to\R,\\&(f,g)\mapsto \sup_{\substack{|s-t|\leq \delta, \\ t\geq 2\delta}}\left|\int_{s}^{1-s}\!f^2(x) \diff x -\int_{t}^{1-t}\! g^2(x)\diff x \right| \end{align*}
	is continuous.
	
	\noindent Let in the following $||\cdot||_\infty$ denote the norm of $\ell^\infty[\delta,1-\delta]$. Let $\left((f_n,g_n)\right)_{n\in\N}\subset \ell^\infty[\delta,1-\delta]\times \ell^\infty[\delta,1-\delta]$ be a sequence such that $(f_n,g_n)\to(f,g)$ with respect to the product norm $||(f,g)||\coloneqq||f||_\infty+||g||_\infty$. Then, the inverse triangle inequality yields
	\begin{align*}
	&\left| \Upsilon(f_n,g_n)-\Upsilon(f,g)\right|\\
	\leq& \sup_{\substack{|s-t|\leq \delta, \\ t\geq 2\delta}} \Bigg|\int_{s}^{1-s}\! f_n^2 (x)-f^2(x)\diff x -\int_{t}^{1-t}\! g_n^2(x)-g^2(x) \diff x \Bigg|\\
%	=&  \sup_{|s-t|\leq \delta, \atop t\geq 2\delta} \left|\int_{s}^{1-s}\! \left( f_n (x)+f(x)\right)\left( f_n (x)-f(x)\right) \diff x +\int_{t}^{1-t}\! \left(g (x)+g_n(x)\right)\left( g (x)-g_n(x)\right) \diff x \right|\\
	\leq&  \sup_{\substack{|s-t|\leq \delta, \\ t\geq 2\delta}} \left|\int_{s}^{1-s}\! \left\|  f_n +f\right\|_\infty \left\|  f_n -f\right\|_\infty \diff x +\int_{t}^{1-t}\! \left\| g +g_n\right\| \left\| g-g_n\right\|_\infty  \diff x \right|\\
	\leq& (1-\delta)\max\left\lbrace  \left\|f_n\right\|_\infty +\left\|f\right\|_\infty,~ \left\|g_n\right\|_\infty +\left\|g\right\|_\infty \right\rbrace \left\|\left( f-f_n, g-g_n\right)\right\|\overset{n\to\infty}{\to}0.
	\end{align*}
	Thus, we have shown $\lim_{n\to\infty}\Upsilon(f_n,g_n)=\Upsilon(f,g)$, i.e., that $\Upsilon$ is sequentially continuous. Hence, a combination of Lemma B.7 and the Continuous Mapping Theorem \citep[Thm. 1.3.6]{vaartWeakConvergenceEmpirical1996} yields that
	\[\Upsilon\left( \UV,\UV\right) \rightsquigarrow \Upsilon\left( \G,\G\right),\]
	where $\G$ denotes the centered Gaussian process defined in Theorem 2.5. Furthermore, Lemma B.8 shows that $\XiUVn(\beta)$ is measurable for $\beta\in[\delta,1-\delta]$ and $n\in\N$. As $\XiUVn$ is continuous in $\beta$ this induces the measurability of $\Upsilon\left( \UV,\UV\right)$ for $n\in\N$. Thus, we find that
	\[\Upsilon\left( \UV,\UV\right) \Rightarrow\Upsilon\left( \G,\G\right).\]
	Let $A=[\epsilon,\infty)\subset\R$. Then, the set $A$ is closed and $(\epsilon,\infty)\subset A$. Hence, an application of the Portmanteau-Theorem \citep[Thm. 2.1]{BillingsleyConvergenceProbabilityMeasures2013} yields that
	\begin{align*}
	\limsup_{n\to\infty} \prob\left(\!\sup_{\substack{|s-t|\leq \delta, \\ t\geq 2\delta}}\left|\XiUVn(s)-\XiUVn(t) \right| >{\epsilon}\!\right)&\leq\limsup_{n\to\infty}\prob\left( \Upsilon\left( \UV,\!\UV\right) \!\! \in\! A\right)\\&\leq\prob\left( \Upsilon\left( \G,\G\right)  \geq{\epsilon}\right).
	\end{align*}
	Next, we remark that 
	\[\sup_{\substack{|s-t|\leq \delta, \\ t\geq 2\delta}}\left|\XiUVn(s)-\XiUVn(t) \right|=\sup_{\substack{|s-t|\leq \delta, \\ s\leq t, t\geq 2\delta}}\left|\XiUVn(s)-\XiUVn(t) \right|.\]
	Hence, we can assume for the treatment of this summand that $s\leq t$. With this, we obtain that
	\begin{align*}
	\limsup_{n\to\infty} \,&\prob\left(\sup_{\substack{|s-t|\leq \delta, \\ t\geq 2\delta}}\left|\XiUVn(s)-\XiUVn(t) \right| >{\epsilon}\right)\\
	\leq&  \prob\left(\sup_{\substack{|s-t|\leq \delta, \\ t\geq 2\delta}}\left|\int_{s}^{1-s}\!\G^2(x) \diff x -\int_{t}^{1-t}\! \G^2(x)\diff x \right|\geq {\epsilon}\right)\\
	%=&\prob\left(\sup_{|s-t|\leq \delta, \atop t\geq 2\delta}\int_{s}^{t}\!\G^2(x) \diff x +\int_{1-t}^{1-s}\! \G^2(x)\diff x \geq {\epsilon}\right)\\
	\leq&\prob\left(\sup_{\substack{|s-t|\leq \delta, \\ t\geq 2\delta}}\int_{s}^{t}\!\G^2(x) \diff x \geq \frac{\epsilon}{2}\right)+\prob\left(\sup_{\substack{|s-t|\leq \delta, \\ t\geq 2\delta}}\int_{1-t}^{1-s}\!\G^2(x) \diff x \geq\frac{\epsilon}{2}\right).
	\end{align*}
	In the following, we focus on the first term. As $0<\delta<1/20$ and $ t\leq1/2$, it holds
	\begin{align*}
	&\prob\left(\sup_{\substack{|s-t|\leq \delta, \\ t\geq 2\delta}}\int_{s}^{t}\!\G^2(x) \diff x \geq \frac{\epsilon}{2}\right)
	%\leq & \prob\left(\sup_{|s-t|\leq \delta, \atop t\geq 2\delta} \left( \sup_{x\in[s,t]}\G^2(x) \right)(t-s) \geq\frac{\epsilon}{2}\right)\\
	\!\leq\!  \prob\left(\sup_{\substack{|s-t|\leq \delta, \\ t\geq 2\delta}} \left( \sup_{x\in[\delta,1-\delta]}\G^2(x) \right)(t-s) \geq \frac{\epsilon}{2}\right)\\
	=& \prob\left( \sup_{x\in[\delta,1-\delta]}\G^2(x)  \geq \frac{\epsilon}{\delta}\right)
	\leq\sqrt{\frac{\delta}{\epsilon}}\Eargs{\sup_{x\in[\delta,1-\delta]}\left| \G(x)\right| }. 
	\end{align*}
	By Lemma D.11 the Gaussian process $\G$ is continuous on $[\delta,1-\delta]$ under the assumptions made, i.e., almost surely bounded on $[\delta,1-\delta]$. Thus, Theorem 2.1.1 of \citet{AdlerRandomFieldsGeometry2007} ensures that $\Eargs{\sup_{x\in[\delta,1-\delta]}\left| \G(x)\right| }<\infty$. Hence, we find that
	\begin{align*}
	\lim_{\delta\to0+} \prob\left(\sup_{\substack{|s-t|\leq \delta, \\ t\geq 2\delta}}\int_{s}^{t}\!\G^2(x) \diff x \geq \frac{\epsilon}{2}\right)\leq \lim_{\delta\to0+}\sqrt{\frac{\delta}{\epsilon}}\Eargs{\sup_{x\in[\delta,1-\delta]}\left| \G(x)\right| }=0.
	\end{align*} 
Analogously,
	\[\lim_{\delta\to0+}\prob\left(\sup_{\substack{|s-t|\leq \delta, \\ t\geq 2\delta}}\int_{1-t}^{1-s}\!\G^2(x) \diff x >\frac{\epsilon}{2}\right)=0.\]
	This concludes the treatment of the second summand in \eqref{eq:twosummands}.\vspace*{2mm}
	
	\noindent Combining the results for \textbf{I} and \textbf{II}, we find that
	\begin{align*}&\lim_{\delta\to0+}\limsup_{n\to\infty}\prob\left( \sup_{|s-t|\leq \delta}\left|\XiUVn(s)-\XiUVn(t) \right| >\epsilon\right)
	%\\\leq&\lim_{\delta\to0+}\limsup_{n\to\infty}\prob\left( \sup_{|s-t|\leq \delta, \atop t< 2\delta}\left|\XiUVn(s)-\XiUVn(t) \right| >{\epsilon} \right)\\
	%+&\lim_{\delta\to0+}\limsup_{n\to\infty}\prob\left(\sup_{|s-t|\leq \delta, \atop t\geq 2\delta}\left|\XiUVn(s)-\XiUVn(t) \right| >{\epsilon}\right) \\
	=0.\end{align*}
	Thus, we have proven \Cref{lemma:Stetigkeitsmodul}.
\end{proof}
Now, we obtain the tightness of the sequence $\left\{\XiUVn\right\}_{n\in\mathbb{N}}$ in $C[0,1/2]$ as a simple consequence of the above results.
\begin{cor}\label{cor:tightness}
	Under \Cref*{cond:secondcondition}, the sequence $\left\{\XiUVn\right\}_{n\in\mathbb{N}}$ is tight in $\left( C[0,1/2],||\cdot||_\infty\right)$.
\end{cor}
\begin{proof}
	By Theorem 7.3 in \citet{BillingsleyConvergenceProbabilityMeasures2013} (and a rescaling argument) it is sufficient to prove that the sequence $\left\{\XiUVn(0)\right\}_{n\in\mathbb{N}}$ is tight in $\R$ and that\[\lim_{\delta\to0^+}\limsup_{n\to\infty}\prob\left(\omega(\XiUVn,\delta))>\epsilon)\right)=0.\] We have already noted that $\left\{\XiUVn(0)\right\}_{n\in\mathbb{N}}$ is tight (by \Cref{thm:finite sample bound for p=2}, which is applicable due to Lemma B.6) and thus \Cref{lemma:Stetigkeitsmodul} yields \Cref{cor:tightness}. 
\end{proof}
We conclude the proof of \Cref{thm:main2} $(ii)$ by using the Skorohod Representation Theorem \cite[Thm. 6.7]{BillingsleyConvergenceProbabilityMeasures2013} to verify (see Section B.2.2) that the tightness of $\left\{\XiUVn \right\}_{n\in\mathbb{N}}$ induces that
	\[\frac{nm}{n+m}\int_{0}^{1}\!\left( U_n^{-1}(t)-V_m^{-1}(t)\right)^2 \, dt\rightsquigarrow\int_{0}^{1}\!(\G(t))^2 \, dt.\]

\subsection{Proof of Theorem \ref*{thm:main} \texorpdfstring{$(ii)$}{(ii)}}\label{subsec:proof of thm:main2}
The most important step of the proof of \Cref{thm:main} $(i)$ (see \cite[Sec. B.4]{supplement}) is to derive the limit distributions of 
 \begin{align}\label{eq:limit1}
 \int_{\beta}^{1-\beta}\!\zeta(t)\mathbb{U}_n^{-1}(t) \, dt  \text{ and } \int_{\beta}^{1-\beta}\!\zeta(t)\mathbb{V}_m^{-1}(t) \, dt
 \end{align}
 under \Cref{cond:firstcondition}, where $\zeta(t)\coloneqq U^{-1}(t)-V^{-1}(t)$.
% and
% \begin{align}\label{eq:limit2}
%\int_{\beta}^{1-\beta}\!\left( U^{-1}(t)-V^{-1}(t)\right)\mathbb{V}_m^{-1}(t) \, dt. \end{align}
 Under \Cref{cond:firstcondition} these can be derived via the distributional limits for the empirical $U$-quantile processes $\mathbb{U}_n^{-1}\coloneqq\sqrt{n}\big(U_n^{-1}\allowdisplaybreaks -U^{-1}\big) $ and $\mathbb{V}_m^{-1}\coloneqq\sqrt{m} \big(V_m^{-1}-V^{-1} \big)$ in $\ell^\infty[\beta,1-\beta]$. However, as already argued, we cannot expect $\ell^\infty(0,1)$-convergence of $\mathbb{U}_n^{-1}$ and $\mathbb{V}_m^{-1}$ under \Cref{cond:secondcondition}. Reconsidering \eqref{eq:limit1}, we realize that $\ell^1(0,1)$-convergence of $\mathbb{U}_n^{-1}$ and $\mathbb{V}_m^{-1}$ is sufficient to derive the corresponding limiting distributions. Convergence in $\ell^1(0,1)$ is much weaker than convergence in $\ell^2(0,1)$ or $\ell^\infty(0,1)$. Indeed, it turns out that this convergence can quickly be verified, since $\phi_{inv}$ is Hadamard differentiable in the present setting as a map from $\mathbb{D}_2\subset D[a,b]\to \ell^1(0,1)$. Using ideas from \citet{kaji2017switching} we can show the following.
 \begin{lemma}\label{lem:Hadamardfiff2}
 	Let $F$ have compact support on $[a,b]$ and let $F$ be continuously differentiable on its support with derivative $f$ that is strictly positive on $(a,b)$ (Possibly, $f(a)=0$ and/or $f(b))=0$). Then the inversion functional $\phi_{inv}:F\mapsto F^{-1}$ as a map $\mathbb{D}_2\subset D[a,b]\to \ell^1(0,1)$ is Hadamard-differentiable at $F$ tangentially to $C[a,b]$ with derivative $\alpha\mapsto-(\alpha/f)\circ F^{-1}$.
 \end{lemma}
 \begin{proof}
 	Let $h_n\to h$ uniformly in $D[a,b]$, where $h$ is continuous, $t_n\to 0$ and $F+t_nh_n\in\mathbb{D}_2$ for all $n\in\N$. Let $\HadPhi(\alpha)=-(\alpha/f)\circ F^{-1}$. We have to demonstrate that
 	\[\left\| \frac{\phi_{inv}(F+t_nh_n)-\phi_{inv}(F)}{t_n}-\HadPhi(h)\right\|_{\ell^1(0,1)} \to 0,\]
 	as $n\to \infty$. We realize that for every $\epsilon>0$ there exist $a_\epsilon,b_\epsilon\in[a,b]$ such that $\max\lbrace F(a_\epsilon),1-F(b_\epsilon)\rbrace<\epsilon $ and $f$ is strictly positive on $[a_\epsilon,b_\epsilon]$. It follows that	
 	\begin{align*}
 	&\left\| \frac{\phi_{inv}(F+t_nh_n)-\phi_{inv}(F)}{t_n}-\HadPhi(h)\right\|_{\ell^1(0,1)}\\\leq& \int_{F(a_\epsilon)+\epsilon}^{F(b_\epsilon)-\epsilon}\!\left| \frac{\phi_{inv}(F+t_nh_n)-\phi_{inv}(F)}{t_n}-\HadPhi(h)\right|(s)\,ds\\+&\int_{0}^{2\epsilon}\!\left| \frac{\phi_{inv}(F+t_nh_n)-\phi_{inv}(F)}{t_n}-\HadPhi(h)\right|(s)\,ds\\ +&\int_{1-2\epsilon}^{1}\!\left| \frac{\phi_{inv}(F+t_nh_n)-\phi_{inv}(F)}{t_n}-\HadPhi(h)\right|(s)\,ds. 
 	\end{align*}
 Next, we treat the summands separately. The claim follows once we have shown that the first summand vanishes for all $\epsilon>0$ as $n\to\infty$ and the other two summands become arbitrarily small for $\epsilon$ small and $n\to\infty$.
 
 \noindent\textit{First summand:} We start with the first summand. Since its requirements are fulfilled for all $\epsilon>0$, we have by Lemma 3.9.23 of \citet{vaartWeakConvergenceEmpirical1996} that
 	\begin{align}\label{eq:li-Hadamard first summand}
 	\sup_{s\in [F(a_\epsilon)+\epsilon,F(b_\epsilon)-\epsilon]}\left| \frac{\phi_{inv}(F+t_nh_n)-\phi_{inv}(F)}{t_n}-\HadPhi(h)\right|(s)\to 0,
 	\end{align}
 	as $n\to\infty$. Thus, the same holds for 
 	\[\int_{F(a_\epsilon)+\epsilon}^{F(b_\epsilon)-\epsilon}\!\left| \frac{\phi_{inv}(F+t_nh_n)-\phi_{inv}(F)}{t_n}-\HadPhi(h)\right|(s)\,ds,\]
 	which is bounded by \eqref{eq:li-Hadamard first summand}.
 	
 	\noindent\textit{Second summand:} We have
 	\begin{align}
 	&\int_{0}^{2\epsilon}\!\left| \frac{\phi_{inv}(F+t_nh_n)-\phi_{inv}(F)}{t_n}-\HadPhi(h)\right|(s)\,ds\nonumber\\
 	\leq &\int_{0}^{2\epsilon}\!\left| \frac{\phi_{inv}(F+t_nh_n)-\phi_{inv}(F)}{t_n}\right|(s)\,ds +\int_{0}^{2\epsilon}\!\left|\HadPhi(h)\right|(s)\,ds.\label{eq:nikolaus}
 	\end{align}
 	In the following, we consider both terms separately. For the first term, we find that
 	\begin{align*}
 	&\int_{0}^{2\epsilon}\!\left| \frac{\phi_{inv}(F+t_nh_n)-\phi_{inv}(F)}{t_n}\right|(s)\,ds\\=&\frac{1}{|t_n|}\int_{0}^{2\epsilon}\!\left| (F+t_nh_n)^{-1}(s)-F^{-1}(s)\right| \,ds.
 	\end{align*}
 	Next, we realize that for $G\in\mathbb{D}_2\subset D[a,b]$ and $s\in(0,1)$, we have that
 	\begin{align*}G^{-1}(s)=-\int_a^{b}\!\indifunc{ s\leq G(x)}\,dx+b.\end{align*}
 	Since $F\in \mathbb{D}_2$ and $F+t_nh_n\in\mathbb{D}_2$ for all $n\in\N$, this yields that
 	\begin{align*}
 	&\frac{1}{|t_n|}\int_{0}^{2\epsilon}\!\left| (F+t_nh_n)^{-1}(s)-F^{-1}(s)\right| \,ds\\
 	%=&\frac{1}{|t_n|}\int_{0}^{2\epsilon}\!\left| \int_a^{b}\!\indifunc{s\leq (F+t_nh_n)(x) }\,dx-\int_a^{b}\!\indifunc{s\leq F(x)}\,dx\right| \,ds\\
 	\leq&\frac{1}{|t_n|} \int_a^{b}\!\int_{0}^{2\epsilon}\!\left|\indifunc{s\leq (F+t_nh_n)(x)}-\indifunc{s\leq F(x)}\right|\,ds\,dx, 
 	\end{align*}
 	where we applied the Theorem of Tonelli/Fubini \cite[Thm. 18.3]{billingsley2008probability} in the last step. Let in the following $F_{t_nh_n}=F+t_nh_n$. Then, we obtain that
% 	\begin{align*}&\int_{0}^{2\epsilon}\!\left|\indifunc{s\leq F_{t_nh_n}(x)}-\indifunc{s\leq F(x)}\right|\,ds\\=&\begin{cases}
% 	\left|F(x)-F_{t_nh_n}(x) \right| , &\mathrm{if}~ F(x),F_{t_nh_n}(x)\leq 2\epsilon\\
% 	2\epsilon-F_{t_nh_n}(x) , &\mathrm{if}~ F_{t_nh_n}(x)\leq 2\epsilon,~F(x)> 2\epsilon\\
% 	2\epsilon-F(x) , &\mathrm{if}~ F(x)\leq 2\epsilon,~F_{t_nh_n}(x)> 2\epsilon\\
% 	0, & \mathrm{else}.
% 	\end{cases}\end{align*}
% 	Thus, we clearly have that
 	\begin{align*}&\int_{0}^{2\epsilon}\!\!\left|\indifunc{s\leq F_{t_nh_n}(x)}-\indifunc{s\leq F(x)}\right|\,ds\\\leq&
 	\begin{cases}
 	\left|F(x)-F_{t_nh_n}(x) \right| , &\mathrm{if}~ \min\left\lbrace F(x),F_{t_nh_n}(x)\right\rbrace \leq 2\epsilon.\\
 	0, & \mathrm{else}.
 	\end{cases}\end{align*}
 	Let now $x\geq F^{-1}\left(2\epsilon+\left\| h_nt_n\right\|_\infty \right)$, then it follows that
 	\[F(x)\geq2\epsilon+\left\| h_nt_n\right\|_\infty\geq2\epsilon\] { as well as } \[F_{t_nh_n}(x)\geq 2\epsilon+\left\| h_nt_n\right\|_\infty+t_nh_n(x)\geq 2\epsilon.\]
 	Combining these findings, we obtain that
 	\begin{align*}
 	&\frac{1}{|t_n|}\int_{0}^{2\epsilon}\!\left| (F+t_nh_n)^{-1}(s)-F^{-1}(s)\right| \,ds\\\leq&\frac{1}{|t_n|}\int_{a}^{F^{-1}\left(2\epsilon+\left\| h_nt_n\right\|_\infty \right)}\!\left|(F+t_nh_n)(x)-F(x)\right| \,dx\\
 	\leq &\left\|h_n-h \right\|_{\ell^1(a,b)}+ \int_{a}^{F^{-1}\left(2\epsilon+\left\| h_nt_n\right\|_\infty \right)}\!\left|h(x)\right| \,dx.
 	\end{align*}
 	We realize that the first term goes to zero as $n\to\infty$ by construction. Further, since $t_n\to0$ for $n\to\infty$, we obtain that $F^{-1}\left(2\epsilon+\left\| h_nt_n\right\|_\infty \right)\to F^{-1}(2\epsilon) $. Thus, for $n\to\infty$ the second term can be made arbitrarily small by the choice of $\epsilon$.
 	
 	\noindent For the second term in \eqref{eq:nikolaus}, we obtain by a change of variables that
 	\begin{align*}
 	\int_{0}^{2\epsilon}\!\left|\HadPhi(h)\right|(s)\,ds=\int_{0}^{2\epsilon}\!\frac{h\left( F^{-1}(u)\right) }{f\left( F^{-1}(u)\right)}\,du=\int_{a}^{F^{-1}(2\epsilon)}\!h (u) \,du.
 	\end{align*}
 	Thus, this term will be arbitrarily small for $\epsilon$ small.
 	
 	\noindent\textit{Third summand:} The third summand can be treated with the same arguments as the second.
 	%\\\noindent Now, that we have handled all summands, the claim follows as already argued.
 \end{proof}
With the above lemma established it is straight forward to prove that
\begin{equation*}\mathbb{U}_n^{-1}\rightsquigarrow\mathbb{G}_1 \text{ in } \ell^1(0,1) \text{ and }
 \mathbb{V}_m^{-1}\rightsquigarrow\mathbb{G}_2 \text{ in } \ell^1(0,1).\end{equation*}
where $\mathbb{G}_1$ and $\mathbb{G}_2$ are centered, independent, continuous Gaussian processes with covariance structures
\begin{align*}
&\Cov\left(\mathbb{G}_1(t),\mathbb{G}_1(t') \right)= \frac{4}{(u\circ U^{-1}(t))(u\circ U^{-1}(t'))}\Gamma_{\dX}(U^{-1}(t),U^{-1}(t')) 
\end{align*}
and
\begin{align*}
&\Cov\left(\mathbb{G}_2(t),\mathbb{G}_2(t') \right)= \frac{4}{(v\circ V^{-1}(t))(v\circ V^{-1}(t'))}\Gamma_{\dY}(V^{-1}(t),V^{-1}(t')),
\end{align*}
where $u=U'$ and $v=V'$. This allows us to essentially repeat the arguments of the proof of \Cref{thm:main} $(i)$. For the missing details we refer to \cite[Sec. B.4]{supplement}.
\section*{Acknowledgements}
We are grateful to F. M\'emoli, V. Liebscher and M. Wendler for interesting discussions and helpful comments and to C. Br\'echeteau for providing $\mathsf{R}$ code for the application of the DTM-test. Support by the DFG Research Training Group 2088 Project A1 and Cluster of Excellence MBExC 2067 is gratefully acknowledged.
\bibliographystyle{plainnat.bst}
\bibliography{Dod}
\end{document}